\documentclass{amsart}

\usepackage{amssymb,mathrsfs,amscd,graphicx, array, mathtools}
\allowdisplaybreaks
\usepackage[nocompress, noadjust]{cite}  

\usepackage{enumitem}

\makeatletter
\newcommand*{\Relbarfill@}{\arrowfill@\Relbar\Relbar\Relbar}
\newcommand*{\xeq}[2][]{\ext@arrow 0055\Relbarfill@{#1}{#2}}
\makeatother

\usepackage{tikz}
\usetikzlibrary{matrix,arrows,cd}

\usepackage[
  pdfencoding=auto, 
  psdextra,
]{hyperref}
\usepackage{letltxmacro}
\LetLtxMacro{\oldsqrt}{\sqrt}
\renewcommand{\sqrt}[2][]{\,\oldsqrt[#1]{#2}\,}

\makeatletter
\def\@tocline#1#2#3#4#5#6#7{\relax
  \ifnum #1>\c@tocdepth 
  \else
    \par \addpenalty\@secpenalty\addvspace{#2}%
    \begingroup \hyphenpenalty\@M
    \@ifempty{#4}{%
      \@tempdima\csname r@tocindent\number#1\endcsname\relax
    }{%
      \@tempdima#4\relax
    }%
    \parindent\z@ \leftskip#3\relax \advance\leftskip\@tempdima\relax
    \rightskip\@pnumwidth plus4em \parfillskip-\@pnumwidth
    #5\leavevmode\hskip-\@tempdima
      \ifcase #1
       \or\or \hskip 1em \or \hskip 2em \else \hskip 3em \fi%
      #6\nobreak\relax
    \dotfill\hbox to\@pnumwidth{\@tocpagenum{#7}}\par
    \nobreak
    \endgroup
  \fi}
\makeatother

\usepackage{bbm}
\newcommand{\bsh}{\backslash}

\makeatletter

\def\greekbolds#1{%
 \@for\next:=#1\do{%
    \def\X##1;{%
     \expandafter\def\csname V##1\endcsname{\boldsymbol{\csname##1\endcsname}}
     }
   \expandafter\X\next;
  }
}

\greekbolds{alpha,beta,iota,gamma,lambda,nu,eta,Gamma,varsigma}

\def\make@bb#1{\expandafter\def
  \csname bb#1\endcsname{{\mathbb{#1}}}\ignorespaces}

\def\make@bbm#1{\expandafter\def
  \csname bb#1\endcsname{{\mathbbm{#1}}}\ignorespaces}

\def\make@bf#1{\expandafter\def\csname bf#1\endcsname{{\bf
      #1}}\ignorespaces} 

\def\make@gr#1{\expandafter\def
  \csname gr#1\endcsname{{\mathfrak{#1}}}\ignorespaces}

\def\make@scr#1{\expandafter\def
  \csname scr#1\endcsname{{\mathscr{#1}}}\ignorespaces}

\def\make@cal#1{\expandafter\def\csname cal#1\endcsname{{\mathcal
      #1}}\ignorespaces} 

\def\do@Letters#1{#1A #1B #1C #1D #1E #1F #1G #1H #1I #1J #1K #1L #1M
                 #1N #1O #1P #1Q #1R #1S #1T #1U #1V #1W #1X #1Y #1Z}
\def\do@letters#1{#1a #1b #1c #1d #1e #1f #1g #1h #1i #1j #1k #1l #1m
                 #1n #1o #1p #1q #1r #1s #1t #1u #1v #1w #1x #1y #1z}
\do@Letters\make@bb   \do@letters\make@bbm
\do@Letters\make@cal  
\do@Letters\make@scr 
\do@Letters\make@bf \do@letters\make@bf   
\do@Letters\make@gr   \do@letters\make@gr
\makeatother

\newcommand{\sel}{\mathrm{sel}}
\newcommand{\non}{\mathrm{non}}

\newcommand{\abs}[1]{\lvert #1 \rvert}
\newcommand{\zmod}[1]{\mathbb{Z}/ #1 \mathbb{Z}}

\newcommand{\wh}{\widehat}

\newcommand{\ul}{\underline}

\newcommand{\ddiv}{\mid}

\newcommand{\scc}{\mathrm{sc}}
\DeclareMathOperator{\Mass}{Mass}

\def\ad{\mathrm{ad}}

\def\op{\mathrm{op}}
\def\sg{\mathrm{sg}}
\def\wsg{\mathrm{wsg}}
\def\Ogp{\mathrm{O}}
\def\wide{\mathrm{wd}}
\def\res{\mathrm{res}}

\DeclareMathSymbol{\twoheadrightarrow} {\mathrel}{AMSa}{"10}

\DeclareMathOperator{\ord}{ord}

\DeclareMathOperator{\Clf}{Clf}

\DeclareMathOperator{\Pic}{Pic}

\DeclareMathOperator{\Hom}{Hom}

\DeclareMathOperator{\Gal}{Gal}
\DeclareMathOperator{\Mat}{Mat}

\DeclareMathOperator{\Tr}{Tr}
\DeclareMathOperator{\Nm}{N}  

\DeclareMathOperator{\Idl}{Idl}
\DeclareMathOperator{\RIdl}{RIdl}
\DeclareMathOperator{\Pol}{Pol}

\DeclareMathOperator{\SO}{SO}

\DeclareMathOperator{\Nr}{Nr}
\DeclareMathOperator{\Cl}{Cl}
\DeclareMathOperator{\Tp}{Tp}
\DeclareMathOperator{\SG}{SG}
\DeclareMathOperator{\SCl}{SCl}
\DeclareMathOperator{\Emb}{Emb}
\DeclareMathOperator{\Frac}{Frac}
\DeclareMathOperator{\WSG}{WSG}
\DeclareMathOperator{\Picent}{Picent}
\DeclareMathOperator{\Stab}{Stab}
\DeclareMathOperator{\REmb}{REmb}

\newcommand{\calRE}{\mathcal{RE}}
\newcommand{\ru}{\mathrm{ru}}
\newcommand{\cls}{{\mathrm{cl}}}
\newcommand{\tp}{{\mathrm{tp}}}

\newcommand{\Z}{\mathbb Z}
\newcommand{\Q}{\mathbb Q}
\newcommand{\R}{\mathbb R}
\newcommand{\F}{\mathbb F}

\newcommand{\wcO}{\widehat{\mathcal{O}}}
\newcommand{\whD}{\widehat{D}}
\newcommand{\whF}{\widehat{F}}
\newcommand{\whK}{\widehat{K}}
\newcommand{\whB}{\widehat{B}}
\newcommand{\whO}{\widehat{O}}

\newcommand{\wbZ}{\widehat{\mathbb{Z}}}
\newcommand{\whI}{\widehat{I}}

\newcommand{\whL}{\widehat{L}}
\newcommand{\wcE}{\widehat{\mathcal{E}}}
\newcommand{\wcRE}{\widehat{\mathcal{RE}}}

\newcounter{thmcounter} 
\numberwithin{thmcounter}{section}  
\newtheorem{thm}[thmcounter]{Theorem}
\newtheorem{lem}[thmcounter]{Lemma}
\newtheorem{cor}[thmcounter]{Corollary}
\newtheorem{prop}[thmcounter]{Proposition}
\theoremstyle{definition}
\newtheorem{defn}[thmcounter]{Definition}
\newtheorem{ex}[thmcounter]{Example}

\newtheorem{rem}[thmcounter]{Remark}

\newtheorem{sect}[thmcounter]{}

\newtheorem{assm}[thmcounter]{Assumptions}

\numberwithin{equation}{section}
\numberwithin{figure}{section}
\numberwithin{table}{section}

\newtheoremstyle{notitle}  
  {}
  {}
  {\itshape}
  {}
  {}
  {\ }
  {.5em}
  {}
\theoremstyle{notitle}

\title[Spinor type number]{The spinor type number formula for totally definite quaternion orders}
\author{Yucui Lin}

\address{(Lin) School of Mathematics and Statistics, Wuhan University, Luojiashan, 430072, Wuhan, Hubei, P.R. China}   

\email{yucui.lin@whu.edu.cn}

\author{Jiangwei Xue}

\address{(Xue) School of Mathematics and Statistics, Wuhan University, Luojiashan, 430072, Wuhan, Hubei, P.R. China}

\email{xue\_j@whu.edu.cn}

\begin{document}
\date{\today} 

\begin{abstract}
  Let $D$ be a totally definite quaternion algebra over a totally real number field $F$, and $\calO$ be an $O_F$-order (of full rank) in $D$. The type number $t(\calO)$ is an important arithmetic invariant of $\calO$ that counts the number of isomorphism classes of orders belonging to the same genus as $\calO$ (i.e.~locally isomorphic to $\calO$ at every finite place $\grp$ of $F$). The type number formula has been studied by Eichler, Peters, Pizer, Vigneras, K\"orner and many others. As the genus of $\calO$ further divides into spinor genera, one naturally seeks  a finer type number formula for the number of isomorphism classes of orders belonging to the same spinor genus of $\calO$. The main goal of this paper is to provide such a refinement for a large class of quaternion $O_F$-orders $\calO$ that includes all Eichler orders. This enables us to prove that $t(\calO)$ is divisible by the order of a quotient group $\WSG(\calO)$ of the Gauss genus group $\Cl^+(O_F)/\Cl^+(O_F)^2$ naturally attached to $\calO$. Similarly, we show that the trace of the $\grn$-Brandt matrix $\grB(\calO, \grn)$ is divisible by the class number $h(F)$ for any nonzero integral $O_F$-ideal $\grn$. In particular, the class number $h(\calO)=\Tr(\grB(\calO, O_F))$  is always divisible by $h(F)$ for such quaternion orders. This generalizes the divisibility result of $h(\calO)$ proved in a different way by Chia-Fu Yu and the second named author [Indiana Univ.~Math.~J., Vol.~70, No.~2 (2021)] in the case when $\calO$ is a maximal $O_F$-order in a totally definite quaternion algebra unramified at all the finite places. 
\end{abstract}

\maketitle  



\section{Introduction}

Let $F$ be a totally real number field. A quaternion $F$-algebra $D$ is called \emph{totally definite} if $D\otimes_{F,\sigma}\bbR$ is isomorphic to the Hamilton quaternion $\R$-algebra $\bbH$ for every embedding $\sigma: F \hookrightarrow \bbR$. Let $O_F$ be the ring of integers of $F$, and $\calO$ be an $O_F$-order (of full rank) in $D$. 
For each finite place $\grp$ of $F$, we write  $\calO_\grp$ for the $\grp$-adic completion of $\calO$. The profinite completion of $\calO$ is denoted by $\wcO$  and  the ring of finite adeles of $D$ is denoted by $\whD$. Two $O_F$-orders $\calO,\calO'$ in $D$ are said to \emph{belong to the same genus}  if there exists $x\in\whD^\times$ such that $\wcO'=x \wcO x^{-1}$, i.e., if $\calO_\grp$ and $\calO'_\grp$ are isomorphic for every finite place $\grp$ of $F$. This defines an equivalence relation on the set of all orders in $D$, and an equivalence class $\scrG$ is called a \emph{genus} (of orders) in $D$. The genus represented by $\calO$ is denoted by $\scrG(\calO)$. Similarly, two $O_F$-orders $\calO,\calO'$ in $D$ are said to be \emph{of the same type} if there exists $x\in D^\times$ such that $\calO'=x \calO x^{-1}$, i.e., they are isomorphic as $O_F$-orders. The isomorphism class of $\calO$ will be  called  the \emph{type} of $\calO$ and denoted by $[\calO]$. Let $\Tp(\scrG)$ be the set of all  types in $\scrG$. When $\scrG=\scrG(\calO)$,  we write $\Tp(\calO)\coloneqq\Tp(\scrG(\calO))$ and regard it as a pointed set with base point $[\calO]$. This is indeed a finite set \cite[\S17.1]{voight-quat-book} and its cardinality is called the \emph{type number} of $\calO$  and  denoted by $t(\calO) \coloneqq \abs{\Tp(\calO)}$.

Historically,  Eichler first gave a type number formula for Eichler orders of square-free levels in his pioneering work \cite{eichler:crelle55}. Some minor errors in Eichler's formula were corrected independently by Peters \cite{Peters1968} over fields of class number one and by Pizer \cite{Pizer1973} for general totally real number fields.  Pizer \cite{Pizer-1976}  also derived a type number formula for (general) Eichler orders over $\Q$, 
 and Vign\'eras \cite[Crollaire~V.2.6]{vigneras} obtained a type number formula for Eichler orders in an arbitrary $D$. 
Their works were further generalized by K\"orner \cite{korner:1987} to a type number formula that applies to general quaternion $O_F$-orders in $D$.   Recently, Boylan, Skoruppa and  Zhou \cite{Boylan-Skoruppa-Zhou-2019} give a new proof of Eichler's type number formula for maximal orders of definite quaternion $\Q$-algebras using the theory of Jacobi forms and reformulate the result in terms of Hurwitz class numbers, and further generalizations to definite quaternion Eichler orders of square-free levels over $\Q$ are made by Li, Skoruppa and Zhou \cite{Li-Skoruppa-Zhou-2022}.  On a related note, inspired by the Gauss number problem,  there is an interest in enumerating definite quaternion orders of small type numbers.  Jagy, Kaplansky, and Schiemann \cite{Jagy-Kaplansky-Schiemann} first draw up a list of Gorenstein quaternion $\Z$-orders of type number 1 (with the result formulated in terms of ternary quadratic forms, cf.~\cite[Theorems~22.1.1 and 25.4.6]{voight-quat-book}). Independently,  Kirschmer and Lorch \cite{Lorch-Kirschmer-2013, Kirschmer-Lorch-2016} confirm the result of Jagy et~al. and produce a list of all definite quaternion
orders of type number at most 2 over number fields. 

As we shall see very soon, inside a totally definite quaternion $F$-algebra $D$,  each genus $\scrG$ is   a disjoint union of several spinor genera.  The main goal of this paper is to derive a spinor type number formula that counts the number of types of orders belonging to the same spinor genus of $\calO$.  Such a spinor type number formula clearly refines the classical type number formula mentioned above. Our method relies heavily on K\"orner's formulation, and the primary new input is the \emph{optimal spinor selectivity theory} that is initiated by Chinburg and Friedman \cite{Chinburg-Friedman-1999} and studied by many others.  However, due to the limitation of the selectivity theory (\S\ref{sec:spinor-trace-formula} and \cite[\S5]{peng-xue:select}), we restrict ourselves  to \emph{residually unramified orders} (Definition~\ref{defn:res-unr-order}), which nevertheless is broad enough to include all Eichler orders by \cite[Lemma~24.3.6]{voight-quat-book}.  In a special case,  spinor type number formulas have already been obtained in \cite[Propositions~6.2 and 6.2.1]{xue-yu:ppas} for all maximal orders in the totally definite quaternion algebra $D_{\infty_1, \infty_2}$ (of discriminant $O_F$) over the quadratic real field $F=\Q(\sqrt{p})$,  and they carry special geometric meanings \cite[Theorem~5.9]{xue-yu:ppas} in terms of superspecial abelian surfaces over $\F_p$ with real Weil numbers $\pm\sqrt{p}$.

We recall the notion of \emph{spinor genus}  from \cite[\S1]{Brzezinski-Spinor-Class-gp-1983}, which refines that of the \emph{genus}. Let $\Nr: \whD^\times \to \whF^\times$ be the reduced norm map and put $\whD^1\coloneqq\ker(\Nr: \whD^\times \to \whF^\times)$. 
\begin{defn}\label{defn:spinor-genus}
    (1) Two orders $\calO, \calO'$  in $D$ are said to be in the \emph{same spinor genus} if there exists $x\in D^\times\whD^1$ such that $\wcO'=x \wcO x^{-1}$. The spinor genus of $\calO$ is denoted by $[\calO]_\sg$.
    
    (2) The set of all spinor genera within $\scrG$ is denoted by $\SG(\scrG)$. We write $\SG(\calO)$ for $\SG(\scrG(\calO))$ and regard it as a pointed set with base point $[\calO]_\sg$.

    (3) Let $\Tp_\sg(\calO)\coloneqq\{[\calO'] \mid [\calO']\subseteq[\calO]_\sg\}$ be the subset of $\Tp(\calO)$ consisting of all types within the spinor genus $[\calO]_\sg$. Its cardinality will be called the  \emph{spinor type number of $\calO$} and denoted by $t_\sg(\calO) \coloneqq \abs{\Tp_\sg(\calO)}$.
\end{defn}

\begin{rem}\label{rem:EC}
By definition, the map sending each $[\calO']$ to $[\calO']_\sg$ induces a  canonical surjection $\Psi_\tp: \Tp(\calO) \to \SG(\calO)$ between the two pointed sets, whose neutral fiber is precisely $\Tp_\sg(\calO)$.  Clearly, everything in Definition~\ref{defn:spinor-genus} remains valid if we drop the ``totally definite'' assumption on $D$ and instead assume that $D$ satisfies the Eichler condition (EC), that is, $F$ is allowed to be an arbitrary number field and $D$ is required to split at an infinite place of $F$. In this case, Brzezinski \cite[Proposition 1.1]{Brzezinski-Spinor-Class-gp-1983} shows that the map $\Psi: \Tp(\calO) \to \SG(\calO)$ is a bijection, that is, each spinor genus consists of exactly one type. This explains why we focus on the totally definite case instead. 
\end{rem}

Before diving straightly into presenting the spinor type number formula (which is expectedly quite involved),  we explain an interesting implication of it.  
The type set $\Tp(\calO)$ naturally decomposes into a disjoint union of the fibers of the map $\Psi_\tp: \Tp(\calO)\twoheadrightarrow \SG(\calO) $ as follows: 
\begin{equation}\label{eq:typ-refine}
    \Tp(\calO)=\bigsqcup_{[\calO']_\sg\in \SG(\calO)} \Tp_\sg(\calO'), \quad \text{and hence}\quad t(\calO)=\sum_{[\calO']_\sg\in \SG(\calO)} t_\sg(\calO'). 
\end{equation}
The spinor type number formula enables us to compute
$t_\sg(\calO')$ for each $[\calO']_\sg\in \SG(\calO)$.  It turns out that if $D$ is ramified at some finite place of $F$, then $t_\sg(\calO')$ stays invariant as $[\calO']_\sg$ ranges in $\SG(\calO)$, which immediately implies that the type number $t(\calO)=\abs{\Tp(\calO)}$ is divisible by $\abs{\SG(\calO)}$.  On the other hand, such a  divisibility result is too good to be true if $D$ is unramified at all the finite places of $F$; see Example~\ref{ex:indivisible} for a counterexample. Nevertheless, the spinor type number formula still yields a slightly weaker divisibility result as follows. From \eqref{eq:118},  the pointed set $\SG(\calO)$ is naturally equipped with an abelian group structure with $[\calO]_\sg$ as its identity element, which is called the \emph{spinor genus group of $\calO$}.  It enjoys a quotient group $\WSG(\calO)$ called the \emph{wide spinor genus group}, much in the same way that the narrow class group $\Cl^+(O_F)$ enjoys the (wide) class group $\Cl(O_F)$ as a quotient group (see \eqref{eqn:df-WSG}).
We shall show in Theorem~\ref{thm:div-typ-5.3} that the type number $t(\calO)$ is always divisible by $\abs{\WSG(\calO)}$.  In summary, we have the following divisibility result for the type number $t(\calO)$.

\begin{thm}[{Theorem~\ref{thm:div-typ-5.3}}]\label{thm:type-num-div}
    Let $\calO$ be a residually unramified $O_F$-order in a totally definite quaternion $F$-algebra $D$. Then $t(\calO)$ is divisible by $\abs{\WSG(\calO)}$. If $D$ is further assumed to be ramified at some finite place of $F$, then $t(\calO)$ is  divisible by $\abs{\SG(\calO)}$.
\end{thm}

The above result will be re-interpreted as the divisibility of the class number of certain totally positive ternary quadratic $O_F$-lattices in Theorem~\ref{thm:ternary-latt}.
Using the same method, we can also obtain  divisibility results for the class number of $\calO$, and more generally for the trace of the $\grn$-Brandt matrix $\grB(\calO, \grn)$ for any nonzero integral $O_F$-ideal $\grn$.  By definition, the class number $h\coloneqq h(\calO)$ is the cardinality of the finite set $\Cl(\calO)$ of locally principal fractional right $\calO$-ideal classes in $D$. It depends only on the genus of $\calO$. 
Given a nonzero integral $O_F$-ideal $\grn$, the $\grn$-Brandt matrix $\grB(\calO, \grn)\in \Mat_h(\Z)$ is an $h\times h$-matrix  with non-negative integer entries constructed from the ideal class set $\Cl(\calO)$ as in \cite[Definition~41.2.4]{voight-quat-book} or \cite[Exercise~5.8]{vigneras}, which will be recalled in Definition~\ref{defn:brandt-matrix}. From \cite[\S41.2.5]{voight-quat-book} or \cite[\S4]{xue-yang-yu:ECNF}, $\grB(\calO, \grn)$ can be interpreted as the matrix of the $\grn$-Hecke operator on a space of algebraic modular forms attached to $\calO$, and its trace can be computed by the renowned Eichler trace formula \cite[Proposition~V.2.4]{vigneras}\cite[Theorem~41.5.2]{voight-quat-book}. Moreover, if $\grn=O_F$, then $\Tr(\grB(\calO, O_F))$ coincides with the class number $h(\calO)$.

\begin{thm}[{Theorem~\ref{thm:div-tr-Brandt}}]\label{thm:class-num-div}
  Let $\calO$ be a residually unramified $O_F$-order in a totally definite quaternion $F$-algebra $D$. Then for every nonzero integral ideal $\grn\subseteq O_F$, $\Tr(\grB(\calO,\grn))$ is divisible by the class number $h(F)$. If $D$ is further assumed to be ramified at some finite place of $F$, then $\Tr(\grB(\calO,\grn)))$ is divisible by the narrow class number $h^+(F)$.  In particular, the above divisibility results hold for the class number $h(\calO)$ as $h(\calO)=\Tr(\grB(\calO, O_F))$.
\end{thm}

It is remarkable how closely the divisibility results of the class number $h(\calO)$ in Theorem~\ref{thm:class-num-div} resembles that of CM-extensions \cite[Theorem~4.10]{Washington-cyclotomic}, that is, for every CM-extension $K/F$, the class number $h(K)$ is always divisible by $h(F)$, and furthermore, if $K/F$ is ramified at some finite place of $F$, then $h(K)$ is divisible by $h^+(F)$. Actually, the proofs of Theorems~\ref{thm:type-num-div}--\ref{thm:class-num-div} and that of \cite[Theorem~4.10]{Washington-cyclotomic} all rely crucially on the linear disjointness between the Hilbert class field of the totally real field $F$ and any CM-extension of $F$.   On the other hand, Theorem~\ref{thm:class-num-div} can be viewed as an improvement of the divisibility result of Vign\'eras \cite[Corollaire~1.1]{vigneras:ens}, who shows that  $2h(\calO)$ is always divisible by $h(F)$ for any Eichler order $\calO$ by studying the relationship between the class number $h(\calO)$ and the type number $t(\calO)$. The divisibility result of Vign\'eras has since been generalized to all quaternion $O_F$-orders by Chia-Fu Yu and the second named author in \cite[Corollary~3.10]{xue-yu:type_no}, and a partial case of Theorem~\ref{thm:class-num-div} was also treated in \cite[Theorem~5.4]{xue-yu:type_no} by a further elaboration of Vign\'eras's method.

Unsurprisingly, Vign\'eras's class-type number relationship is also the starting point of our current endeavor for the spinor type number formula.
Our formula is also closely modeled after K\"orner's type number formula \cite{korner:1987}, although we will not follow the same route of derivation via the Selberg trace formula.  To  explain our formula, it is better to recall  K\"orner's type number formula, but to recast its derivation in our way and then explain the extra input needed to get the spinor type number formula.

For each locally principal fractional right $\calO$-ideal $I\subset D$, we write $[I]$ for its ideal class in $\Cl(\calO)$, and $\calO_l(I)\coloneqq\{x\in D \mid xI\subseteq I\}$ for its left order in $D$.   
There is a canonical surjective map 
\begin{equation}\label{eq:Cl-to-Tp}
   \xi:  \Cl(\calO) \to \Tp(\calO), \qquad [I] \mapsto [\calO_l(I)],  
\end{equation}
which can also be interpreted as the quotient map of the group action of $\Picent(\calO)$ on $\Cl(\calO)$ as follows.  By definition \cite[Definition~55.6]{curtis-reiner:2},  the \emph{central Picard group} $\Picent(\calO)$ is the quotient of the group of invertible (equivalently, locally principal by \cite{Kaplansky-quat-invertible}) two-sided fractional $\calO$-ideals in $D$ modulo the subgroup of principal ideals generated by elements of $F^\times$, and it is a finite group by \cite[Proposition 18.4.10]{voight-quat-book}.  Given an invertible two-sided fractional $\calO$-ideal $P\subset D$,  we still write $[P]$ for its class in $\Picent(\calO)$ by an abuse of notation.   The group $\Picent(\calO)$ acts  on $\Cl(\calO)$ by right multiplication: 
\begin{equation}\label{eq:Picent-action-defn}
    \mu:    \Cl(\calO) \times \Picent(\calO) \to \Cl(\calO), \qquad ([I],[P]) \mapsto [IP],
\end{equation}
and two locally principal right $\calO$-ideals  $I, I'$ share isomorphic left orders if and only if their classes in $\Cl(\calO)$ belong to the same orbit.   Thus there is a canonical bijection
\begin{equation}
    \Tp(\calO)\simeq \Cl(\calO)/\Picent(\calO),
\end{equation}
and the type number $t(\calO)$ can be computed via Burnside's lemma \cite[Theorem 2.113]{Rotman-alg}. 
Indeed, Burnside's lemma states that if $G$ is a finite group acting on a finite set $X$, then the number  of orbits $\abs{G\backslash X}$ is given by $\frac{1}{\abs{G}}\sum_{g\in G}\abs{X^g}$, where $X^g$ denotes the subset of fixed points of $g\in G$.  It turns out that a similar strategy can be used to compute the spinor type number $t_\sg(\calO)$.  This strategy not only helps us circumvent the more technical Selberg trace formula employed by Pizer \cite{Pizer1973} and K\"orner \cite{korner:1987} in their efforts for the type number formula, but is also more adaptable to the spinor type number situation.   The details will be carried out in Section~\ref{sec:Picent-for-tot-def-quat-alg}.

Certainly Burnside's lemma is not the only tool needed for the (spinor) type number formula. 
Classically in the proof of the Eichler trace formula \cite[Theorem~41.5.2]{voight-quat-book}, a key ingredient is the \emph{trace formula for optimal embeddings} \cite[Theorem~30.4.7]{voight-quat-book}, which 
works for any  quaternion algebra without imposing the totally definite requirement.  Nevertheless, we keep the totally definite assumption since it is most relevant here. A quadratic extension of $F$ embeds into $D$ only if it is a CM-extension, and an $O_F$-order $B$ of full rank in a CM-extension  $K/F$ will be called a \emph{CM $O_F$-order}.
  We write $\Emb(B, \calO)$ for the set of optimal embeddings of $B$ into $\calO$, that is, 
\begin{equation}\label{eq:def-opt-emb}
   \Emb(B, \calO)\coloneqq \{\varphi\in \Hom_F(K, D)\mid \varphi(K)\cap \calO= \varphi(B)\}.  
\end{equation}
The unit group $\calO^\times$ acts on $\Emb(B,  \calO)$ from the right by sending $\varphi\mapsto u^{-1} \varphi u$ for each  $\varphi\in \Emb(B, \calO)$ and $u\in \calO^\times$, and the number of orbits is denoted by 
\begin{equation}\label{eq:no-conjcls-opt}
    m(B, \calO, \calO^\times)\coloneqq \abs{\Emb(B,  \calO)/\calO^\times}, 
\end{equation}
which is known to be finite. A similar finiteness result holds for the number $m(B_\grp, \calO_\grp, \calO_\grp^\times)$ of $\calO_\grp^\times$-conjugacy classes of optimal embeddings from $B_\grp$ into $\calO_\grp$ at each finite place $\grp$ of $F$, where $B_\grp$ (resp.~$\calO_\grp$) denotes the $\grp$-adic completion of $B$ (resp.~$\calO$). The trace formula relates the numbers of local and global  conjugacy classes of optimal embeddings as follows:
\begin{equation}\label{eq:classic-trf-opt-intro}
    \sum_{[I]\in \Cl(\calO)} m(B, \calO_l(I), \calO_l(I)^\times)=h(B)\prod_\grp m(B_\grp, \calO_\grp, \calO_\grp^\times), 
\end{equation}
where $h(B)\coloneqq \abs{\Pic(B)}$ denotes the class number of  $B$, 
and the product runs over all finite places of $F$.  The above trace formula is also used implicitly in Pizer's type number formula \cite[Theorem~A]{Pizer1973} for Eichler orders of square-free levels.

However, optimal embeddings alone are inadequate for the (spinor) type number formula for more general quaternion orders.  We compare the situation with that of the Eichler trace formula case or Pizer's type number formula case  to understand why. 
Implicitly, each CM $O_F$-order $B$ appearing in the Eichler trace formula carries a distinguished point $\lambda\in B\smallsetminus O_F$.   Indeed,  in \cite[Theorem~41.5.2]{voight-quat-book}, $\lambda$ is given by the image of $x$ in the quotient ring $R[x]/(x^2-tx+un)$ as denoted there, which has a prescribed norm (cf.~\eqref{eq:def-TBn}). Similar points are also present in the relevant CM $O_F$-orders in Pizer's type number formula \cite[Theorem~A]{Pizer1973}. 
For these two formulas, optimal embeddings \emph{automatically} send such points to the desired elements of $\calO$. In the setting of Eichler trace formula, an optimal embedding preserves the norm of the distinguished point;  and in Pizer's setting which focuses on Eichler orders of square-free levels, every optimal embedding $\varphi\in \Emb(B, \calO)$
 sends the distinguished point $\lambda$ into the normalizer group $\calN(\calO)\subset D^\times$ of $\calO$ by \cite[Lemma~5]{Pizer1973}.  
Comparably, if the square-free condition is dropped, the latter statement would not be valid in general, which causes difficulty in the case of more general quaternion orders. 
The insight of K\"orner to overcome this difficulty is to replace optimal embeddings by restricted optimal embeddings, which is indispensable for his success in extending the type number formula from Eichler orders of square-free levels to arbitrary  quaternion $O_F$-orders. 
By definition, a \emph{restricted optimal embedding} from a pointed CM $O_F$-order $(B, \lambda)$ to $\calO$ is an optimal embedding $\varphi\in \Emb(B, \calO)$ that sends $\lambda$ into $\calN(\calO)$. We denote the set of restricted optimal embeddings by $\REmb((B, \lambda), \calO)$, that is, 
\begin{equation}\label{eq:global-REmb-intro}
    \REmb((B, \lambda), \calO)\coloneqq \{\varphi\in \Emb(B, \calO)\mid \varphi(\lambda)\in \calN(\calO)\},
\end{equation}
which is again $\calO^\times$-conjugate stable. 
Mimicking  \eqref{eq:no-conjcls-opt}, we put
\begin{equation}
    n((B, \lambda), \calO, \calO^\times)\coloneqq \abs{\REmb((B, \lambda), \calO)/\calO^\times}. 
\end{equation}
It turns out that a similar trace formula as \eqref{eq:classic-trf-opt-intro} holds for restricted optimal embeddings (see Proposition~\ref{prop:trf-res-opt-emb}): 
     \begin{equation}\label{eq:110}
        \sum_{[I]\in\Cl(\calO)} n((B,\lambda),\calO_l(I),\calO_l(I)^\times)=h(B)\prod_\grp n((B_\grp,\lambda),\calO_\grp,\calO_\grp^\times). 
    \end{equation}
This formula was used implicitly by K\"orner in \cite{korner:1987}, and  we recall his type number formula below.  For simplicity, we put $n_\grp(B, \lambda)\coloneqq n((B_\grp,\lambda),\calO_\grp,\calO_\grp^\times)$ since it  depends only on the genus $\scrG(\calO)$, which is always clear from the context. 

\begin{thm}[{\cite[Theorem 3]{korner:1987}}] 
  The type number $t(\calO)$ of any  $O_F$-order $\calO$ in a totally definite quaternion $F$-algebra $D$  can be computed as follows:
    \begin{equation}\label{eq:Korner-formula-intro}
        t(\calO)=\frac{1}{\abs{\Picent(\calO)}} \Big( \Mass(\calO)+\frac{1}{2} \sum_{(B,\lambda)\in\scrB} \frac{h(B)}{w(B)}\prod_{\grp\mid \grd(\calO)} n_\grp(B,\lambda)\Big), 
    \end{equation}
 where 
 \begin{itemize}
     \item  $\Mass(\calO)$ denotes the mass of $\calO$ as defined in \eqref{eq:mass-def} whose valued can be computed by the Eichler mass formula \eqref{eq:234},
     \item $\scrB$ denotes the finite set of pointed CM $O_F$-orders $(B, \lambda)$ constructed in Proposition~\ref{prop:def-B},
     \item  $h(B)$ (resp.~$w(B)$) denotes the class number (resp.~the unit index $[B^\times: O_F^\times]$) of the  CM $O_F$-order $B$, and the value of  
     $h(B)$ can be calculated by the Dedekind’s formula in \cite[\S III.5, p.~95]{vigneras} or \cite[p.~74]{vigneras:ens},
     \item $\grd(\calO)$ denotes the reduced discriminant of $\calO$, and we have $n_\grp(B, \lambda)=1$ for all $(B, \lambda)\in \scrB$ if $\grp\nmid \grd(\calO)$ by Lemma~\ref{lem:n1-pnmid-d}.
 \end{itemize}
\end{thm}

The spinor type number formula for $t_\sg(\calO)=\abs{\Tp_\sg(\calO)}$ that we aim to produce is a natural refinement of K\"orner's formula, as it counts the number of types within the spinor genus of $\calO$. 
Necessarily, this forces us to study restricted optimal embeddings from pointed CM $O_F$-orders into quaternion orders within a fixed spinor genus.
The \emph{selective} phenomenon naturally shows up, and we are led to the reformulation of the \emph{optimal spinor selectivity theory} \cite[\S31]{voight-quat-book} \cite{Xue-Yu-Selec-2024} in terms of restricted optimal embeddings rather than optimal embeddings. 
    Fix a pointed CM $O_F$-order $(B, \lambda)$,  and let $\scrG(\calO)$ (resp.~$[\calO]_\sg$) be the genus (resp.~spinor genus) of $\calO$  as in Definition~\ref{defn:spinor-genus}.  From \eqref{eq:110},  there exists $\calO_0\in \scrG(\calO)$ with $\REmb((B,\lambda), \calO_0)\neq \emptyset$ if and only if $\REmb((B_\grp,\lambda), \calO_\grp)\neq \emptyset$ for every finite place $\grp$ of $F$ (Compare with the similar statement for optimal embeddings in \cite[Corollary~30.4.18]{voight-quat-book}). However, the statement would not remain valid in general if the genus $\scrG(\calO)$ is replaced by the spinor genus $[\calO]_\sg$.  Given a spinor genus $[\calO']_\sg\in \SG(\calO)$, 
to encode  whether there exists a restricted optimal embedding from $(B, \lambda)$ into some order $\calO''\in [\calO']_\sg$, we introduce the symbol 
\begin{equation*}
    \Delta^{\res}((B,\lambda), \calO')\coloneqq
    \begin{cases}
        1 & \text{if}\ \exists\ \calO''\in [\calO']_\sg\ \text{such that}\ \REmb((B,\lambda), \calO'')\not=\emptyset,\\
        0 & \text{otherwise}.
    \end{cases}
\end{equation*}
One naturally asks the following questions.
\begin{enumerate}[label=(\roman*)]
    \item Does the value of $\Delta^{\res}((B,\lambda), \calO')$ vary  as $[\calO']_\sg$ ranges within $\SG(\calO)$?
    \item If the value of $\Delta^{\res}((B,\lambda), \calO')$ does vary as $[\calO']_\sg$ changes, how to compute the value for each specific  $[\calO']_\sg\in \SG(\calO)$?
\end{enumerate}
If the answer to Question~(i) is affirmative, then we say that $\scrG(\calO)$ is \emph{optimally spinor selective in the restricted sense} (selective for short) for the pointed order $(B, \lambda)$; see Definition~\ref{defn:oss}. Luckily, there is no need to redo every step of the \emph{optimal spinor selectivity theory} in \cite{M.Arenas-et.al-opt-embed-trees-JNT2018}\cite[\S31]{voight-quat-book} or \cite{Xue-Yu-Selec-2024} for restricted optimal embeddings.  Suppose that $\calO$ is a residually unramified order with $\REmb((B_\grp,\lambda), \calO_\grp)\neq \emptyset$ for every finite place $\grp$ of $F$.  It is shown in Proposition~\ref{prop:ru-equal-sel-symb} that $\Delta^{\res}((B,\lambda), \calO')=\Delta(B, \calO')$ for every $[\calO']_\sg\in \SG(\calO)$, where $\Delta(B, \calO')$ is a similar symbol  defined in terms of optimal embeddings with the distinguished point $\lambda$ ignored (see \eqref{eq:defn-Delta}). Thus the seemingly newer version of  selectivity theory essentially reduces to the old one studied in \cite{Xue-Yu-Selec-2024}. Keeping the residually unramified assumption on $\calO$, we put $s(B, \calO)\coloneqq 1$ if $\scrG(\calO)$ is selective for $(B, \lambda)$ and $0$ otherwise (cf.~\eqref{eq:defn-sBO} and Theorem~\ref{thm:spinor-selectivity}). In Proposition~\ref{prop:spinor-trf}, we develop a refinement of the trace formula in \eqref{eq:110} called \emph{the spinor trace formula for restricted optimal embeddings}:
   \begin{equation}\label{eq:112}
        \sum_{[I]\in\Cl_\scc(\calO)} n((B,\lambda),\calO_l(I),\calO_l(I)^\times)=\frac{2^{s(B,\calO)}\Delta(B,\calO)h(B)}{h^+(F)}\prod_\grp n_\grp(B,\lambda),
    \end{equation}
where $\Cl_\scc(\calO)$ is a subset of $\Cl(\calO)$ in Definition~\ref{defn:spinor-class} 
consisting of all ideal classes $[I]$ belonging to the same spinor class as the trivial class $[\calO]$.  Observe the similarity between \eqref{eq:112} and its optimal embedding counterpart in \cite[Proposition~4.3]{Xue-Yu-Selec-2024}. This formula plays the same role for our spinor type number formula as the ones played by the classical trace formula \eqref{eq:classic-trf-opt-intro}  for the Eichler trace formula and Pizer's type number formula.   This also explains why we need the residually unramified assumption for the spinor type number formula.

\begin{thm}[{Theorem~\ref{thm:spinor-type-number-formula}}]\label{thm:spinor-type-number-formula-intro}
    Let $\calO$ be a residually unramified $O_F$-order in a totally definite quaternion $F$-algebra $D$, and $\omega(\calO)$ be the number of distinct prime divisors of the reduced discriminant $\grd(\calO)$.  Put $C(\calO)\coloneqq 2^{\omega(\calO)}h(F)\abs{\SG(\calO)}$ for brevity. 
    Then
    \begin{equation}
        t_\sg(\calO)=\frac{1}{C(\calO)} \Big( \Mass(\calO)+\frac{1}{2} \sum_{(B,\lambda)\in\scrB} \frac{2^{s(B,\calO)}\Delta(B,\calO)h(B)}{w(B)}\prod_{\grp\mid \grd(\calO)} n_\grp(B,\lambda) \Big).
    \end{equation}
    If $D$ is further assumed to be ramified at some finite place of $F$, then
    \begin{equation}
        t_\sg(\calO)=\frac{1}{C(\calO)} \Big( \Mass(\calO)+\frac{1}{2} \sum_{(B,\lambda)\in\scrB} \frac{h(B)}{w(B)}\prod_{\grp\mid\grd(\calO)} n_\grp(B,\lambda) \Big).
    \end{equation}
\end{thm}

This paper is organized as follows. In Section~\ref{sec:prelim-on-Picent}, we recall the basic properties of $\Picent(\calO)$ and fix a complete set of representatives of it. We provide a simplified proof of K\"orner's type number formula in Section~\ref{sec:Picent-for-tot-def-quat-alg} and then modify the process suitably to obtain a preliminary version of the spinor type number formula.  Both the trace formula \eqref{eq:110} and the spinor trace formula \eqref{eq:112} for restricted optimal embeddings will be proved in Section~\ref{sec:spinor-trace-formula}, where we reduce the spinor selectivity theory for restricted optimal embeddings to  that of optimal embeddings.  We complete the derivation of the spinor type number formula in Section~\ref{sec:spinor-type-number-formula} and prove the aforementioned divisibility results as its applications.  In Section~\ref{sec:ternary-latt} we re-interpret the spinor type number as the class number of certain ternary quadratic lattices within a fixed spinor genus.

\section{Preliminaries on the central Picard group}\label{sec:prelim-on-Picent}

As explained in the introduction, our derivation of the (spinor) type number formula for a totally definite quaternion order $\calO$  depends  on the $\Picent(\calO)$-action on the ideal class set $\Cl(\calO)$. In this section, we recall some basic properties of the central Picard group $\Picent(\calO)$ and fix a complete set of locally principal  two-sided 
integral ideals representing $\Picent(\calO)$. There is no need to specialize to totally definite quaternion algebras yet, so throughout this section we assume that $D$ is an arbitrary quaternion algebra over a number field $F$, and $\calO$ is an $O_F$-order (of full rank) in $D$.  

 From \cite[Definition~55.6]{curtis-reiner:2},  the \emph{central Picard group} $\Picent(\calO)$ is the quotient of the group $\Idl(\calO)$ of invertible  two-sided fractional $\calO$-ideals in $D$ modulo the subgroup of principal ideals generated by elements of $F^\times$.  Similarly, we define the local central Picard group $\Picent(\calO_\grp)$ at each finite place $\grp$ of $F$, where $\calO_\grp$ denotes the $\grp$-adic completion of $\calO$. 
 From \cite[Theorem~2]{Kaplansky-quat-invertible}, any invertible two-sided fractional $\calO_\grp$-ideal is principal,   hence generated by an element of the normalizer group $\calN(\calO_\grp)\subset D_\grp^\times$ of $\calO_\grp$. Thus there is a canonical isomorphism 
 \begin{equation}\label{eqn:loc-desc-picent}
     \calN(\calO_\grp)/F_\grp^\times\calO_\grp^\times\xrightarrow{\simeq} \Picent(\calO_\grp),\qquad g_\grp F_\grp^\times\calO_\grp^\times \mapsto [g_\grp\calO_\grp].
 \end{equation}
 If $\grp$ is coprime to the reduced discriminant $\grd(\calO)$ of $\calO$, then $\calO_\grp\simeq \Mat_2(O_{F_\grp})$, and hence $\Picent(\calO_\grp)$ is trivial since $\calN(\calO_\grp)= F_\grp^\times\calO_\grp^\times$ in this case. From \cite[Theorem~6]{Frohlich},  we have a short exact sequence 
    \begin{equation}\label{eqn:picent-seq}
        1 \to \Cl(O_F) \to \Picent(\calO) \to \prod_{\grp \ddiv \grd(\calO)} \Picent(\calO_\grp) \to 1,
    \end{equation}
which reduces the computation of the order $|\Picent(\calO)|$ to that of the finitely many local ones computed by K\"orner \cite{korner:1985}. For example, if $\calO$ is a maximal order in $D$, then it is well known that\begin{equation}\label{eqn:card-picent}
  |\Picent(\calO)|=h(F)2^{\omega(\calO)}  
\end{equation}
  where $h(F)$ denotes the class number of $F$, and $\omega(\calO)$ denotes  the number of distinct prime divisors of  $\grd(\calO)$.  In fact, equality \eqref{eqn:card-picent} holds for a class of quaternion orders called \emph{residually unramified orders}, 
 which is a slight generalization of Eichler orders. To state the precise definition of such orders, we recall the concept of \emph{Eichler invariants} from \cite[Definition~1.8]{Brzezinski-1983} and \cite[Definition 24.3.2]{voight-quat-book}.

\begin{defn}\label{defn:res-unr-order}
(1) Fix a finite place $\grp$ of $F$.   Let  $\grk_\grp:= O_F/\grp$
be the finite residue field of $\grp$,  and $\grk_\grp'/\grk_\grp$ be the unique
quadratic field extension.  When $\calO_\grp\not\simeq
\Mat_2(O_{F_\grp})$, the quotient of $\calO_\grp$ by its Jacobson radical
$\grJ(\calO_\grp)$ falls into the following three cases: 
\[\calO_\grp/\grJ(\calO_\grp)\simeq \grk_\grp\times \grk_\grp, \qquad \grk_\grp,
\quad\text{or}\quad \grk_\grp', \]
and the \emph{Eichler invariant} $e_\grp(\calO)$ of $\calO$ at $\grp$ is defined to be
$1, 0, -1$ accordingly.  As a convention, if $\calO_\grp\simeq
\Mat_2(O_{F_\grp})$, then $e_\grp(\calO)$ is defined to be
$2$.

(2) We say that $\calO$ is \emph{residually unramified at $\grp$} if $e_\grp(\calO)\neq 0$. If $e_\grp(\calO)\neq 0$ for every finite place $\grp$ of $F$, then we simply say that $\calO$ is \emph{residually unramified}.
\end{defn}

For example, if $D$ is ramified at $\grp$ and $\calO_\grp$ is maximal, then
$e_\grp(\calO)=-1$. It is shown in
\cite[Proposition~2.1]{Brzezinski-1983} that 
$e_\grp(\calO)=1$ if and only if
$\calO_\grp$ is a non-maximal Eichler order (particularly,
$D$ is split at $\grp$).  As a result, if $\calO$ is an Eichler
order, then $e_\grp(\calO)\neq 0$ for every finite place $\grp$, which shows that all Eichler orders are residually unramified.

\begin{lem}\label{lem:res-unr-card-of-Picent}
    Suppose that $\calO$ is residually unramified. Then
    \begin{equation}
        \abs{\Picent(\calO)}=h(F)2^{\omega(\calO)}. 
    \end{equation}
\end{lem}
\begin{proof}
    In light of \eqref{eqn:loc-desc-picent}--\eqref{eqn:picent-seq}, it suffices to prove that $[\calN(\calO_\grp): F_\grp^\times \calO_\grp^\times]=2$ for every prime divisor $\grp$ of the reduced discriminant $\grd(\calO)$. This group index has been computed by K\"orner in \cite[Satz~2]{korner:1985}, and it can also be read off from the description of $\calN(\calO_\grp)$ by Brzezinski in \cite[Theorem~2.2]{Brzezinski-crelle-1990}.
\end{proof}

In \cite{korner:1987},  K\"orner developed a systematic way of fixing a complete set of invertible  two-sided integral $\calO$-ideals representing $\Picent(\calO)$.  As we shall see in Proposition~\ref{prop:choice-rep-picent}, the choices of representatives are based on the notion of primitive ideals. 

\begin{defn}\label{defn:primitive-ideal} 
Let $A$ be a (finite dimensional) simple $F$-algebra, and $\Lambda$ be an $O_F$-order (of full rank) in $A$.  Locally at a finite place $\grp$ of $F$,  a fractional right $\Lambda_\grp$-ideal $L_\grp\subset A_\grp$ is called a 
\emph{primitive right $\Lambda_\grp$-ideal} if it is integral, principal, and not contained in $\grp\Lambda_\grp$.  Globally, a  fractional right $\Lambda$-ideal $L\subset A$ is called a \emph{primitive right $\Lambda$-ideal} if $L_\grp$ is  primitive  at every finite place $\grp$.
\end{defn}

For example, $\Lambda_\grp$ itself is a primitive two-sided $\Lambda_\grp$-ideal.  More generally, given a principal fractional right $\Lambda_\grp$-ideal $L_\grp$,  we write $i(L_\grp)$ for the smallest integer $i\in \Z$ such that  $\grp^iL_\grp\subseteq \Lambda_\grp$. Then $L_\grp$ factors as the product of $\grp^{-i(L_\grp)}$ with the primitive right $\Lambda_\grp$-ideal $L_\grp^\sharp\coloneqq \grp^{i(L_\grp)}L_\grp$. Moreover, such a factorization is  unique. A similar result holds globally as follows.  

\begin{lem}\label{lem:unique-fac-prim}
    Each locally principal fractional right $\Lambda$-ideal $L\subset A$ factors uniquely as the product of a fractional $O_F$-ideal $\gra(L)$ and a primitive right $\Lambda$-ideal $L^\sharp$. 
\end{lem}
\begin{proof}
    The uniqueness of the factorization $L=\gra(L)L^\sharp$ follows directly from the local case.  For the existence, we put $\gra(L)\coloneqq \prod_\grp \grp^{-i(L_\grp)}$, where the product runs over all finite places of $F$.  Since $L_\grp=\Lambda_\grp$ for all but finitely many $\grp$, we have $i(L_\grp)=0$ for  almost all $\grp$, which shows that $\gra(L)$ is a well-defined fractional $O_F$-ideal. Now $L^\sharp\coloneqq \gra(L)^{-1}L$ is a primitive right $\Lambda$-ideal as desired. 
\end{proof}

We return to the quaternion case. Let $\calO$ be an arbitrary $O_F$-order in $D$, and $\Idl(\calO)$ be the group of invertible two-sided fractional $\calO$-ideals.  Similarly, we define its local counterpart $\Idl(\calO_\grp)$ for each finite place $\grp$ of $F$. By the above discussion, the map sending each class $[P_\grp]\in\Picent(\calO_\grp)$ to the primitive factor $P_\grp^\sharp\in \Idl(\calO_\grp)$
 provides a well-defined set-theoretical section for the quotient map \[\Idl(\calO_\grp)\to \Picent(\calO_\grp)\simeq \calN(\calO_\grp)/F_\grp^\times\calO_\grp^\times.\]
The image $\scrP_\grp$ of this section $\Picent(\calO_\grp)\to \Idl(\calO_\grp)$ consists of all primitive two-sided $\calO_\grp$-ideals, and $\abs{\scrP_\grp}=[\calN(\calO_\grp):F_\grp^\times\calO_\grp^\times]$. Recall that if $\grp$ is coprime to the reduced discriminant $\grd(\calO)$, then $\Picent(\calO_\grp)$ is trivial, and $\scrP_\grp$ is a singleton consisting of only  $\calO_\grp$. 

Globally, let $\scrP$ be the set of all primitive two-sided $\calO$-ideals, which can also be realized as the image of a set-theoretical section for the composition of the canonical surjective homomorphisms 
\[\Idl(\calO)\to \Picent(\calO)\to \prod_{\grp\mid \grd(\calO)}\Picent(\calO_\grp)\]
as follows. 
Let $\grp_1, \cdots, \grp_{\omega(\calO)}$   be a complete list of the distinct prime divisors of $\grd(\calO)$.  Suppose that we are given a primitive two-sided $\calO_{\grp_i}$-ideal $P_{\grp_i}\in\scrP_{\grp_i}$ for each $1\leq i\leq \omega(\calO)$. From the local-global dictionary of lattices \cite[Proposition~4.21]{curtis-reiner:1}, there exists a unique  two-sided $\calO$-ideal $\calP$ such that 
\begin{equation}\label{eqn:loc-glob}
  \calP_{\grp_i}=P_{\grp_i}, \qquad \forall 1\leq i\leq \omega(\calO), \quad\text{and}\quad \calP_\grq=\calO_\grq, \qquad \forall \grq\nmid \grd(\calO).  
\end{equation}
Such an ideal $\calP$ is primitive by definition. 
Conversely, every primitive two-sided $\calO$-ideal arises in this way. In particular, we have
\begin{equation}\label{eqn:global-prim-card}
    \abs{\scrP}=\prod_{\grp\mid \grd(\calO)}\abs{\scrP_\grp}=\prod_{\grp \ddiv \grd(\calO)} [\calN(\calO_\grp):F_\grp^\times\calO_\grp^\times].
\end{equation}
Moreover, combining the above discussion with the short exact sequence \eqref{eqn:picent-seq}, we obtain the following proposition. 
\begin{prop}\label{prop:choice-rep-picent}
    A complete set of representatives of $\Picent(\calO)$ can be produced in the following steps:
    \begin{enumerate}
        \item Fix a complete set $\scrI=\{\gra_1, \cdots, \gra_{h(F)}\}$ of representatives of $\Cl(O_F)$, where each $\gra_i$ is an integral ideal of $O_F$, and $ h(F)\coloneqq\abs{\Cl(O_F)}$.  
        \item List all members of the set $\scrP$ of primitive two-sided $\calO$-ideals by first listing the local ones at each prime divisor $\grp$ of the reduced discriminant $\grd(\calO)$, and then applying the local-global dictionary as in \eqref{eqn:loc-glob}.
        \item Each element of $\Picent(\calO)$ is  represented by an invertible integral two-sided ideal of the form $\gra\calP$ 
        for a unique pair $(\gra, \calP)\in \scrI\times \scrP$.  
    \end{enumerate}
    As a convention, we always choose $\gra_1=O_F$ so that the identity element of $\Cl(O_F)$ (resp.~$\Picent(\calO)$) is represented by $O_F$ (resp.~$\calO$) itself. 
\end{prop}

We conclude this section with the following simple lemma, which will be applied frequently in the next section.

\begin{lem}\label{lem:equiv-primitive}
    Let $K$ be a maximal subfield of $D$, which is necessarily a quadratic extension of $F$. Let $\grb$ be a locally principal integral ideal of the order $B\coloneqq K\cap \calO$. Then $\grb$ is a primitive $B$-ideal if and only if $\grb\calO$ is a primitive right $\calO$-ideal. 
\end{lem}
\begin{proof}
    Suppose that $\grb\calO$ is a primitive right $\calO$-ideal. Then $\grb$ is necessarily a primitive $B$-ideal. Otherwise, there exists a prime ideal $\grp$ of $O_F$ such that $\grb\subseteq \grp B$, which then implies that $\grb\calO\subseteq \grp B\calO=\grp\calO$,  contradicting the primitivity of $\grb\calO$.  

    Conversely, suppose that $\grb$ is a primitive $B$-ideal. If $\grb\calO\subseteq \grp\calO$ for some prime ideal $\grp$ of $O_F$, then $\grp^{-1}\grb\calO\subseteq \calO$, which implies that $\grp^{-1}\grb\subseteq \calO\cap K=B$. This shows that $\grb\subseteq \grp B$, contradicting  the primitivity of $\grb$. Therefore, $\grb\calO$ is a primitive right $\calO$-ideal as desired. 
\end{proof}

\section{\texorpdfstring{The spinor type number formula via the $\Picent(\calO)$-action}{The spinor type number formula via the Picent(O)-action}}\label{sec:Picent-for-tot-def-quat-alg}

In this section, we carry out the first step of the derivation of the \emph{spinor type number formula for $t_\sg(\calO)\coloneqq \abs{\Tp_\sg(\calO)}$}.  Instead of jumping directly into the spinor type number formula, we first provide a simplified proof of K\"orner's type number formula \cite{korner:1987} by substituting the more complicated \emph{Selberg trace formula}  with the much simpler Burnside's lemma based on the $\Picent(\calO)$-action on $\Cl(\calO)$. While a large portion of the material in this part is not new, this re-derivation effort will not be spent in vain. Indeed,  
the derivation of the spinor type number formula follows almost the same route, with the modifications explained in the final part of this section where we obtain a preliminary version of the desired formula. Henceforth we assume that $F$ is a totally real field with ring of integers $O_F$, and $\calO$ is an arbitrary $O_F$-order in a  totally definite quaternion $F$-algebra $D$.

As explained in the introduction, the central Picard group $\Picent(\calO)$ acts on the ideal class set $\Cl(\calO)$ of locally principal fractional right $\calO$-ideal classes  by right multiplication as follows: 
 \[  \mu:  \Cl(\calO) \times \Picent(\calO) \to \Cl(\calO), \qquad ([I],[P]) \mapsto [IP].\]
The action induces a canonical bijection $\Tp(\calO)\simeq \Cl(\calO)/ \Picent(\calO)$, which enables us to compute the type number $t(\calO)\coloneqq \abs{\Tp(\calO)}$ via Burnside's lemma \cite[Theorem 2.113]{Rotman-alg}.  More precisely, we have 
\begin{equation}\label{eq:t3.1}
    t(\calO)=\abs{\Cl(\calO)/ \Picent(\calO)}=\frac{1}{\abs{\Picent(\calO)}}\sum_{[P]\in \Picent(\calO)} \abs{\Cl(\calO)^{[P]}},
\end{equation}
where $\Cl(\calO)^{[P]}\coloneqq \{[I]\in \Cl(\calO)\mid [IP]=[I]\}$ denotes the subset of $\Cl(\calO)$ consisting of the fixed elements of $[P]\in \Picent(\calO)$.  

We fix a complete set of representatives of $\Picent(\calO)$ as in Proposition~\ref{prop:choice-rep-picent} and denote an arbitrary element of it by $P$. In other words, we first fix  a complete set  $\scrI=\{\gra_1, \cdots, \gra_{h(F)}\}$  of integral representatives of $\Cl(O_F)$ with $\gra_1=O_F$.   Then each element of $\Picent(\calO)$ is uniquely represented by an invertible integral two-sided ideal of the form $P=\gra_i\calP$, where $\calP\in\scrP$ is a primitive two-sided $\calO$-ideal uniquely determined by the class $[P]\in \Picent(\calO)$.  Moreover, the identity element of $\Picent(\calO)$ is represented by $\calO$ itself.  Henceforth whenever we write $[P]\in \Picent(\calO)$, we implicitly mean that $P$ has been chosen to be the unique representative of the form $\gra_i\calP$  as above.

Let $\{I_i\subset D\mid 1\leq i\leq h\coloneqq h(\calO)\} $ be a complete set of representatives of $\Cl(\calO)$, and we put $\calO_i\coloneqq \calO_l(I_i)=\{x\in D\mid xI_i\subseteq I_i\}$ for each $1\leq i\leq h$. 
Then $\calO_i=I_iI_i^{-1}$ and $\calO=I_i^{-1}I_i$ for each $i$, where $I_i^{-1}\coloneqq \{y\in D\mid I_iy I_i \subseteq I_i\}$.
 It follows that 
\begin{equation}\label{eq:fixed-pt-equiv}
    [I_i]\in \Cl(\calO)^{[P]} \quad \Longleftrightarrow \quad\exists \alpha\in D^\times \quad\text{such that } I_iPI_i^{-1} =\alpha \calO_i. 
\end{equation}
Naturally,  we define 
\begin{equation}\label{eq:def-X}
      X_P(\calO_i)\coloneqq\{\alpha\in D^\times\mid \alpha\calO_i=I_iPI_i^{-1}\}\subset \calO_i \cap \calN(\calO_i)
\end{equation}
for each $1\leq i\leq h$ and each representative $P$ of $\Picent(\calO)$ as above. Here  $X_P(\calO_i)\subset \calO_i$ because $P$ 
is $\calO$-integral, and $X_P(\calO_i)$ is contained in the normalizer group $\calN(\calO_i)$ because $I_iPI_i^{-1}$ is a two-sided ideal of $\calO_i$ (cf.~\cite[Exercise~I.4.6]{vigneras}). 
If $X_P(\calO_i)\neq\emptyset$, then it carries a simply transitive right $\calO_i^\times$-action. Thus
\begin{equation}\label{eq:orb-one-for-X_P}
    \abs{X_P(\calO_i)/\calO_i^\times}=
    \begin{cases}
        1 & \text{if } [I_i]\in\Cl(\calO)^{[P]},\\
        0 & \text{otherwise},
    \end{cases}
\end{equation}
from which we find that 
\begin{equation}\label{eqn:fpx}
    \abs{\Cl(\calO)^{[P]}}=\sum_{i=1}^h \abs{X_P(\calO_i)/\calO_i^\times}.
\end{equation}
Since $D$ is totally definite,  the unit group index $w_i\coloneqq [\calO_i^\times: O_F^\times]$ is finite for all $i$ by \cite[Lemma~26.5.1]{voight-quat-book}. Thus each $O_F^\times$-quotient $X_P(\calO_i)/O_F^\times$ is a finite set, and \eqref{eqn:fpx} can be rewritten as 
\begin{equation}\label{eqn:fpy}
    \abs{\Cl(\calO)^{[P]}}=\sum_{i=1}^h \frac{\abs{X_P(\calO_i)/O_F^\times}}{w_i}.
\end{equation}

Next, we count the finite sets $X_P(\calO_i)/O_F^\times$ by choosing suitable representatives for each $O_F^\times$-orbit in $X_P(\calO_i)$. 
Following \cite[Definition~16.3.1]{voight-quat-book},   the \emph{reduced norm} $\Nr(L)$ of an $O_F$-lattice $L$ in $D$ is defined to be the fractional $O_F$-ideal generated by the set $\{\Nr(x)\mid x\in L\}$. The map sending each $P$ to its reduced norm $\Nr(P)$ induces a group homomorphism from $\Picent(\calO)$ to the narrow class group $\Cl^+(O_F)$, which is still denoted by $\Nr$ by an abuse of notation. 
\begin{lem}\label{lem:norm-prin-gen}
    If  $[P]\not\in \ker\left(\Nr: \Picent(\calO)\to \Cl^+(O_F)\right)$, then $\Cl(\calO)^{[P]}=\emptyset$. 
\end{lem}
\begin{proof}
  Indeed, $\Cl(\calO)^{[P]}\neq \emptyset$ only if  $X_P(\calO_i)\neq \emptyset$ for some $1\leq i\leq h$, in which case
\begin{equation}\label{eqn:red-norm-alp}
  \Nr(P)=\Nr(I_iPI_i^{-1})=\Nr(\alpha \calO_i)=\Nr(\alpha)O_F, \qquad \forall \alpha\in X_P(\calO_i).  
\end{equation}
Since $D$ is totally definite, the reduced norm of every nonzero element of $D$ is totally positive. Therefore, $\Cl(\calO)^{[P]}\neq \emptyset$ only if $\Nr(P)$ is principally generated by a totally positive element of $F$, that is, $[P]\in \ker\left(\Nr: \Picent(\calO)\to \Cl^+(O_F)\right)$. 
\end{proof}

In light of \eqref{eqn:red-norm-alp}, we can choose representatives for the $O_F^\times$-orbits in $X_P(\calO_i)$ by specifying their reduced norms whenever $X_P(\calO_i)\neq \emptyset$. To do this,  we first fix some totally positive generators of $\Nr(P)$ for each $[P]\in \ker(\Nr)$ in a systematic way below.   Since $P$ factors uniquely as $\gra\calP$ with $\gra\in \scrI$ and $\calP\in \scrP$, its reduced norm is of the form $\gra^2\Nr(\calP)$.

\begin{sect}[Representatives for the $O_F^\times$-orbits in $X_P(\calO_i)$]\label{sect:rep-orbits-XPO}
\begin{enumerate}[label=(\arabic*), wide=0pt]
    \item Let $\nu(\calO)\coloneqq \{\Nr(\calP)\mid \calP\in \scrP\}$ be the set of reduced norms of all primitive two-sided ideals of $\calO$ as in \cite{korner:1987}.   We fix a totally positive generator $\beta_{(\gra, \wp)}$ of $\gra^2\wp$ for every pair $(\gra, \wp)\in \scrI\times \nu(\calO)$ such that the narrow ideal class $[\gra^2\wp]_+\in \Cl^+(O_F)$ is trivial.
    \item Let $\scrU$ be a complete set of representatives of $O_{F, +}^\times/O_F^{\times2}$, where $O_{F, +}^\times$ denotes the group of totally positive units of $O_F$, and $O_F^{\times2}$ its subgroup consisting of all perfect square units. 
    \item For each $P=\gra \calP$ with $[P]\in \ker(\Nr)$, we put $\wp=\Nr(\calP)$ and construct a subset of $X_P(\calO_i)$ for every $1\leq i\leq h$ as follows: 
\begin{equation}\label{eq:def-Z}
    Z_P(\calO_i)\coloneqq \{\alpha\in X_P(\calO_i)\mid \Nr(\alpha)=\varepsilon \beta_{\gra, \wp} \text{ for some } \varepsilon\in \scrU\}.
\end{equation}
The set $Z_P(\calO_i)$ is closed under multiplication by $\{\pm 1\}$, and the inclusion map $Z_P(\calO_i)\hookrightarrow X_P(\calO_i)$ induces a bijection $Z_P(\calO_i)/\{\pm 1\}\to X_P(\calO_i)/O_F^\times$.
\end{enumerate}
\end{sect}

Now \eqref{eqn:fpy} can be further rewritten as 
\begin{equation}\label{eq:fix-Z}
    \abs{\Cl(\calO)^{[P]}}=\frac{1}{2}\sum_{i=1}^h \frac{\abs{Z_P(\calO_i)}}{w_i}.
\end{equation}
To count the cardinality of $Z_P(\calO_i)$, we separate its non-central part 
\begin{equation}\label{eq:noncentral-Z}
    Z_P^\circ(\calO_i)\coloneqq Z_P(\calO_i)\smallsetminus F^\times
\end{equation}
from its central part $Z_P(\calO_i)\cap F^\times$.
\begin{lem}\label{lem:central-Z}
  We have $X_P(\calO_i)\cap F^\times\neq \emptyset$ for some   $1\leq i \leq h$ if and only if $P=\calO$, in which case $X_P(\calO_i)\cap F^\times=O_F^\times$ for every $i$. In particular, 
  \begin{equation}\label{eq:ZOF}
   \forall 1\leq i \leq h, \qquad    \abs{Z_P(\calO_i)\cap F^\times}=\begin{cases*}
          2& if $P=\calO$, \\
          0& otherwise.
      \end{cases*}
  \end{equation}
\end{lem}
\begin{proof}
    If $P=\calO$, then by definition 
    \[X_\calO(\calO_i)\cap F^\times =\{\alpha\in D^\times\mid \alpha\calO_i=I_i\calO I_i^{-1}\}\cap F^\times=\calO_i^\times\cap F^\times=O_F^\times, \quad \forall 1\leq i\leq h.\]
    Conversely, suppose that there exists $a\in X_P(\calO_i)\cap F^\times$ for some $1\leq i\leq h$. Then $P=I_i^{-1}aI_i=a\calO$, which represents the identity element of $\Picent(\calO)$.  Thus $P=\calO$ by our choice of representatives of $\Picent(\calO)$ in Proposition~\ref{prop:choice-rep-picent}. Now \eqref{eq:ZOF} follows directly from  the canonical bijection $Z_P(\calO_i)/\{\pm 1\}\to X_P(\calO_i)/O_F^\times$.
\end{proof}

From Lemma~\ref{lem:central-Z}, if $P\neq \calO$, then $Z_P^\circ(\calO_i)=Z_P(\calO_i)$ for every $1\leq i\leq h$. 
Recall from \cite[\S V.2]{vigneras} or \cite[Definition~26.1.4]{voight-quat-book} that the \emph{mass of $\calO$} is defined to be the following weighted sum 
\begin{equation}\label{eq:mass-def}
    \Mass(\calO)\coloneqq \sum_{i=1}^h\frac{1}{w_i}=\sum_{[I]\in \Cl(\calO)}\frac{1}{[\calO_l(I)^\times: O_F^\times]},
\end{equation}
whose  value can be computed explicitly by the Eichler mass formula  in \cite[Theorem~1]{korner:1987} or \cite[Theorem~26.1.5]{voight-quat-book} given as follows
    \begin{equation}
\label{eq:234}
  \Mass(\calO)=\frac{h(F)\abs{\zeta_F(-1)}\Nm(\grd(\calO))}{2^{[F:\Q]-1}}\prod_{\grp\ddiv
    \grd(\calO)}\frac{1-\Nm(\grp)^{-2}}{1-e_\grp(\calO)\Nm(\grp)^{-1}}.  
\end{equation}
Here $\grd(\calO)$ denotes the reduced discriminant of $\calO$, $e_\grp(\calO)$ is the Eichler invariant of $\calO$ at $\grp$ in Definition~\ref{defn:res-unr-order}, and $\Nm(\gra)\coloneqq \abs{O_F/\gra}$ for any integral $O_F$-ideal $\gra$.

Combining the equations \eqref{eq:t3.1}, \eqref{eq:fix-Z},  \eqref{eq:ZOF} and \eqref{eq:mass-def}, we get 
\begin{equation}\label{eq:type-step1}
\begin{split}
        t(\calO)&=\frac{\sum_{[P]\in \Picent(\calO)} \abs{\Cl(\calO)^{[P]}}}{\abs{\Picent(\calO)}}=\frac{1}{2\abs{\Picent(\calO)}}\,\sum_{[P]\in \Picent(\calO)} \sum_{i=1}^h\frac{\abs{Z_P(\calO_i)}}{w_i}\\
       &= \frac{\Mass(\calO)}{\abs{\Picent(\calO)}}+\frac{1}{2\abs{\Picent(\calO)}}\,\sum_{[P]\in \Picent(\calO)} \sum_{i=1}^h\frac{\abs{Z_P^\circ(\calO_i)}}{w_i}.
\end{split}
\end{equation}

For the rest of this section, we focus on the computation of the summation of the $\abs{Z_P^\circ(\calO_i)}/w_i$'s. Given $\alpha\in Z_P^\circ(\calO_i)$, we put $K_\alpha\coloneqq F(\alpha)$ and $B_\alpha\coloneqq K_\alpha\cap \calO_i$. Since  $Z_P^\circ(\calO_i)\subseteq X_P(\calO_i)\subset \calO_i\cap \calN(\calO_i)$ by \eqref{eq:def-X} and $D$ is totally definite,  we find that 
\begin{enumerate}[label=(\roman*)]
    \item $K_\alpha$ is a CM-extension of $F$ and $(B_\alpha, \alpha)$ is a pointed CM $O_F$-order in $K_\alpha$ with a distinguished point $\alpha\in B_\alpha$;
    \item the canonical inclusion $\iota: B_\alpha\hookrightarrow \calO_i$ is an optimal  embedding (i.e.~$\iota(K_\alpha)\cap \calO_i=\iota(B_\alpha)$) that sends $\alpha$ inside the normalizer group $\calN(\calO_i)$. In other words, $\iota$ is a restricted optimal embedding in the sense of Definition~\ref{def:opt-emb} below (see also \eqref{eq:global-REmb-intro}). 
\end{enumerate}
By definition, two pointed CM $O_F$-orders $(B,\lambda)$ and $(B', \lambda')$ are  isomorphic if there exists an $O_F$-isomorphism $\psi: B\to B'$ such that $\psi(\lambda)=\lambda'$. Such an isomorphism is necessarily unique. We collect all pointed CM $O_F$-orders of the form $(B_\alpha, \alpha)$  together and consider their set of isomorphism classes: 
\begin{equation}\label{eq:def-B0}
    \scrB_0\coloneqq \left \{(B_\alpha, \alpha)\,\middle\vert\, \alpha\in \coprod_{[P]\in \Picent(\calO)}\coprod_{i=1}^h Z_P^\circ(\calO_i)\right\}/\!\cong.
\end{equation}
Each isomorphism class $[(B_\alpha, \alpha)]\in \scrB_0$ can be realized as an abstract pointed CM $O_F$-order as follows. Let  $f(x)=x^2-\Tr(\alpha)x+\Nr(\alpha)\in O_F[x]$ be the minimal polynomial of $\alpha$ over $F$, and $\calK_f\coloneqq F[x]/(f(x))$ be the abstract quadratic extension of $F$ with the distinguished point $\lambda_0\coloneqq x+(f(x))$.  There is a canonical $F$-isomorphism $\psi: \calK_f\to K_\alpha$ mapping  $\lambda_0$ to $\alpha$.  If we put $B_0\coloneqq \psi^{-1}(B_\alpha)$, then $\psi$ defines an isomorphism between the pointed CM $O_F$-orders $(B_0, \lambda_0)$ and $(B_\alpha, \alpha)$.  Here the subscript $_0$ indicates that $(B_0, \lambda_0)$ is constructed from a member of $\scrB_0$, and  we  shall identify $\scrB_0$ with the set of abstract pointed CM $O_F$-orders consisting of the $(B_0, \lambda_0)$'s thus constructed.  
For each pointed CM $O_F$-order $(B, \lambda)$, consider the following  subset of $Z_P^\circ(\calO_i)$: 
\begin{equation}\label{eq:def-ZPBlambda}
    Z_P^\circ(\calO_i, (B, \lambda))\coloneqq \{\alpha\in Z_P(\calO_i)\mid (B_\alpha, \alpha)\cong (B, \lambda)\}. 
\end{equation}
Now $Z_P^\circ(\calO_i)$ partitions into a disjoint union of subsets according to the isomorphism classes of the $(B_\alpha, \alpha)$'s as follows: 
\begin{equation}\label{eq:part-ZP}
    Z_P^\circ(\calO_i)=\coprod_{(B_0, \lambda_0)\in \scrB_0}Z_P^\circ(\calO_i, (B_0, \lambda_0)).
\end{equation}

In the above discussion, we have constructed a finite set $\scrB_0$ of pointed CM $O_F$-orders from the finite sets $Z_P^\circ(\calO_i)$ with  $[P]\in \Picent(\calO)$ and $1\leq i \leq h$.
To compute the summation of $\abs{Z_P^\circ(\calO_i)}/w_i$'s in \eqref{eq:type-step1}, we do it in the reverse way to recover the $Z_P^\circ(\calO_i)$'s from $\scrB_0$, provided that $\scrB_0$ itself can be  narrowed down abstractly.  This can be carried out with the help of \emph{restricted optimal embeddings},  a  notion introduced by K\"orner \cite{korner:1987} that is finer than the  classical notion of \emph{optimal embeddings}. For the reader's convenience, we  provide its precise definition below. It can be defined more generally for both the global and local cases. In particular, there is no need to restrict to the totally definite case.  See \cite[Chapter~30]{voight-quat-book} for a detailed exposition on optimal embeddings.

\begin{defn}\label{def:opt-emb}
Let $O_\calF$ be the ring of integers of either a global field or a nonarchimedean local field $\calF$, and $\grO$ be an  $O_\calF$-order in an arbitrary quaternion $\calF$-algebra $\calD$. Let $\calB$ be an $O_\calF$-order in a semisimple quadratic $\calF$-algebra $\calK$.   We write $\Emb(\calB, \grO)$ for the set of optimal embeddings from $\calB$ into $\grO$ as in \eqref{eq:def-opt-emb}.  Fix an element $\lambda\in \calB\smallsetminus O_\calF$. 
    A \emph{restricted optimal embedding} from the pointed  $O_\calF$-order $(\calB, \lambda)$ into $\grO$ is an optimal embedding $\varphi\in \Emb(\calB, \grO)$ that sends $\lambda$ into the normalizer group $\calN(\grO)$.  The set of restricted optimal embeddings is denoted by
\begin{equation}
    \REmb((\calB, \lambda), \grO)\coloneqq \{\varphi\in \Emb(\calB, \grO)\mid  \varphi(\lambda)\in \calN(\grO)\}.  
\end{equation}
The set $ \REmb((\calB, \lambda), \grO)$  is stable under the right conjugate action of $\grO^\times$ on $\Emb(\calB, \grO)$, and we denote the number of orbits as 
\begin{equation}\label{eq:def-conj-res-opt-emb}
    n((\calB, \lambda), \grO, \grO^\times)\coloneqq \abs{\REmb((\calB, \lambda), \grO)/\grO^\times}. 
\end{equation}
\end{defn}


\begin{rem}\label{rem:res-opt-emb-eq-1}
    If $\lambda\in \calF^\times\calB^\times$, then $\varphi(\lambda)\in \calF^\times\grO^\times \subset \calN(\grO)$ for every $\varphi\in \Emb(\calB, \grO)$, from which it follows that $\REmb((\calB, \lambda), \grO)=\Emb(\calB, \grO)$ in this case.  If further $\calF$ is a nonarchimedean local field and $\grO$ is the split maximal order $\Mat_2(O_\calF)$, then 
    $n((\calB, \lambda), \grO, \grO^\times)=m(\calB, \grO, \grO^\times)=1$
    according to \cite[Theorem~II.3.2]{vigneras}, where  $m(\calB,  \grO, \grO^\times)\coloneqq \abs{\Emb(\calB, \grO)/\grO^\times}$  as in \eqref{eq:no-conjcls-opt}.
\end{rem}

We return to the previous totally definite setting. Let $(B, \lambda)$ be a pointed CM $O_F$-order with fraction field $K$.  
For each $1\leq i\leq h$, an element  $\varphi\in \REmb((B, \lambda), \calO_i)$ determines an invertible two-sided $\calO$-ideal $I_i^{-1}\varphi(\lambda) I_i$. Given $[P]\in \Picent(\calO)$, we put 
\begin{equation}
    \REmb_P((B, \lambda), \calO_i)\coloneqq \{\varphi\in \REmb((B, \lambda), \calO_i)\mid I_i^{-1}\varphi(\lambda)I_i=P\}. 
\end{equation}
For example, if $\alpha\in Z^\circ_P(\calO_i)$, then the inclusion map $\iota: B_\alpha\hookrightarrow \calO_i$ is an element of $\REmb_P((B_\alpha, \alpha), \calO_i)$, and there is a canonical bijection 
\begin{equation}\label{eq:tauto-bij}
     \REmb_P((B_\alpha, \alpha), \calO_i)\xrightarrow{\simeq} Z_P^\circ(\calO_i, (B_\alpha, \alpha)),\qquad \varphi\mapsto \varphi(\alpha). 
\end{equation}
In light of the decomposition \eqref{eq:part-ZP}, we would like to show that the above bijection holds more generally for all pointed CM $O_F$-orders $(B_0, \lambda_0)\in \scrB_0$ in place of $(B_\alpha, \alpha)$.  
By definition $\varphi(\lambda_0)\in X_P(\calO_i)$ for every $\varphi\in \REmb_P((B_0, \lambda_0), \calO_i)$. However, it is not a priori clear that $\varphi(\lambda_0)\in Z_P^\circ(\calO_i)$, since  $(B_0, \lambda_0)$ might arise from some element of $Z_{P'}^\circ(\calO_j)$ with $[P']\in \Picent(\calO)$ and $1\leq j\leq h$ different from $[P]$ and $i$. To clarify the situation, recall from Section~\ref{sect:rep-orbits-XPO} that $\nu(\calO)=\{\Nr(\calP)\mid \calP\in \scrP\}$ denotes the set of reduced norms of all primitive two-sided ideals of $\calO$. 
For each pair $(\gra, \wp)\in \scrI\times \nu(\calO)$,   we have fixed a totally positive generator $\beta_{(\gra, \wp)}$ of $\gra^2\wp$ whenever  $[\gra^2\wp]_+\in\Cl^+(O_F)$ is trivial. Moreover, we have also fixed  a complete set of representatives $\scrU$ of $O_{F, +}^\times/O_F^{\times2}$.  In the following proposition, we produce  a  finite set $\scrB\supset \scrB_0$ of pointed CM $O_F$-orders whose members serve as possible candidates for $\scrB_0$. The set $\scrB_0$ is then cut out of $\scrB$ and precisely characterized  in terms of restricted optimal embeddings in \eqref{eq:char-B0}, and the desired generalization of \eqref{eq:tauto-bij} is shown to hold true for every member $(B, \lambda)\in \scrB$ in \eqref{eq:bij-gen-BZ}.

\begin{prop}\label{prop:def-B}
    \begin{enumerate}[label={(\arabic*)},  leftmargin=*]
    \item For each $(B_0, \lambda_0)\in \scrB_0$, there is a unique triple $(\gra, \wp, \varepsilon)\in \scrI\times \nu(\calO)\times \scrU$ such that both of the following hold: 
    \begin{enumerate}[label={\upshape(\roman*)}, align=right, widest=iii,  leftmargin=*]
        \item $\gra^{-1}\lambda_0 B_0$ is a primitive $B_0$-ideal;
        \item $\Nm_{K_0/F}(\lambda_0)=\varepsilon \beta_{(\gra, \wp)}$,  where $K_0\coloneqq \Frac(B_0)$. 
    \end{enumerate}
    Particularly, we have $\Tr_{K_0/F}(\lambda_0)\in \gra$, and $\Nm_{K_0/F}(\lambda_0)O_F=\gra^2\wp$. 
   \item Conversely, for each pair $(\gra, \wp)\in \scrI\times \nu(\calO)$ with $[\gra^2\wp]_+\in\Cl^+(O_F)$ trivial, we construct 
   a  set  $\scrB_{(\gra, \wp)}$ of  pointed CM $O_F$-orders  in the following three steps: 
    \begin{enumerate}[label=(Step~\arabic*), align=right,  leftmargin=*]
    \item Let  $\Pol(\gra, \wp)$ be the set of all monic quadratic  polynomials  of the form
    \begin{equation}\label{eq:irr-poly}
        f_{t, \varepsilon}(x)\coloneqq x^2-tx+\varepsilon\beta_{(\gra,\wp)}\in O_F[x], 
    \end{equation}
        such that $t\in\gra$, $\varepsilon\in \scrU$ and $t^2-4\varepsilon\beta_{(\gra,\wp)}$ is totally negative.
    \item Given $f_{t, \varepsilon}(x)\in \Pol(\gra, \wp)$,  let $\calK_{t, \varepsilon}/F$ be the CM-extension  $F[x]/(f_{t, \varepsilon}(x))$, which is equipped with a distinguished point $\lambda_{t, \varepsilon}\coloneqq x+(f_{t, \varepsilon}(x))$.
    \item Let $\scrB_{(\gra, \wp)}$ be the finite set of pointed CM $O_F$-orders $(B, \lambda)$ such that  all of the following hold
    \begin{itemize}
        \item $B$ is an $O_F$-order in $\calK_{t, \varepsilon}$ for some $f_{t, \varepsilon}(x)\in \Pol(\gra, \wp)$;
        \item the distinguished point  $\lambda\coloneqq \lambda_{t, \varepsilon}\in \calK_{t, \varepsilon}$ is contained in $\gra B$;
        \item $\gra^{-1}\lambda B$ is a primitive $B$-ideal. 
    \end{itemize}
    \end{enumerate}
       Then the union 
       \begin{equation}\label{eq:union-Bap}
          \scrB\coloneqq \bigcup \scrB_{(\gra, \wp)} 
       \end{equation}
        of all such $\scrB_{(\gra, \wp)}$ constructed above is a finite set of pointed CM $O_F$-orders containing a copy of $\scrB_0$. Moreover, the union in \eqref{eq:union-Bap} is disjoint and  all members of $\scrB$ are mutually non-isomorphic. 
           \item For every $(B, \lambda)\in \scrB$ and every $1\leq i\leq h$, the set $\REmb((B, \lambda), \calO_i)$  decomposes into a disjoint union 
     \begin{equation}\label{eq:decomp-REmb}       \REmb((B,\lambda),\calO_i)=\bigsqcup_{[P]\in \Picent(\calO)} \REmb_P((B,\lambda),\calO_i),
    \end{equation}
 where each resulting piece  of the decomposition admits          
 a canonical bijection 
\begin{equation}\label{eq:bij-gen-BZ}
    \REmb_P((B, \lambda), \calO_i)\xrightarrow{\simeq} Z_P^\circ(\calO_i, (B, \lambda)),\qquad \varphi\mapsto \varphi(\lambda). 
\end{equation} 
    In particular, 
    \begin{equation}\label{eq:char-B0}
        \scrB_0=\{(B, \lambda)\in \scrB\mid \exists 1\leq i\leq h \text{ such that } \REmb((B, \lambda), \calO_i)\neq \emptyset\}. 
    \end{equation}
    \end{enumerate}
\end{prop}

\begin{proof}
(1) Let $(B_0, \lambda_0)$ be a pointed CM $O_F$-order in $\scrB_0$.
By definition \eqref{eq:def-B0}, there exists $\alpha\in Z_P^\circ(\calO_i)$  for some $[P]\in \Picent(\calO)$ and $1\leq i \leq h$ such that $(B_0, \lambda_0)\cong (B_\alpha, \alpha)$.   Let $P=\gra \calP$ be the unique factorization with $(\gra, \calP)\in \scrI\times \scrP$.  Then $I_i^{-1}\calP I_i$ is a primitive two-sided $\calO_i$-ideal that coincides with $\gra^{-1}\alpha\calO_i$ by \eqref{eq:def-X}. In particular, $\gra^{-1}\alpha\subset \calO_i\cap F(\alpha)=B_\alpha$. Now it follows from Lemma~\ref{lem:equiv-primitive} that $\gra^{-1}\alpha B_\alpha$ is a primitive $B_\alpha$-ideal. By transport of structure, $\gra^{-1}\lambda_0B_0$ is a primitive $B_0$-ideal, so $\lambda_0 B_0=\gra \cdot (\gra^{-1}\lambda_0B_0)$ is a factorization of $\lambda_0 B_0$ into the product of an $O_F$-ideal $\gra$ and a primitive $B_0$-ideal $\gra^{-1}\lambda_0B_0$. Such a factorization is  unique  by Lemma~\ref{lem:unique-fac-prim}, and hence $\gra\in \scrI$ depends uniquely on the pointed CM $O_F$-order $(B_0, \lambda_0)$ and not on the choice of $\alpha\in Z_P^\circ(\calO_i)$ (nor on the choice of $P$).
We have seen in \eqref{eqn:red-norm-alp} that 
\[\Nm_{K_0/F}(\lambda_0)O_F=\Nr(\alpha)O_F=\Nr(P)=\Nr(\gra\calP)=\gra^2\Nr(\calP), \]
so there exists an $O_F$-ideal $\wp\coloneqq \Nr(\calP)\in \nu(\calO)$ with $\Nm_{K_0/F}(\lambda_0)O_F=\gra^2\wp$. Since $\gra\in \scrI$ is already uniquely determined by $(B_0,\lambda_0)$, the same holds for $\wp\in \nu(\calO)$. From \eqref{eq:def-Z}, there is a unique $\varepsilon\in \scrU$ depending only on $(B_0, \lambda_0)$ such that  
\[\Nm_{K_0/F}(\lambda_0)=\Nr(\alpha)=\varepsilon \beta_{(\gra, \wp)}.\]
This proves the existence and uniqueness of the triple $(\gra, \wp, \varepsilon)\in \scrI\times \nu(\calO)\times \scrU$ satisfying conditions (i)--(ii).  Observe that the above argument also shows that given $[P']\in \Picent(\calO)$, we have  $Z_{P'}^\circ(\calO_i, (B_0,\lambda_0))\neq \emptyset$ only if $P'=\gra\calP'$ with $\Nr(\calP')=\wp$. 
Lastly,  the primitive $B_0$-ideal $\gra^{-1} \lambda_0 B_0$ is integral by definition, so $\lambda_0\in \gra B_0$, which in turn implies that $\Tr_{K_0/F}(\lambda_0)\in \gra B_0\cap O_F=\gra$.

(2) By construction, $\scrB_0$ forms a subset of $\scrB$. To show that $|\scrB|<\infty$, it is enough to show that each $\scrB_{(\gra, \wp)}$ is a finite set for every pair  $(\gra, \wp)\in \scrI\times \nu(\calO)$ with $[\gra^2\wp]_+\in\Cl^+(O_F)$ trivial.  
For each $f_{t, \varepsilon}\in \Pol(\gra, \wp)$, there are only finitely many $O_F$-orders in $\calK_{t, \varepsilon}$ containing $O_F[\lambda_{t, \varepsilon}]$. Thus the question reduces further to proving that $\Pol(\gra, \wp)$ is finite, or equivalently, proving that for every $\varepsilon\in \scrU$ there are only finitely many $t\in \gra$ with  $t^2-4\varepsilon\beta_{(\gra, \wp)}$ totally negative. This last finiteness condition can be easily verified, as $\gra$ forms a $\Z$-lattice in $F\otimes_\Z\R$ and all archimedean absolute values of the desired $t$'s are bounded. This concludes the proof that $|\scrB|<\infty$.  

To show that the union in \eqref{eq:union-Bap} is disjoint, observe that each $(B, \lambda)\in \scrB$ remembers which $\scrB_{(\gra, \wp)}$ it resides in. Indeed, $\gra$ is the unique $O_F$-ideal such that $\gra^{-1} \lambda B$ is a primitive $B$-ideal, and $\wp=\gra^{-2}\Nm_{K/F}(\lambda)O_F$ with $K=\Frac(B)$. In other words, $(\gra, \wp)$ is a pair of invariants of the isomorphism class of $(B, \lambda)$.  Thus $\scrB_{(\gra, \wp)}$ and $\scrB_{(\gra', \wp')}$
 share no isomorphic pointed CM $O_F$-orders whenever $(\gra, \wp)\neq (\gra', \wp')$. Inside a particular $\scrB_{(\gra, \wp)}$, each member $(B, \lambda)$ is uniquely determined  by another pair of isomorphism invariants $(f(x), \grf_B)$, where $f(x)\in \Pol(\gra, \wp)$ is the minimal polynomial of $\lambda$ over $F$, and $\grf_B$ is the conductor of $B$ (i.e.~the unique $O_F$-ideal such that $B=O_F+\grf_B O_K$).  Therefore, all members of $\scrB$ are mutually non-isomorphic.

(3) Let $(B, \lambda)$ be an arbitrary pointed CM $O_F$-order in $\scrB$ and fix an $\calO_i$ with $1\leq i\leq h$. 
By definition, the union $\bigcup_P \REmb_P((B,\lambda),\calO_i)$  is  disjoint, so the right hand side of \eqref{eq:decomp-REmb}  forms a subset of $\REmb((B, \lambda), \calO_i)$.  To prove the equality in \eqref{eq:decomp-REmb},  it is enough to show that every $\varphi\in \REmb((B,\lambda),\calO_i)$ falls inside $\REmb_{P_\varphi}((B,\lambda),\calO_i)$ for some $P_\varphi=\gra\calP_\varphi$ with $(\gra, \calP_\varphi)\in \scrI\times \scrP$. 
 From part (2) of the current proposition, there is a unique pair $(\gra, \wp)\in \scrI\times \nu(\calO)$ such that $(B, \lambda)\in \scrB_{(\gra, \wp)}$. In particular, $\gra^{-1}\lambda B$ is a primitive $B$-ideal.
 Since $\varphi \in \Emb(B, \calO_i)$ by definition, it follows from Lemma~\ref{lem:equiv-primitive} that $\gra^{-1}\varphi(\lambda)\calO_i$ is a primitive two-sided $\calO_i$-ideal, or equivalently, $\calP_\varphi\coloneqq \gra^{-1} I_i^{-1}\varphi(\lambda) I_i$ is a primitive two-sided $\calO$-ideal.
If we put $P_\varphi\coloneqq \gra \calP_\varphi$, then $P_\varphi$ is one of the fixed representatives of $\Picent(\calO)$ fixed in Proposition~\ref{prop:choice-rep-picent}, and we have $\varphi\in \REmb_{P_\varphi}((B,\lambda),\calO_i)$. This proves the equality in \eqref{eq:decomp-REmb}. Observe that $\Nr(\calP_\varphi)=\Nr(\gra^{-1} I_i^{-1}\varphi(\lambda) I_i)=\gra^{-2}\Nm_{K/F}(\lambda)=\gra^{-2}\varepsilon \beta_{(\gra, \wp)}=\wp$, where $\varepsilon\in \scrU$ and $K=\Frac(B)$.  Thus the decomposition of $\REmb((B, \lambda), \calO_i)$ in \eqref{eq:decomp-REmb} can be further sharpened into 
\begin{equation}
    \REmb((B,\lambda),\calO_i)=\bigsqcup_{\Nr(\calP)=\wp} \REmb_{\gra\calP}((B,\lambda),\calO_i). 
\end{equation} 
To prove that  the map $\varphi\mapsto \varphi(\lambda)$ establishes a bijection between $\REmb_P((B, \lambda), \calO_i)$ and $Z_P^\circ(\calO_i, (B,\lambda))$ for every $[P]\in \Picent(\calO)$, we may restrict ourselves to those  $P$ that  factorize as $\gra\calP$ with $\Nr(\calP)=\wp$, for otherwise both sets are empty. By the construction of $\scrB_{(\gra, \wp)}$, we have $\Nm_{K/F}(\lambda)=\varepsilon \beta_{(\gra, \wp)}$ for some $\varepsilon\in \scrU$, so it follows from \eqref{eq:def-Z} that $\varphi(\lambda)\in Z_P^\circ(\calO_i)$ for every $\varphi\in \REmb_P((B, \lambda), \calO_i)$. Now the bijectiveness of the canonical map $\REmb_P((B, \lambda), \calO_i)\to Z_P^\circ(\calO_i, (B,\lambda))$
follows directly from the fact that an isomorphism between two pointed CM $O_F$-orders $(B, \lambda)$ and $(B_\alpha, \alpha)$ is unique if it exists. 

Lastly, a pointed CM $O_F$-order  $(B, \lambda)$ belongs to $\scrB_0$ if and only if $Z_P^\circ(\calO_i, (B, \lambda))$ is nonempty for some 
$[P]\in \Picent(\calO)$ and $1\leq i\leq h$. Combining \eqref{eq:decomp-REmb} and \eqref{eq:bij-gen-BZ}, we immediately obtain the characterization of $\scrB_0$ in \eqref{eq:char-B0}.
\end{proof}

We return to the computation of the type number $t(\calO)$. Plugging \eqref{eq:part-ZP}  into \eqref{eq:type-step1} and applying the bijection \eqref{eq:bij-gen-BZ}, we get 
\begin{equation*}
\begin{split}
         t(\calO)&=\frac{\Mass(\calO)}{\abs{\Picent(\calO)}}+\frac{1}{2\abs{\Picent(\calO)}}\,\sum_{[P]\in \Picent(\calO)} \sum_{i=1}^h \sum_{(B_0, \lambda_0)\in \scrB_0} \frac{\abs{Z_P^\circ(\calO_i, (B_0, \lambda_0))}}{w_i}\\
        &= \frac{\Mass(\calO)}{\abs{\Picent(\calO)}}+\frac{1}{2\abs{\Picent(\calO)}}\,\sum_{[P]\in \Picent(\calO)} \sum_{i=1}^h \sum_{(B_0, \lambda_0)\in \scrB_0} \frac{\abs{\REmb_P((B_0, \lambda_0), \calO_i)}}{w_i}.
        \end{split}
        \end{equation*}
        We  swap the order of summation and then apply \eqref{eq:decomp-REmb} to obtain 
        \begin{equation*}
            \begin{split}
   t(\calO)&= \frac{\Mass(\calO)}{\abs{\Picent(\calO)}}+\frac{1}{2\abs{\Picent(\calO)}}\, \sum_{(B_0, \lambda_0)\in \scrB_0}\sum_{i=1}^h \sum_{[P]\in \Picent(\calO)} \frac{\abs{\REmb_P((B_0, \lambda_0), \calO_i)}}{w_i}\\    
    &=\frac{\Mass(\calO)}{\abs{\Picent(\calO)}}+\frac{1}{2\abs{\Picent(\calO)}}\, \sum_{(B_0, \lambda_0)\in \scrB_0} \sum_{i=1}^h \frac{\abs{\REmb((B_0, \lambda_0), \calO_i)}}{w_i}\\  
\end{split}
\end{equation*}
From the characterization of $\scrB_0$ in \eqref{eq:char-B0}, it is harmless to replace $\scrB_0$ by $\scrB$ in the above summation: 
\begin{equation}\label{eq:type-step2}
    t(\calO)    
    =\frac{\Mass(\calO)}{\abs{\Picent(\calO)}}+\frac{1}{2\abs{\Picent(\calO)}}\, \sum_{(B, \lambda)\in \scrB} \sum_{i=1}^h \frac{\abs{\REmb((B, \lambda), \calO_i)}}{w_i}.
\end{equation}
Recall from Definition~\ref{def:opt-emb} that the unit group $\calO_i^\times$ acts on $\REmb((B, \lambda), \calO_i)$ from the right by sending $\varphi\mapsto u^{-1} \varphi u$ for each  $\varphi\in \REmb((B, \lambda), \calO_i)$ and $u\in \calO_i^\times$. Clearly,  the stabilizer of $\varphi$ is $\varphi(B^\times)=\calO_i^\times\cap \Frac(B)^\times$, so the number of elements in the $\calO_i^\times$-orbit of $\varphi$ is given by 
\[[\calO_i^\times: \varphi(B^\times)]=\frac{[\calO_i^\times:O_F^\times]}{[B^\times:O_F^\times]}=\frac{w_i}{[B^\times:O_F^\times]},\]
which is independent of $\varphi$. Thus if we put $w(B)\coloneqq [B^\times: O_F^\times]$, then we have 
\begin{equation}\label{eq:orbits-res-opt}
    \abs{\REmb((B, \lambda), \calO_i)}=\abs{\REmb((B, \lambda), \calO_i)/\calO_i^\times}\cdot \frac{w_i}{w(B)}=\frac{n((B, \lambda), \calO_i, \calO_i^\times)w_i}{w(B)},
\end{equation}
where $n((B, \lambda), \calO_i, \calO_i^\times)$ denotes the number of conjugacy classes of restricted optimal embeddings as in \eqref{eq:def-conj-res-opt-emb}.
Plugging \eqref{eq:orbits-res-opt} into \eqref{eq:type-step2}, we get 
\begin{equation}\label{eq:type-step3}
     t(\calO)    
    =\frac{\Mass(\calO)}{\abs{\Picent(\calO)}}+\frac{1}{2\abs{\Picent(\calO)}}\, \sum_{(B, \lambda)\in \scrB} \frac{1}{w(B)} \sum_{i=1}^h n((B, \lambda), \calO_i, \calO_i^\times).
\end{equation}

Classically in the proof of the Eichler class number formula \cite[Theorem~30.8.6]{voight-quat-book} or more generally the Eichler trace formula \cite[Theorem~41.5.2]{voight-quat-book}, the final step is to apply the \emph{trace formula for optimal embeddings} \cite[Theorem~30.4.7]{voight-quat-book}, which relates the numbers $m(B, \calO_i, \calO_i^\times)\coloneqq \abs{\Emb(B, \calO_i)/\calO_i^\times}$ of conjugacy classes of global optimal embeddings  with those of the local ones for any CM $O_F$-order $B$:
\begin{equation}\label{eq:classic-trf-opt}
    \sum_{i=1}^h m(B, \calO_i, \calO_i^\times)=h(B)\prod_\grp m(B_\grp, \calO_\grp, \calO_\grp^\times). 
\end{equation}
Here $h(B)\coloneqq \abs{\Pic(B)}$ denotes the class number of the CM $O_F$-order $B$, the product runs over all finite places of $F$, and $B_\grp$ (resp.~$\calO_\grp$) denotes the $\grp$-adic completion of $B$ (resp.~$\calO$).  In Proposition~\ref{prop:trf-res-opt-emb},  we shall prove an analogous trace formula for restricted optimal embeddings that applies to every pointed CM $O_F$-order $(B, \lambda)$:
\begin{equation}\label{eq:trf-roe}
    \sum_{i=1}^h n((B,\lambda), \calO_i, \calO_i^\times)=h(B)\prod_\grp n((B_\grp,\lambda), \calO_\grp, \calO_\grp^\times). 
\end{equation}
Here the product on the right hand side is finite by Remark~\ref{rem:finite-prod}, and it can be further simplified when $(B, \lambda)\in \scrB$ by the following lemma. 
\begin{lem}\label{lem:n1-pnmid-d}
  Let $\grd(\calO)$ be the reduced discriminant of $\calO$. 
  Then
\begin{equation}\label{eq:n1-p-nmid-d}
    n((B_\grp,\lambda), \calO_\grp, \calO_\grp^\times)=1, \qquad \forall \grp\nmid \grd(\calO), \quad \forall (B, \lambda)\in \scrB.
\end{equation}
\end{lem}
\begin{proof}
Let $(\gra, \wp)\in \scrI\times \nu(\calO)$ be the pair of isomorphism invariants of $(B, \lambda)\in \scrB$ such that $(B, \lambda)\in \scrB_{(\gra, \wp)}$ and put $K\coloneqq \Frac(B)$.  By construction,  $\Nm_{K/F}(\lambda)O_F=\gra^2\wp$, and $\wp=\Nr(\calP)$ for some primitive two-sided $\calO$-ideal $\calP$.  Since $\grp\nmid \grd(\calO)$, we have $\calO_\grp\simeq \Mat_2(O_{F_\grp})$, and hence $\calP_\grp=\calO_\grp$.   Let $\pi_\grp$ be a uniformizer of $F_\grp$. Then $\Nm_{K/F}(\pi_\grp^{-\ord_\grp(\gra)}\lambda)O_{F_\grp}=\wp_\grp=\Nr(\calO_\grp)=O_{F_\grp}$. On the other hand, $\pi_\grp^{-\ord_\grp(\gra)}\lambda\in B_\grp$ since $\gra^{-1}\lambda B$ is $B$-primitive (particularly, $B$-integral). We conclude that $\pi_\grp^{-\ord_\grp(\gra)}\lambda\in B_\grp^\times$. Now  \eqref{eq:n1-p-nmid-d} follows immediately from Remark~\ref{rem:res-opt-emb-eq-1}.     
\end{proof}

The trace formula \eqref{eq:trf-roe} was used implicitly\footnote{In K\"orner's notation, the order $\calO$ here is denoted by $G$ and a pointed CM $O_F$-order $(B, \lambda)\in \scrB$ is denoted by $(\Omega, x)$ (where $x$ is often omitted). Correspondingly,  the  number  $n((B_\grp, \lambda), \calO_\grp, \calO_\grp^\times)$  is denoted by him as $E'(\Omega_p, G_p)$, and their product as $E'_G(\Omega)$. However, the left hand side of \eqref{eq:trf-roe} does not appear explicitly in  K\"orner's paper \cite{korner:1987} as he  works with orbital integrals instead.  } by K\"orner in \cite{korner:1987}.   Taking it granted for now, we immediately recover K\"orner's type number formula by combining \eqref{eq:trf-roe} with \eqref{eq:type-step3}.

\begin{thm}[{\cite[Theorem 3]{korner:1987}}] \label{thm:Korner}
  The type number $t(\calO)$ of any  $O_F$-order $\calO$ in a totally definite quaternion $F$-algebra $D$  can be computed as follows:
    \begin{equation}\label{eq:Korner-formula}
        t(\calO)=\frac{1}{\abs{\Picent(\calO)}} \Big( \Mass(\calO)+\frac{1}{2} \sum_{(B,\lambda)\in\scrB} \frac{h(B)}{w(B)}\prod_{\grp\mid \grd(\calO)} n((B_\grp, \lambda), \calO_\grp, \calO_\grp^\times) \Big), 
    \end{equation}
 where 
 \begin{itemize}
     \item  $\Mass(\calO)$ denotes the mass of $\calO$ as defined in \eqref{eq:mass-def} whose value can be computed by the Eichler mass formula \eqref{eq:234},
     \item $\scrB$ denotes the finite set of pointed CM $O_F$-orders $(B, \lambda)$ constructed in Proposition~\ref{prop:def-B},
     \item $h(B)$ (resp.~$w(B)$) denotes the class number (resp.~the unit index $[B^\times: O_F^\times]$) of the  CM $O_F$-order $B$, and the value of  
     $h(B)$ can be calculated by the Dedekind’s formula in \cite[\S III.5, p.~95]{vigneras} or \cite[p.~74]{vigneras:ens}.  
 \end{itemize}
\end{thm}

Finally, we embark on our main goal of deriving a \emph{spinor type number formula} for $t_\sg(\calO)\coloneqq \abs{\Tp_\sg(\calO)}$, where as in Definition~\ref{defn:spinor-genus}, $\Tp_\sg(\calO)\coloneqq \{[\calO']\in \Tp(\calO)\mid [\calO']\subseteq [\calO]_\sg\}$ denotes the subset of $\Tp(\calO)$ consisting of all types of orders belonging to the same spinor genus of $\calO$.  Recall from \eqref{eq:Cl-to-Tp} that we have a canonical surjective map $\xi: \Cl(\calO)\twoheadrightarrow \Tp(\calO)$ whose fibers are identified with the $\Picent(\calO)$-orbits of the right multiplicative action $\mu: \Cl(
\calO)\times \Picent(\calO)\to \Cl(\calO)$ in \eqref{eq:Picent-action-defn}.  Indeed, it is this group action that allows us to identify $\Tp(\calO)$ with $\Cl(\calO)/\Picent(\calO)$ so that its cardinality  can be computed via Burnside's lemma.  Naturally, if we define
\begin{equation}\label{eq:defn-cl-sg}
    \Cl_\sg(\calO)\coloneqq \xi^{-1}(\Tp_\sg(\calO))=\{[I]\in \Cl(\calO)\mid [\calO_l(I)]\subseteq [\calO]_\sg\}, 
\end{equation}
then $\Cl_\sg(\calO)$ is stable under the $\Picent(\calO)$-action, and $\Tp_\sg(\calO)$ can be canonically identified with $\Cl_\sg(\calO)/\Picent(\calO)$. We apply Burnside's lemma again to obtain 
\begin{equation}
      t_\sg(\calO)=\abs{\Cl_\sg(\calO)/ \Picent(\calO)}=\frac{1}{\abs{\Picent(\calO)}}\sum_{[P]\in \Picent(\calO)} \abs{\Cl_\sg(\calO)^{[P]}}.
\end{equation}

At the beginning of this section,  we have fixed a complete set $\{I_i\subset D\mid 1\leq i\leq h\} $  of representatives of $\Cl(\calO)$. Relabeling if necessary, we assume that the first $h_\sg\coloneqq \abs{\Cl_\sg(\calO)}$ of them $I_1, \cdots, I_{h_\sg}$ form a complete set of representatives of $\Cl_\sg(\calO)$.
Now to compute $t_\sg(\calO)$, we perform almost the same calculation as before, only replacing $\Cl(\calO)$ by $\Cl_\sg(\calO)$ (and respectively, replacing the summation $\sum_{i=1}^h$ by $\sum_{i=1}^{h_\sg}$) whenever applicable.  For example, analogous to \eqref{eq:fix-Z},   the number of fixed points of each $[P]\in \Picent(\calO)$ in $\Cl_\sg(\calO)$ is given by 
\begin{equation}
      \abs{\Cl_\sg(\calO)^{[P]}}=\frac{1}{2}\sum_{i=1}^{h_\sg} \frac{\abs{Z_P(\calO_i)}}{w_i}.
\end{equation}
Correspondingly, \eqref{eq:type-step1}  changes into 
\begin{equation}
  t_\sg(\calO)=  \frac{\Mass_\sg(\calO)}{\abs{\Picent(\calO)}}+\frac{1}{2\abs{\Picent(\calO)}}\,\sum_{[P]\in \Picent(\calO)} \sum_{i=1}^{h_\sg}\frac{\abs{Z_P^\circ(\calO_i)}}{w_i},
\end{equation}
where $\Mass_\sg(\calO)$ is defined by the following formula and called \emph{the mass of $\Cl_\sg(\calO)$}:
\begin{equation}\label{eq:defn-mass-sg}
     \Mass_\sg(\calO)\coloneqq \sum_{i=1}^{h_\sg}\frac{1}{w_i}=\sum_{[I]\in \Cl_\sg(\calO)}\frac{1}{[\calO_l(I)^\times: O_F^\times]}.
\end{equation}
  Observe that both partitions in \eqref{eq:part-ZP} and \eqref{eq:decomp-REmb} hold true for every fixed $i$ between $1$ and $h$, and so is the bijection \eqref{eq:bij-gen-BZ}.  Continuing the same calculation as before, we obtain a preliminary version of the spinor type number formula that applies to any arbitrary $\calO$.
\begin{prop}\label{prop:spinor-typ-prelim}
      The spinor type number $t_\sg(\calO)$ of any  $O_F$-order $\calO$ in a totally definite quaternion $F$-algebra $D$  is given by
  \begin{equation}\label{eq:spinor-typ-prelim}
        t_\sg(\calO)    
    =\frac{\Mass_\sg(\calO)}{\abs{\Picent(\calO)}}+\frac{1}{2\abs{\Picent(\calO)}}\, \sum_{(B, \lambda)\in \scrB} \frac{1}{w(B)} \sum_{i=1}^{h_\sg} n((B, \lambda), \calO_i, \calO_i^\times).
  \end{equation}    
\end{prop}
Admittedly, this formula is not effective since we generally do not know  how to compute the partial sum $\sum_{i=1}^{h_\sg} n((B, \lambda), \calO_i, \calO_i^\times)$.  To obtain an effective formula in the same style as K\"orner's type number formula, we need a refinement of the trace formula \eqref{eq:trf-roe}. 
Such a refinement will be worked out in Section~\ref{sec:spinor-trace-formula} under the additional assumption that $\calO$ is residually unramified.   This is also precisely the reason why we need the residually-unramified assumption in our final form of the spinor type number formula in Theorem~\ref{thm:spinor-type-number-formula-intro}.

\section{Spinor trace formula for restricted optimal embeddings}\label{sec:spinor-trace-formula}

In this section we produce a refinement of the trace formula \eqref{eq:trf-roe}  called \emph{the spinor trace formula for restricted optimal embeddings}. It is modeled after the spinor trace formula for optimal embeddings in \cite[Proposition~4.3]{Xue-Yu-Selec-2024} developed by Chia-Fu Yu and the second named author, which  in turn is a generalization of the Vign\'eras-Voight formula \cite[Theorem~31.1.7(c) and Corollary~31.1.10]{voight-quat-book} \cite[Proposition 4.1]{Xue-Yu-Selec-2024} for Eichler orders in quaternion algebras that satisfy the Eichler Condition (cf.~Remark~\ref{rem:EC}). This refined trace formula allows us to evaluate the partial sum $\sum_{i=1}^{h_\sg} n((B, \lambda), \calO_i, \calO_i^\times)$ in \eqref{eq:spinor-typ-prelim} and thus  yields an effective spinor type number formula.  Since we are considering optimal embeddings of the fixed CM $O_F$-order $B$ into quaternion $O_F$-orders $\calO_i$'s belonging to the same spinor genus, the theory of \emph{optimal spinor  selectivity} \cite{M.Arenas-et.al-opt-embed-trees-JNT2018}\cite[Chapter~31]{voight-quat-book}\cite{Xue-Yu-Selec-2024}  naturally arises and plays a critical role. We shall  adapt it to 
the current setting of restricted optimal embeddings.  

Throughout this section $D$ denotes an arbitrary quaternion algebra over a number field $F$. We will only specialize to the residually-unramified case in the latter half of this section, so for the moment $\scrG$ denotes an arbitrary genus of $O_F$-orders in $D$, and $\calO\in \scrG$ is a  member of $\scrG$.  We make heavy use of the adelic language. Let $\wbZ\coloneqq \prod_p\Z_p$ be the profinite completion of $\Z$. If $M$ is a finitely generated $\Z$-module or a finite dimensional $\Q$-vector space, we put $\wh{M}\coloneqq M\otimes_\Z \wbZ$. For example,  $\whD$ denotes the ring of finite adeles of $D$, and $\wcO\coloneqq \prod_\grp\calO_\grp$, where $\grp$ runs over all finite places of $F$.  As usual, $\calO_\grp$ (resp.~$M_\grp$) denotes the $\grp$-adic completion of $\calO$ (resp.~$M$). We write $\calN(\wcO)$ for the normalizer group of $\wcO$ inside $\whD^\times$.

The trace formula  \eqref{eq:trf-roe}  for restricted optimal embeddings was only  
used by K\"orner implicitly in \cite{korner:1987}, so we state it more precisely and sketch its proof. It is analogous to the classical trace formula \eqref{eq:classic-trf-opt} for optimal embeddings and takes the same form.  

\begin{prop}[Trace formula for restricted optimal embeddings]\label{prop:trf-res-opt-emb}
    Let $K/F$ be a quadratic field extension that is $F$-embeddable into $D$, and $B$ be an $O_F$-order in $K$ with a fixed point $\lambda\in B\smallsetminus O_F$. Then 
       \begin{equation}\label{eq:restricted-trace-formula}
        \sum_{[I]\in\Cl(\calO)} n((B,\lambda),\calO_l(I),\calO_l(I)^\times)=h(B)\prod_\grp n((B_\grp,\lambda),\calO_\grp,\calO_\grp^\times),
    \end{equation}
    where the product runs over all finite places $\grp$ of $F$.
\end{prop}

\begin{rem}\label{rem:finite-prod}
The product on the right hand side of \eqref{eq:restricted-trace-formula} is finite because the number $n((B_\grp,\lambda),\calO_\grp,\calO_\grp^\times)=1$ for all but finitely many $\grp$.   Indeed, for almost all $\grp$, we have $\lambda\in B_\grp^\times$ and $\calO_\grp\simeq \Mat_2(O_{F_\grp})$, so $n((B_\grp,\lambda), \calO_\grp, \calO_\grp^\times)=1$ by Remark~\ref{rem:res-opt-emb-eq-1}.
Observe that for each $\grp$ the  number $n((B_\grp, \lambda), \calO_\grp, \calO_\grp^\times)\coloneqq \abs{\REmb((B_\grp, \lambda), \calO_\grp^\times)/\calO_\grp^\times}$  depends only on the fixed genus $\scrG$ (i.e.~on the local isomorphism class of $\calO$ at each $\grp$) and not on the particular choice of $\calO\in \scrG$. Thus if $\scrG$ is clear from the context,  we simply put 
\begin{equation}
    n_\grp(B, \lambda)\coloneqq n((B_\grp, \lambda), \calO_\grp, \calO_\grp^\times).
\end{equation}
Similarly, the left hand side of \eqref{eq:restricted-trace-formula} also depends essentially only on the genus $\scrG$  in the following sense. From the surjectivity of $\xi: \Cl(\calO)\to \Tp(\calO)$ in \eqref{eq:Cl-to-Tp}, every pair of orders in $\scrG$ is \emph{linked}. In other words, given another order $\calO'\in \scrG$, there exists an invertible fractional right $\calO'$-ideal $L'$ with $\calO_l(L')=\calO$. If we write $\RIdl(\calO)$ for the set of invertible fractional right ideals of $\calO$, then from \cite[\S Lemma~I.4.9]{vigneras} the map $I\mapsto IL'$ induces a bijection $\RIdl(\calO)\to \RIdl(\calO')$ that preserves the  ideal classes and the attached left orders (i.e.~$\calO_l(I)=\calO_l(IL')$ for every $I\in \RIdl(\calO)$). Therefore, when computing a partial sum of the form $\sum_{[I]\in \grS} n((B,\lambda),\calO_l(I),\calO_l(I)^\times)$ for some subset $\grS\subseteq \Cl(\calO)$ (such as the ones in the spinor type number formula \eqref{eq:spinor-typ-prelim} with $\grS=\Cl_\sg(\calO)$), we may replace the base order $\calO$ by a suitable $\calO'\in \scrG$ and transform $\grS$ to some $\grS'\subseteq \Cl(\calO')$ accordingly without affecting the value of each term of the summation. 
\end{rem}

\begin{proof}[Proof of Proposition~\ref{prop:trf-res-opt-emb}]
    The proof of Proposition~\ref{prop:trf-res-opt-emb} follows exactly the same line of argument for the classical trace formula given in \cite[Theorem~30.4.7]{voight-quat-book} or \cite[Theorem~III.5.11]{vigneras}.  Fix an $F$-embedding $\varphi_0:K\hookrightarrow D$ and identify $K$ with its image $\varphi_0(K)\subset D$.   Consider the following sets 
\begin{align}
    \calRE((B,\lambda),\calO)\coloneqq
    &\{x\in D^\times\mid K\cap x \calO x^{-1}=B,\ \lambda\in\calN(x \calO x^{-1})\},\label{eq:global-REmb}\\
    \calRE_\grp((B,\lambda),\calO)\coloneqq
    &\{x_\grp\in D_\grp^\times\mid K_\grp\cap x_\grp \calO_\grp x_\grp^{-1}=B_\grp,\ \lambda\in\calN(x_\grp \calO_\grp x_\grp^{-1})\},\label{eq:local-REmb}\\
    \wcRE((B,\lambda),\calO)\coloneqq
    &\{\hat{x}=(x_\grp)_\grp\in \whD^\times\mid \whK\cap \hat{x} \wcO \hat{x}^{-1}=\whB,\ \lambda\in\calN(\hat{x} \wcO \hat{x}^{-1})\},\label{eq:Adelic-REmb}\\
    \wcE(B, \calO)\coloneqq &\{\hat{x}=(x_\grp)_\grp\in \whD^\times\mid \whK\cap \hat{x} \wcO \hat{x}^{-1}=\whB\},\label{eq:Adelic-Emb}
\end{align}
where $\calRE_\grp((B,\lambda),\calO)$ is defined for every finite place $\grp$ of $F$. 
By definition, the set $\calRE((B,\lambda),\calO)$ is left translation invariant by $K^\times$ and right translation invariant by the normalizer group $\calN(\calO)\subset D^\times$.
Similar to \cite[Lemma~30.3.8]{voight-quat-book}, we have a right $\calO^\times$-equivariant bijection 
\begin{equation}
   K^\times\backslash \calRE((B,\lambda),\calO) \to \REmb((B, \lambda), \calO), \qquad K^\times x\mapsto x^{-1}\varphi_0 x,
\end{equation}
from which we deduce that
\begin{equation}
    n((B, \lambda), \calO, \calO^\times)=\abs{\REmb((B, \lambda), \calO)/\calO^\times}=\abs{K^\times\backslash \calRE((B,\lambda),\calO)/\calO^\times}.
\end{equation}
By the same token, at each finite place  $\grp$ we have 
\begin{equation}\label{eq:double-coset-loc-ropt}
    n_\grp(B, \lambda)=\abs{\REmb((B_\grp, \lambda), \calO_\grp)/\calO_\grp^\times}=\abs{K_\grp^\times\backslash \calRE_\grp((B,\lambda),\calO)/\calO_\grp^\times}.
\end{equation}

In the adelic case, $\wcE\coloneqq \wcE(B, \calO)$ is left translation invariant by $\whK^\times$ and right translation invariant by the normalizer group $\calN(\wcO)$, and $\wcRE\coloneqq \wcRE((B,\lambda),\calO)$ forms a $(\whK^\times, \calN(\wcO))$-bi-invariant subset of $\wcE$.  The classical trace formula \eqref{eq:classic-trf-opt} is obtained by evaluating the cardinality of the double coset space $K^\times\backslash\wcE/\wcO^\times$ in two different ways.  We take the same approach and evaluate the cardinality of
$K^\times\backslash\wcRE/\wcO^\times$
in two different ways to obtain the trace formula for restricted optimal embeddings. 

Firstly, observe that the projection map 
\begin{equation}
 \Phi_{\wcRE}:   K^\times\backslash\wcRE/\wcO^\times\to \whK^\times\backslash\wcRE/\wcO^\times
\end{equation}
is just the restriction of the map $\Phi_{\wcE}: K^\times\backslash\wcE/\wcO^\times\to \whK^\times\backslash\wcE/\wcO^\times$ considered in \cite[(30.4.9)]{voight-quat-book}.  Since $\wcRE$ is $(\whK^\times, \wcO^\times)$-bi-invariant,  the fiber of $\Phi_{\wcRE}$ over each $ [\hat{x}]\in \whK^\times\backslash\wcRE/\wcO^\times$ coincides with  $\Phi_{\wcE}^{-1}([\hat{x}])$, which has been shown to have cardinality $h(B)$ in the proof of \cite[Theorem~30.4.7]{voight-quat-book}. On the other hand,  there is a canonical bijective map 
\[   \whK^\times\backslash \wcRE /\wcO^\times
            \xrightarrow{\thicksim} \prod_\grp K_\grp^\times\backslash \calRE_\grp((B,\lambda),\calO) /\calO_\grp^\times, \]
where the product runs over all finite places $\grp$ of $F$. Indeed, for all but finitely many $\grp$, the following conditions $\varphi_0^{-1}(\calO_\grp)=O_{K_\grp}=B_\grp$, $\lambda\in B_\grp^\times$, and $\calO_\grp\simeq \Mat_2(O_{F_\grp})$ hold simultaneously, in which case the double coset space $K_\grp^\times\backslash \calRE_\grp((B,\lambda),\calO) /\calO_\grp^\times$ is a singleton represented by $1\in D_\grp^\times$ according to Remark~\ref{rem:res-opt-emb-eq-1}. 
Combining the above discussion with \eqref{eq:double-coset-loc-ropt}, we obtain
\begin{equation} \label{eq:RHS-rtf} \abs{K^\times\backslash\wcRE/\wcO^\times}=h(B)\prod_{\grp}\abs{K_\grp^\times\backslash \calRE_\grp((B,\lambda),\calO_\grp) /\calO_\grp^\times}=h(B)\prod_{\grp}n_\grp(B, \lambda). 
\end{equation}

Secondly,  the right ideal class set $\Cl(\calO)$ can be described adelically as follows:
   \begin{equation}\label{eq:ideal-class-set}
         \Cl(\calO) \simeq D^\times\backslash \whD^\times /\wcO^\times.
     \end{equation}
Recall from Section~\ref{sec:Picent-for-tot-def-quat-alg} that we have fixed a complete set $\{I_i\subset D\mid 1\leq i\leq h\}$ of representatives of $\Cl(\calO)$ with $h\coloneqq h(\calO)$.  For each $1\leq i\leq h$, we choose $\hat{g}_i\in \whD^\times$ such that $\whI_i=\hat{g}_i\wcO$.  Then the profinite completion $\wcO_i$ of the left order $\calO_i\coloneqq \calO_l(I_i)$ of $I_i$ is given by $ \hat{g}_i\wcO \hat{g}_i^{-1}$. From \eqref{eq:ideal-class-set}, the group $\whD^\times$ admits a decomposition $\whD^\times=\bigsqcup_{i=1}^h D^\times \hat{g}_i \wcO^\times$, which induces a decomposition of the set $\wcRE$ as follows:
\begin{equation}\label{eq:decomp-REhat}
    \wcRE=\bigsqcup_{i=1}^h \wcRE_i, \qquad\text{where}\quad \wcRE_i\coloneqq \wcRE\cap (D^\times \hat{g}_i \wcO^\times).
\end{equation}
Now it is a straightforward to check  that 
\begin{equation}
    \wcRE_i \cdot \hat{g}_i^{-1}=  \calRE((B,\lambda),\calO_i)\cdot\wcO_i^\times. 
\end{equation}
We then compute 
\begin{equation}\label{eq:adelic-to-global}
 \begin{split}
     K^\times\backslash\wcRE_i/\wcO^\times &\simeq K^\times\backslash \left(\wcRE_i\cdot \hat{g}_i^{-1}\right)/\hat{g}_i \wcO^\times \hat{g}_i^{-1}\\ &=K^\times\backslash \left(\calRE((B,\lambda),\calO_i)\cdot\wcO_i^\times\right)/\wcO_i^\times \\ &\simeq K^\times\backslash \calRE((B,\lambda),\calO_i)/\calO_i^\times\simeq \REmb((B, \lambda), \calO_i)/\calO_i^\times. 
 \end{split}  
\end{equation}
Combining \eqref{eq:decomp-REhat} and \eqref{eq:adelic-to-global}, we obtain 
\begin{equation}\label{eq:LHS-rtf}
  \begin{split}
      \abs{K^\times\backslash\wcRE/\wcO^\times}&=\sum_{i=1}^h \abs{K^\times\backslash\wcRE_i/\wcO^\times}=\sum_{i=1}^h n((B, \lambda), \calO_i, \calO_i^\times)\\
      &= \sum_{[I]\in\Cl(\calO)} n((B,\lambda),\calO_l(I),\calO_l(I)^\times).
  \end{split}  
\end{equation}
We conclude the proof of the proposition by combining \eqref{eq:RHS-rtf} with \eqref{eq:LHS-rtf}. 
\end{proof}

Clearly, if $\REmb((B, \lambda), \calO)\neq \emptyset$, then necessarily, $\REmb((B_\grp, \lambda), \calO_\grp)\neq \emptyset$ for every finite place $\grp$ of $F$. 
An immediate corollary of Proposition~\ref{prop:trf-res-opt-emb} is the following converse result analogous to \cite[Corollary~30.4.18]{voight-quat-book}.

\begin{cor}\label{cor:local-imply-exists-global}
    Keep the assumption of Proposition~\ref{prop:trf-res-opt-emb}. If $\REmb((B_\grp, \lambda), \calO_\grp)\neq \emptyset$ for every finite place $\grp$ of $F$, then there exists an order $\calO_0\in \scrG$ such that $\REmb((B, \lambda), \calO_0)\neq \emptyset$.
\end{cor}

\begin{rem}
    Let $\grp$ be a finite place of $F$ coprime to the reduced discriminant $\grd(\calO)$ so that $\calO_\grp\simeq \Mat_2(O_{F_\grp})$. In Remark~\ref{rem:res-opt-emb-eq-1}, we have seen that if $\lambda\in F_\grp^\times B_\grp^\times$, then $n_\grp(B, \lambda)=1$, so particularly, $\REmb((B_\grp, \lambda), \calO_\grp)\neq \emptyset$. Conversely, we claim that the condition $\lambda\in F_\grp^\times B_\grp^\times$ is also necessary for the non-emptiness of $\REmb((B_\grp, \lambda), \calO_\grp)$ in this case. Indeed, given $\varphi_\grp \in  \REmb((B_\grp, \lambda), \calO_\grp)$, we have $\varphi_\grp(\lambda)\in \calN(\calO_\grp)=F_\grp^\times\calO_\grp^\times$ since  $\calO_\grp\simeq \Mat_2(O_{F_\grp})$. Therefore, there exists $a_\grp\in F_\grp^\times$ such that $a_\grp^{-1}\lambda\in \varphi_\grp^{-1}(\calO_\grp^\times)=B_\grp^\times$, from which the claim follows. 
\end{rem}

Recall that the main goal of this section is to develop a refinement of the trace formula \eqref{eq:restricted-trace-formula} that enables us to effectively compute the spinor type number. If $n_\grp(B, \lambda)=0$ for some $\grp$, then every term on the left hand side of \eqref{eq:restricted-trace-formula} is 0.  Without loss of generality,  we  make the following assumption for the remainder of this section (unless explicitly specified otherwise). 
\begin{assm}\label{assm:exist-local-everywhere}
    Let $K/F$ and $(B, \lambda)$ be as in Proposition~\ref{prop:trf-res-opt-emb}. We further assume that $\REmb((B_\grp, \lambda), \calO_\grp)\neq \emptyset$ for every finite place $\grp$ of $F$. 
\end{assm}

In the spirit of the Hasse–Brauer–Noether–Albert theorem \cite[Theorem~III.3.8]{vigneras}, one naturally asks whether Assumption~\ref{assm:exist-local-everywhere} implies that $\REmb((B, \lambda), \calO)\neq \emptyset$ for every $O_F$-order $\calO\in \scrG$.  If the term ``restricted optimal embedding" is replaced by ``optimal embeddings" (resp.~by ``embeddings") everywhere in the preceding question,  we immediately arrive  at the question of \emph{optimal selectivity} (resp.~\emph{selectivity}). The study of this problem is initiated by Chevalley \cite{Chevalley-matrices-1936}, and  modern day research on this topic is heavily influenced by the pioneering work of Chinburg and Friedman \cite{Chinburg-Friedman-1999} who coined the term \emph{selectivity}. Many subsequent works have been carried out since then. 
See
\cite[\S31.7.7]{voight-quat-book} for a brief historical note on the
contributions of Arenas-Carmona 
\cite{Arenas-Carmona-Spinor-CField-2003, Arenas-Carmona-2013,
  Arenas-Carmona-cyclic-orders-2012, Arenas-Carmona-Max-Sel-JNT-2012, M.Arenas-et.al-opt-embed-trees-JNT2018},  Chan-Xu \cite{Chan-Xu-Rep-Spinor-genera-2004},
Guo-Qin \cite{Guo-Qin-embedding-Eichler-JNT2004},  Linowitz
\cite{Linowitz-Selectivity-JNT2012}, Maclachlan 
\cite{Maclachlan-selectivity-JNT2008}, 
and many others.   Many of the techniques to be employed here have already been  developed by Arenas et al.~\cite{M.Arenas-et.al-opt-embed-trees-JNT2018} and  Voight \cite[\S31]{voight-quat-book}.  We shall closely follow  the exposition of \cite{Xue-Yu-Selec-2024} and adapt it to restricted optimal embeddings.  As it turns out,  in the totally definite case,  it does not make sense to expect $\REmb((B, \lambda), \calO)\neq \emptyset$ for every $\calO\in \scrG$, as demonstrated by taking $(B, \lambda)=(\Z[\sqrt{-1}], \sqrt{-1})$ in \cite[Example~2.1]{Xue-Yu-Selec-2024}. Following  \cite[Remark, p.~99]{M.Arenas-et.al-opt-embed-trees-JNT2018}, we reformulate the question in terms of \emph{spinor genus}. 

Recall from Definition~\ref{defn:spinor-genus} that the genus $\scrG$ is naturally divided into spinor genera, the spinor genus of $\calO$ is denoted by $[\calO]_\sg$, and the set of all spinor genera within $\scrG$ is denoted by $\SG(\scrG)$. Keeping the Assumption~\ref{assm:exist-local-everywhere}, for each spinor genus $[\calO]_\sg\in \SG(\scrG)$  we ask   whether there exists an order $\calO'\in [\calO]_\sg$ such that $\REmb((B, \lambda), \calO')\neq \emptyset$.  The answer to this question shall be encoded in the symbol $\Delta^{\res}((B, \lambda), \calO)$,
which mimics the \emph{optimal spinor selectivity symbol} $\Delta(B, \calO)$ in \cite[(2.1)]{Xue-Yu-Selec-2024}  recalled below
\begin{align}
  \Delta^{\res}((B,\lambda), \calO)&\coloneqq
  \begin{cases}
    1 \qquad &\text{if } \exists\, \calO'\in [\calO]_\sg\text{ such that }\REmb((B,\lambda), \calO')\neq \emptyset, \\
    0 \qquad &\text{otherwise};
  \end{cases}\\
  \Delta(B, \calO)&\coloneqq
  \begin{cases}
    1 \qquad &\text{if } \exists\, \calO'\in [\calO]_\sg\text{ such that }\Emb(B, \calO')\neq \emptyset, \\
    0 \qquad &\text{otherwise}.
  \end{cases} \label{eq:defn-Delta}
  \end{align}
As mentioned in Remark~\ref{rem:EC}, if  $D$ satisfies the Eichler Condition, then   each spinor genus consists of exactly one type by \cite[Proposition 1.1]{Brzezinski-Spinor-Class-gp-1983}.   In this case $\Delta^{\res}((B,\lambda), \calO)=1$ if and only if $\REmb((B,\lambda),  \calO)\neq \emptyset$, and the newly  formulated question in terms of spinor genus is equivalent to the old one.  Note that the symbol  $\Delta^{\res}((B,\lambda), \calO)$ is defined even if Assumption~\ref{assm:exist-local-everywhere} fails, in which case it takes value 0. Now we adapt the definition of \emph{optimal spinor selectivity} 
given by Voight in \cite[Appendix~A]{xue-yu:spinor-class-no} to define \emph{optimal spinor selectivity in the restricted sense}, which does not require Assumption~\ref{assm:exist-local-everywhere}.

\begin{defn}\label{defn:oss}
We say the genus $\scrG$ is \emph{spinor genial} for the pointed $O_F$-order $(B, \lambda)$ if $\Delta^{\res}((B,\lambda), \calO)$ takes constant value 
 as $[\calO]_\sg$ ranges within $\SG(\scrG)$ (that is,  
either $\REmb((B,\lambda), \calO)=\emptyset$ for every order $\calO\in \scrG$ or $\Delta^\res((B,\lambda), \calO)=1$ for every  $[\calO]_\sg\in \SG(\scrG)$).  If $\scrG$ is not spinor genial for $(B,\lambda)$, then we say that it is \emph{optimally spinor selective  in the restricted sense} (selective for short) for $(B, \lambda)$.  If $\scrG$ is selective for $(B, \lambda)$, then a
  spinor genus $[\calO]_\sg$ with $\Delta^{\res}((B,\lambda), \calO)=1$
  is said to be \emph{selected} by $(B,\lambda)$.
\end{defn}

Naturally, one seeks for a concrete criterion for when the genus $\scrG$ is selective for $(B, \lambda)$, and furthermore, if it is selective, which spinor genera are selected by $(B, \lambda)$. In general, it is quite technical to detect whether an arbitrary genus $\scrG$ is optimally spinor selective for a given order $B$; see \cite[Theorem~2.6]{peng-xue:select} for a glimpse of its complexity in general. Nevertheless, the situation is much simpler if we assume that  $\calO$ is residually unramified.
Indeed,  it will be shown very soon in Proposition~\ref{prop:ru-equal-sel-symb} that for a genus $\scrG_{\ru}$ of residually unramified orders (see Definition~\ref{defn:res-unr-order}), the equality $\Delta^{\res}((B, \lambda), \calO)=\Delta(B, \calO)$ holds true for every $[\calO]_\sg\in \SG(\scrG_{ru})$ under Assumption~\ref{assm:exist-local-everywhere}.   Therefore, to state a theorem of optimal spinor selectivity in the restricted sense for $\scrG_\ru$,  it is enough to take one for optimal spinor selectivity such as \cite[Theorem~2.5]{xue-yu:spinor-class-no},  replace 
$\Delta(B, \calO)$ by $\Delta^{\res}((B, \lambda), \calO)$ and modify everything else accordingly.  Nevertheless, we still need to introduce quite some notation to prove the desired equality $\Delta^{\res}((B, \lambda), \calO)=\Delta(B, \calO)$. 

For the moment, let us keep the genus $\scrG$ arbitrary and pick $\calO\in \scrG$. 
 Following \cite[\S
 III.4]{vigneras}, we write $F_D^\times$ for the subgroup of
 $F^\times$ consisting of all elements 
that are positive at every infinite place of $F$ ramified in $D$.  The Hasse-Schilling-Maass theorem \cite[Theorem~33.15]{reiner:mo}
\cite[Theorem~III.4.1]{vigneras} implies that $\Nr(D^\times)=F_D^\times$.  
 The pointed  set $\SG(\calO)\coloneqq (\SG(\scrG), [\calO]_\sg)$ of spinor genera in $\scrG$ admits the following adelic description
(cf.~\cite[Propositions~1.2 and 1.8]{Brzezinski-Spinor-Class-gp-1983})
\begin{equation}
  \label{eq:118}
\SG(\calO)\simeq  (D^\times\whD^1)\bsh \whD^\times/\calN(\wcO)\xrightarrow[\simeq]{\Nr}
  F_D^\times\bsh \whF^\times/\Nr(\calN(\wcO)),  
\end{equation}
where the  two double coset spaces are canonically bijective via
the reduced norm map. It follows that $\SG(\calO)$ is naturally
equipped with an abelian group structure, with its distinguished point
$[\calO]_\sg$ as the
identity element.  Since $\Nr(\calN(\wcO))$ is an open subgroup of
$\whF^\times$ containing $(\whF^\times)^2$, the group $\SG(\calO)$ is a
finite elementary $2$-group
\cite[Proposition~3.5]{Linowitz-Selectivity-JNT2012}. Clearly, the
group $\Nr(\calN(\wcO))$ depends only on the genus $\scrG$ and not on
the choice of $\calO$. 

\begin{defn}[{\cite[\S2]{Arenas-Carmona-Spinor-CField-2003},
    \cite[\S3]{Linowitz-Selectivity-JNT2012}}] \label{defn:spinor-field}
  The \emph{spinor genus field} of $\scrG$ is the abelian field extension
  $\Sigma/F$ corresponding to the open subgroup
  $F_D^\times\Nr(\calN(\wcO))\subseteq \whF^\times$ via the class
  field theory \cite[Theorem~X.5]{Lang-ANT}. Given an $O_F$-order $B$ in a quadratic field extension $K/F$, 
  we introduce the symbol
\begin{equation}\label{eq:defn-sBO}
        s(B,\calO)\coloneqq
    \begin{cases}
        1 & \text{if}\ K\subseteq\Sigma,\\
        0 & \text{otherwise},
    \end{cases}
\end{equation}
which depends only on the  field $K$ and the genus $\scrG$. 
\end{defn}
By the definition of $\Sigma$, there are isomorphisms:
\begin{equation}
  \label{eq:8}
\SG(\calO)\simeq  F_D^\times\bsh \whF^\times/
  \Nr(\calN(\wcO))\simeq   \Gal(\Sigma/F). 
\end{equation}
Given another order $\calO'\in \scrG$, we define $\rho(\calO, \calO')$
to be the image of $[\calO']_\sg\in \SG(\calO)$ in
$\Gal(\Sigma/F)$ under the above isomorphism.  More canonically, we regard the base
point free set $\SG(\scrG)$ as a principal homogeneous space over
$\Gal(\Sigma/F)$ via (\ref{eq:8}). Then $\rho(\calO, \calO')$ is
the unique element of $\Gal(\Sigma/F)$ that sends $[\calO]_\sg$
to $[\calO']_\sg$. By definition, $\rho(\calO, \calO')$ depends only on the spinor genera of $\calO$ and $\calO'$. The map  $(\calO, \calO')\mapsto \rho(\calO, \calO')$ enjoys the following
properties:
\begin{enumerate}[label=(\alph*)]
\item $\rho(\calO, \calO')=1$ if and only if $\calO\sim \calO'$,
\item $\rho(\calO, \calO')=\rho(\calO', \calO)$,
\item   $\rho(\calO, \calO'')=\rho(\calO, \calO')\rho(\calO', \calO'')$
\end{enumerate}
for all $\calO, \calO', \calO''\in \scrG$. 

Next, we introduce two more class fields closely related to (restricted) optimal embeddings. Keep Assumption~\ref{assm:exist-local-everywhere}.  In light of Corollary~\ref{cor:local-imply-exists-global}, there exists $\calO_0\in \scrG$ such that $\REmb((B, \lambda), \calO_0)\neq \emptyset$. We take $\calO=\calO_0$ and $\varphi_0\in \REmb((B, \lambda), \calO_0)$ in the definition \eqref{eq:Adelic-REmb} (resp.~\eqref{eq:Adelic-Emb}) and put 
\begin{equation}\label{eq:new-Adelic-REmb-Emb}
    \wcRE\coloneqq \wcRE((B,\lambda),\calO_0),\qquad \wcE\coloneqq \wcE(B, \calO_0). 
\end{equation}
 By construction,  we have 
    $\whK^\times\calN(\wcO)\subseteq \wcRE\subseteq \wcE$, 
where $K$ is identified with its image $\varphi_0(K)\subset D$.  Taking the reduced norms produces the following chain of inclusions: 
\begin{equation}\label{eq:inclusions-of-subgp}
    F_K^\times\Nr(\whK^\times) \subseteq F_D^\times\Nr(\whK^\times)\Nr(\calN(\wcO)) \subseteq F_D^\times\Nr(\wcRE) \subseteq F_D^\times\Nr(\wcE) \subseteq \whF^\times,
\end{equation}
where $F_K^\times$ denotes the subgroup of $F^\times$ consisting of the elements that are positive at each infinite place of $F$ that is ramified in $K/F$. Moreover, $F_K^\times\subset F_D^\times$ since $K$ is $F$-embeddable into $D$ by our assumption.  In \cite[\S31.3.14]{voight-quat-book}, it is shown that $F_D^\times \Nr(\wcE)$ is an open subgroup of $\whF^\times$ of index at most $2$ independent of the choice of the pair $(\calO_0, \varphi_0)$ (as long as they satisfy the condition $\varphi_0\in \REmb((B,\lambda), \calO_0)\neq\emptyset$).  By the same token, $F_D^\times\Nr(\wcRE)$ enjoys the same properties. 

\begin{defn}[{ \cite[\S2]{M.Arenas-et.al-opt-embed-trees-JNT2018},\cite[\S3]{Arenas-Carmona-Spinor-CField-2003}}]
We define the \emph{optimal representation field} $E_\op/F$ and the \emph{restricted optimal representation field} $E_\op^{\res}$ as the class fields corresponding to the open subgroups $F_D^\times\Nr(\wcE)$ and $F_D^\times\Nr(\wcRE)$ of $\whF^\times$ respectively. 
\end{defn}

Note that the first  two terms of \eqref{eq:inclusions-of-subgp} correspond via the class field theory to $K$ and $K\cap \Sigma$ respectively. Hence there is a chain of class fields of $F$ as follows 
\begin{equation}\label{eq:inclusions-of-field-extn}
    K \supseteq K\cap\Sigma \supseteq E_\op^{\res} \supseteq E_\op \supseteq F.
\end{equation}
The first part of the following lemma is adapted directly from \cite[Lemma~2.2]{M.Arenas-et.al-opt-embed-trees-JNT2018} (see also \cite[Theorem~3.2]{Maclachlan-selectivity-JNT2008} and \cite[Proposition~31.4.4]{voight-quat-book}), and the second part can be proved in exactly the same way, so we omit their proofs. 

\begin{lem}\label{lem:sel-symbol-rep-field}
Keep the order $\calO_0\in \scrG$ with $\REmb((B, \lambda), \calO_0)\neq \emptyset$ fixed, and let $\calO$ be an arbitrary order in  $\scrG$.  \begin{enumerate}[label=(\arabic*), leftmargin=*]
    \item  $\Delta(B,\calO)=1$ if and only if $\rho(\calO,\calO_0)\in\Gal(\Sigma/F)$ restricts to identity on $E_\op$.
    \item $\Delta^{\res}((B,\lambda), \calO)=1$ if and only if $\rho(\calO,\calO_0)\in\Gal(\Sigma/F)$ restricts to identity on $E_\op^{\res}$.
\end{enumerate}
\end{lem}

In \cite[Proposition~31.5.7]{voight-quat-book}, Voight proves that if $\scrG$ is a genus of Eichler orders, then the inclusion $\whK^\times\calN(\wcO)\subseteq \wcE$ induces an equality $\Nr(\whK^\times)\Nr(\calN(\wcO)) = \Nr(\wcE)$. This result has since been 
extended to any genus of residually unramified orders by Chia-Fu Yu and the second named author in \cite[Lemma~2.14]{Xue-Yu-Selec-2024}. From the chain of inclusions in \eqref{eq:inclusions-of-subgp} (correspondingly, \eqref{eq:inclusions-of-field-extn}) 
and Lemma~\ref{lem:sel-symbol-rep-field}, we immediately get the following proposition. 

\begin{prop}\label{prop:ru-equal-sel-symb}
  Keep  Assumption~\ref{assm:exist-local-everywhere} and suppose further that $\scrG=\scrG_\ru$ is a genus of residually unramified orders in $D$. Then \begin{equation*}
      \Nr(\whK^\times)\Nr(\calN(\wcO)) =\Nr(\wcRE)= \Nr(\wcE),\quad \text{and hence} \quad  
  K\cap \Sigma=E_\op^{\res}=E_\op. 
  \end{equation*}
In particular, 
      $\Delta^{\res}((B, \lambda), \calO)=\Delta(B, \calO)$ for every  $\calO\in \scrG_\ru$. 
\end{prop}

For the reader's convenience, we reproduce the main theorem of \emph{optimal spinor selectivity for residually unramified orders} \cite[Theorem~2.5]{xue-yu:spinor-class-no} (cf.~\cite[Theorem~2.15 and Lemma~2.17]{Xue-Yu-Selec-2024})  and reformulate it in terms of restricted optimal embeddings.

\begin{thm}\label{thm:spinor-selectivity}
Keep the assumption of Proposition~\ref{prop:ru-equal-sel-symb}. Then
\begin{enumerate}[label=(\arabic*), leftmargin=*]
\item  the genus $\scrG_\ru$ is  optimally spinor selective in the restricted sense for the pointed order $(B, \lambda)$ if and only if $K\subset \Sigma$ (i.e.~$s(B, \calO)=1$ from \eqref{eq:defn-sBO}), or equivalently, if and only if
both of the following conditions hold:
 \begin{enumerate}[label=(\alph*)]
  \item both $K$ and  $D$ are unramified
  at every finite place $\grp$ of $F$  and ramify at exactly the same (possibly empty) set of infinite places;
  \item $K/F$ splits at every finite place $\grp$ of $F$
    satisfying $\ord_\grp(\grd(\calO))\equiv
  1\pmod{2}$, where $\grd(\calO)$ denotes the reduced discriminant of $\calO$.  
\end{enumerate}
\item If $\scrG_\ru$ is  selective for $(B,\lambda)$, then
\begin{enumerate}[label=(\roman*)]
\item  for any two orders   $\calO, \calO'\in \scrG_\ru$,  
   \begin{equation}
\label{eq:164}
    \Delta^{\res}((B,\lambda), \calO)=\rho(\calO, \calO')|_K
+\Delta^{\res}((B,\lambda), \calO'), 
  \end{equation}
  where $\rho(\calO, \calO')|_K$ is the restriction of $\rho(\calO,
  \calO')\in \Gal(\Sigma/F)$ to $K$,  and the
  summation on the right is taken inside $\zmod{2}$ with the canonical
  identification $\Gal(K/F)\simeq \zmod{2}$;

\item $\Delta^{\res}((B,\lambda), \calO)=1$ for   exactly half of the spinor genera
$[\calO]_\sg\in \SG(\scrG_\ru)$.
\end{enumerate}
\end{enumerate} 
 \end{thm}

We return to the general setting where $\scrG$ is an arbitrary genus of $O_F$-orders in $D$ and $\calO$ is a member of $\scrG$, as there is another ingredient for the refined trace formula for restricted optimal embeddings  yet to be introduced.  Just like the genus $\scrG$ is partitioned into a disjoint union of spinor genera, the set $\RIdl(\calO)$ of invertible fractional right $\calO$-ideals can similarly be partitioned into a disjoint union of \emph{spinor classes}.  Recall that  $\whD^1\coloneqq \ker(\Nr: \whD^\times\to \whF^\times)$ denotes the reduced norm one subgroup of $\whD^\times$.

\begin{defn}\label{defn:spinor-class}
\begin{enumerate}[label=(\arabic*), leftmargin=*]
\item Two invertible fractional right $\calO$-ideals $I$ and $I'$ are said to belong to the \emph{same spinor class} if there exists $x\in D^\times\whD^1$ such that $\whI'=x\whI$. The spinor class of $I$ is the set $[I]_\scc$ consisting of all $I'\in \RIdl(\calO)$ in the same spinor class as $I$.

\item  The set of all spinor classes of invertible fractional right $\calO$-ideals is denoted by $\SCl(\calO)$ and regarded as a pointed set with base point $[\calO]_\scc$.

\item Given a fixed spinor class $[I]_\scc\in \SCl(\calO)$, we denote the set
  of ideal classes in  $[I]_\scc$ by $\Cl(\calO,
  [I]_\scc)$. In other words, 
  \begin{equation}
    \label{eq:125}
    \Cl(\calO,
  [I]_\scc):=\{[I']\in \Cl(\calO)\mid [I']\subseteq [I]_\scc\}. 
  \end{equation}
  For simplicity, we put $\Cl_\scc(\calO):=\Cl(\calO, [\calO]_\scc)$.
    \end{enumerate}
\end{defn}

By definition,  the ideal class set $\Cl(\calO)$ decomposes into a disjoint union: 
\begin{equation}\label{eq:decomp-Cl}
    \Cl(\calO)=\bigsqcup_{[I]_\scc\in\SCl(\calO)} \Cl(\calO,[I]_\scc).
\end{equation}
From Remark~\ref{rem:finite-prod}, the left order preserving map  $I'\mapsto I'I^{-1}$ induces a bijection
\begin{equation}\label{eq:spinor-cl-set-bij}
    \Cl(\calO, [I]_\scc)\to \Cl_\scc(\calO_l(I)), \qquad [I']\mapsto [I'I^{-1}]. 
\end{equation}
If we write $\whI=x\wcO$ for some $x\in \whD^\times$ and put $[\whI]_\scc\coloneqq D^\times \whD^1 x\wcO^\times$, then  $\Cl(\calO, [I]_\scc)$ can be described adelically as 
\begin{equation}\label{eq:adelic-des-spinor-cl}
\Cl(\calO,[I]_\scc) \simeq D^\times\backslash \left(D^\times \whD^1 x\wcO^\times\right)/\wcO^\times =D^\times\backslash [\whI]_\scc/\wcO^\times,
\end{equation}
where $[\whI]_\scc$ is a $(D^\times, \wcO^\times)$-bi-invariant subset of $\whD^\times$ depending only on the spinor class $[I]_\scc$. 

Similar to \eqref{eq:118}, the pointed set $\SCl(\calO)$ can be described adelically as follows
\begin{equation}
  \label{eq:14}
  \SCl(\calO)\simeq (D^\times\whD^1)\bsh \whD^\times/\wcO^\times\xrightarrow[\simeq]{\Nr}
  F_D^\times\bsh \whF^\times/\Nr(\wcO^\times),   
\end{equation}
where the  two double coset spaces are canonically bijective via
the reduced norm map. This equips $\SCl(\calO)$ with a finite abelian group structure with $[\calO]_\scc$ as its
identity element,   so we
call it  the \emph{spinor class group} of $\calO$.
In light of the
adelic description of $\SG(\calO)$, there is a canonical surjective
group homomorphism
\begin{equation}
  \label{eq:58}
\Xi:   \SCl(\calO)\twoheadrightarrow \SG(\calO),\qquad
  [I]_\scc\mapsto [\calO_l(I)]_\sg. 
\end{equation}
In particular,  if $I$ and $I'$ belong to the same spinor class, then their
left orders $\calO_l(I)$ and $\calO_l(I')$ belong to  the same spinor
genus.    

\begin{rem}\label{rem:order-SGO}
Suppose that  $\calO$ is residually unramified. It has been shown  in the proof of \cite[Lemma~2.17]{Xue-Yu-Selec-2024} that $\Nr(\wcO^\times)=\whO_F^\times$.  It follows  from \eqref{eq:14} that $\abs{\SCl(\calO)}$ coincides with the \emph{restricted class number $h_D(F)$ of $F$ with respect to $D$} defined in  \cite[Corollary~III.5.7]{vigneras} and reproduced below
\begin{equation}
    h_D(F)\coloneqq \abs{F_D^\times\bsh \whF^\times/\whO_F^\times}.
\end{equation}
If  $D$ is further assumed to be totally definite, then $F_D^\times$ coincides with the subgroup $F_+^\times$ of totally positive elements of $F^\times$. This yields a canonical isomorphism between $\SCl(\calO)$ and the narrow ideal class group $\Cl^+(O_F)$ as follows
\begin{equation}\label{eq:SCl-narrow-Cl}
\SCl(\calO)\cong \Cl^+(O_F), \qquad [I]_\scc\mapsto [\Nr(I)]_+, 
\end{equation}
 hence $\abs{\SCl(\calO)}=h^+(F)$ in this case. Since $\calN(\wcO)\supset \whF^\times\wcO^\times$, the adelic description of the spinor genus group $\SG(\calO)$ in $\eqref{eq:118}$ implies that $\SG(\calO)$ can be identified with a quotient of the \emph{Gauss genus group} $\Cl^+(O_F)/\Cl^+(O_F)^2$, where $\Cl^+(O_F)^2$ denotes the subgroup of $\Cl^+(O_F)$ consisting of the perfect square narrow classes. Moreover, we have $\SG(\calO)\cong \Cl^+(O_F)/\Cl^+(O_F)^2$ if $\Nr(\calN(\calO))=\whF^{\times2}\whO_F^\times$. For example, this condition holds if $\calO$ is a maximal $O_F$-order in a totally definite quaternion $F$-algebra that is unramified at all the finite places of $F$, in which case $\calN(\wcO)=\whF^\times\wcO^\times$.
\end{rem}

 Finally, we are ready to state the spinor trace formula for restricted optimal embeddings, which is a generalization of \cite[Proposition~4.3]{Xue-Yu-Selec-2024}
and admits a similar proof. 

\begin{prop}[Spinor trace formula for restricted optimal embeddings]\label{prop:spinor-trf}
 Keep the assumption of Proposition~\ref{prop:trf-res-opt-emb} and assume  that  $\calO$ is a residually unramified $O_F$-order in $D$. Then for each $[J]_\scc\in\SCl(\calO)$, we have
    \begin{equation}\label{eq:spinor-restricted-trace-formula-2}
        \sum_{[I]\in\Cl(\calO,[J]_\scc)} n((B,\lambda),\calO_l(I),\calO_l(I)^\times)=\frac{2^{s(B,\calO)}\Delta(B,\calO_l(J))h(B)}{h_D(F)}\prod_\grp n_\grp(B,\lambda),
    \end{equation}
    where $h_D(F)=\abs{\SCl(\calO)}$ and
    the product runs over all finite places $\grp$ of $F$.
    If $D$ is further assumed to be ramified at some finite place of $F$, then
    \begin{equation}\label{eq:sp-trf-ram-case}
        \sum_{[I]\in\Cl(\calO,[J]_\scc)} n((B,\lambda),\calO_l(I),\calO_l(I)^\times)=\frac{h(B)}{h_D(F)}\prod_\grp n_\grp(B,\lambda).
    \end{equation}
\end{prop}

Note that \eqref{eq:spinor-restricted-trace-formula-2} holds even without Assumption~\ref{assm:exist-local-everywhere}, in which case both sides of the equation are zero.  On the other hand, if Assumption~\ref{assm:exist-local-everywhere} holds so that $\prod_\grp n_\grp(B,\lambda)\neq 0$, then  
$\Delta(B,\calO_l(J))=\Delta^{\res}((B,\lambda),\calO_l(J))$ by Proposition~\ref{prop:ru-equal-sel-symb}. This is precisely the reason why we have $\Delta(B,\calO_l(J))$ on the right hand side of \eqref{eq:spinor-restricted-trace-formula-2} rather than the more cumbersome $\Delta^{\res}((B,\lambda),\calO_l(J))$. 

\begin{proof}[Sketch of proof of Proposition~\ref{prop:spinor-trf}]
  Without loss of generality, suppose that Assumption~\ref{assm:exist-local-everywhere} holds so that there exists $\calO_0\in \scrG(\calO)$ with $\REmb((B, \lambda), \calO_0)\neq \emptyset$. In light of Remark~\ref{rem:finite-prod}, we may replace $\calO$ by $\calO_0$ and fix $\varphi_0\in \REmb((B, \lambda), \calO)$.    Let $\wcRE\coloneqq \wcRE((B, \lambda), \calO)$ be the subset of $\whD^\times$ defined in \eqref{eq:Adelic-REmb} and  \eqref{eq:new-Adelic-REmb-Emb}. 
    From \eqref{eq:adelic-to-global}, we have 
  \begin{equation}
        \sum_{[I]\in\Cl(\calO,[J]_\scc)} n((B,\lambda),\calO_l(I),\calO_l(I)^\times)=\abs{K^\times\backslash (\wcRE\cap[\widehat{J}]_\scc) /\wcO^\times}, 
    \end{equation} 
    where $[\widehat{J}]_\scc$ is the subset of $\whD^\times$ in the adelic description of $\Cl(\calO,[J]_\scc)$ in \eqref{eq:adelic-des-spinor-cl}.
    In Proposition~\ref{prop:ru-equal-sel-symb}, we have seen that the residually unramified assumption on $\calO$ implies the equality $\Nr(\whK^\times)\Nr(\calN(\wcO)) =\Nr(\wcRE)$. Now the rest of the proof of \eqref{eq:spinor-restricted-trace-formula-2} can be carried over almost verbatim from that of \cite[Proposition~4.3]{Xue-Yu-Selec-2024}, replacing $\wcE$ there by $\wcRE$ whenever applicable (and keeping in mind that $\abs{\SCl(\calO)}=h_D(F)$ in this case). We omit the details.  

    Lastly,  if $D$ is ramified at some finite place of $F$, then $s(B, \calO)=0$,  and the genus $\scrG(\calO)$ is spinor genial for $(B, \lambda)$ by Theorem~\ref{thm:spinor-selectivity}, from which it follows that the value of $\Delta(B,\calO_l(J))\prod_\grp n_\grp(B,\lambda)$ coincides with that of $\prod_\grp n_\grp(B,\lambda)$ in this case. Thus, \eqref{eq:sp-trf-ram-case} is a special form of \eqref{eq:spinor-restricted-trace-formula-2}.  
\end{proof}

If $D$ is
   assumed to satisfy the Eichler condition, then
  $\Cl(\calO, [J]_\scc)$ is a singleton with the unique member $[J]$
  by \cite[Proposition~1.1]{Brzezinski-Spinor-Class-gp-1983},
  so \eqref{eq:spinor-restricted-trace-formula-2} reduces further to produce a generalization of the Vign\'eras-Voight
  formula  \cite[Proposition~4.1]{Xue-Yu-Selec-2024}.

\begin{cor}
    Keep the assumption of Proposition~\ref{prop:trf-res-opt-emb} and assume further that $D$ satisfies the Eichler condition. Then for every residually unramified $O_F$-order $\calO$ in $D$, we have 
    \begin{equation}
n((B,\lambda),\calO,\calO^\times)=\frac{2^{s(B,\calO)}\Delta(B,\calO)h(B)}{h_D(F)}\prod_\grp n_\grp(B,\lambda).
    \end{equation}
\end{cor}

\section{\texorpdfstring{The spinor type number formula for $t_\sg(\calO)$ and divisibility}{The spinor type number formula for tsg(O) and divisibility}}\label{sec:spinor-type-number-formula}

In this section, we combine the preliminary version of the spinor type number formula in Proposition~\ref{prop:spinor-typ-prelim} with the spinor trace formula for restricted optimal embeddings in Proposition~\ref{prop:spinor-trf} to produce an effective spinor type number formula for a residually unramified order $\calO$ in a totally definite quaternion $F$-algebras $D$.  We then apply it to obtain the divisibility result of the type number $t(\calO)$ in Theorem~\ref{thm:type-num-div}.  The divisibility of  the trace of the Brandt matrix can be proved in the same way.  Throughout this section,  $D$ denotes a totally definite quaternion algebra over a totally real field $F$, and $\calO$ is assumed to be a residually unramified $O_F$-order in $D$ (unless specified otherwise).

In the preliminary spinor type number formula \eqref{eq:spinor-typ-prelim}, there are two terms that are not effectively computable so far. The first is $\Mass_\sg(\calO)$ defined in \eqref{eq:defn-mass-sg}, and the second is the partial sum $\sum_{[I]\in \Cl_\sg(\calO)} n((B, \lambda), \calO_l(I), \calO_l(I)^\times)$ for each pointed CM $O_F$-order $(B, \lambda)\in \scrB$, where $\Cl_\sg(\calO)$ is the subset of $\Cl(\calO)$ defined in \eqref{eq:defn-cl-sg}. Comparably, the left hand side of the spinor trace formula \eqref{eq:spinor-restricted-trace-formula-2} is a similar sum over the  set $\Cl(\calO, [J]_\scc)$ 
for a spinor class $[J]_\scc\in \SCl(\calO)$. To bridge this discrepancy, note that the canonical surjective group homomorphism  $\Xi: \SCl(\calO)\to \SG(\calO)$ given in \eqref{eq:58} fits into a commutative diagram 
\begin{equation}
    \begin{tikzcd} 
\Cl(\calO)\ar[r, "\xi", twoheadrightarrow] \ar[d, "\Psi_\cls"', twoheadrightarrow]& \Tp(\calO)\ar[d, "\Psi_\tp", twoheadrightarrow]\\
\SCl(\calO) \ar[r, "\Xi", twoheadrightarrow] & \SG(\calO)
\end{tikzcd}
\end{equation}
where the top horizontal map is defined in \eqref{eq:Cl-to-Tp}, and the vertical maps are canonical projections.  By definition, $\Tp_\sg(\calO)\coloneqq \Psi_\tp^{-1}([\calO]_\sg)$, so 
$\Cl_\sg(\calO)\coloneqq \xi^{-1}(\Tp_\sg(\calO))$ is the neutral fiber of the composition map $\Psi_\tp\circ \xi: \Cl(\calO)\to \SG(\calO)$. Therefore, if we write the neutral fiber of $\Xi$ as 
\begin{equation}
    \Xi^{-1}([\calO]_\sg)=\{[J_1]_\scc, \cdots, [J_r]_\scc\}\subseteq \SCl(\calO)\quad \text{with}\quad r\coloneqq \frac{\abs{\SCl(\calO)}}{\abs{\SG(\calO)}},
\end{equation}
then $\Cl_\sg(\calO)$ decomposes into a disjoint union 
    \begin{equation}\label{eq:Picent-stable-subset}
    \Cl_\sg(\calO)\coloneqq\bigsqcup_{j=1}^{r} \Cl(\calO,[J_j]_\scc).
\end{equation}

With some extra input from \cite[\S3]{xue-yu:spinor-class-no} for the mass formula, we are now ready to write down the final version of the spinor type number formula.

\begin{thm}\label{thm:spinor-type-number-formula}
Keep the notation and assumption of Theorem~\ref{thm:Korner}.  Let $\calO$ be a residually unramified $O_F$-order in $D$, and $\omega(\calO)$ be the number of distinct prime divisors of the reduced discriminant $\grd(\calO)$. Put $C(\calO)\coloneqq 2^{\omega(\calO)}h(F)\abs{\SG(\calO)}$ for brevity. 
\begin{enumerate}[label=(\roman*), leftmargin=*]
    \item The  spinor type number $t_\sg(\calO)$  is given by the following formula
    \begin{equation}\label{eq:spinor-typ-final}
        t_\sg(\calO)=\frac{1}{C(\calO)} \Big( \Mass(\calO)+\frac{1}{2} \sum_{(B,\lambda)\in\scrB} \frac{2^{s(B,\calO)}\Delta(B,\calO)h(B)}{w(B)} \prod_{\grp\mid \grd(\calO)} n_\grp(B,\lambda) \Big).
    \end{equation}
\item     If $D$ is further assumed to be ramified at some finite place of $F$, then we have 
    \begin{equation}\label{eq:spinor-typ-ram-final}
        t_\sg(\calO)=\frac{1}{C(\calO)} \Big( \Mass(\calO)+\frac{1}{2} \sum_{(B,\lambda)\in\scrB} \frac{h(B)}{w(B)} \prod_{\grp\mid \grd(\calO)} n_\grp(B,\lambda) \Big).
    \end{equation}
    In this case every spinor genus $[\calO']_\sg\in \SG(\calO)$ contains the same number of types of $O_F$-orders, so $t(\calO)=\abs{\SG(\calO)}\cdot t_\sg(\calO)$.
\end{enumerate}
\end{thm}
\begin{proof}
  (i)  From the definition \eqref{eq:defn-mass-sg} of $\Mass_\sg(\calO)$ and 
    the decomposition \eqref{eq:Picent-stable-subset} of $\Cl_\sg(\calO)$, we have 
    \begin{align}
        \Mass_{\sg}(\calO)&=\sum_{j=1}^r \Mass_\scc(\calO, [J_j]_\scc), \qquad \text{where}\\
         \Mass_\scc(\calO, [J]_\scc)&\coloneqq\sum\limits_{[I]\in\Cl(\calO,[J]_\scc)} \frac{1}{[\calO_l(I)^\times:O_F^\times]}, \qquad \forall\, [J]_\scc\in \SCl(\calO).\notag
    \end{align}
    It is shown in \cite[Lemma~3.1]{xue-yu:spinor-class-no} that the value of $\Mass_\scc(\calO, [J]_\scc)$  is independent of the choice of $[J]_\scc\in \SCl(\calO)$, and hence 
    \begin{equation}\label{eq:same-Mass}
    \Mass_\scc(\calO,[J]_\scc)=\frac{\Mass(\calO)}{\abs{\SCl(\calO)}}, \quad \forall\ [J]_\scc\in\SCl(\calO).
\end{equation}
This implies that 
\begin{equation}\label{eq:mass-sg-value}
    \Mass_{\sg}(\calO)=r\cdot \Mass(\calO,[\calO]_\scc)=\frac{\abs{\SCl(\calO)}}{\abs{\SG(\calO)}}\cdot \frac{\Mass(\calO)}{\abs{\SCl(\calO)}}=\frac{\Mass(\calO)}{\abs{\SG(\calO)}}.
\end{equation}

Similarly, for each fixed pointed CM $O_F$-order $(B, \lambda)\in \scrB$, the optimal spinor selectivity symbol $\Delta(B, \calO_l(J_j))$ takes the same value as $\Delta(B, \calO)$ since $\calO_l(J_j)$ belongs to the same spinor genus as $\calO$ for every $1\leq j\leq r$. Combining the decomposition  of $\Cl_\sg(\calO)$ in \eqref{eq:Picent-stable-subset} with the spinor trace formula \eqref{eq:spinor-restricted-trace-formula-2}, we obtain
\begin{equation}\label{eq:spinor-tr-sg-value}
    \begin{split}
        &\sum_{[I]\in \Cl_\sg(\calO)} n((B, \lambda), \calO_l(I), \calO_l(I)^\times)\\=&\sum_{j=1}^r\; \sum_{[I]\in \Cl_\scc(\calO, [J_j]_\scc)} n((B, \lambda), \calO_l(I), \calO_l(I)^\times)\\
        =&\sum_{j=1}^r\; \frac{2^{s(B,\calO)}\Delta(B,\calO_l(J_j))h(B)}{\abs{\SCl(\calO)}}\prod_\grp n_\grp(B,\lambda)\\
         =&\frac{\abs{\SCl(\calO)}}{\abs{\SG(\calO)}}\cdot \frac{2^{s(B,\calO)}\Delta(B,\calO)h(B)}{\abs{\SCl(\calO)}}\prod_\grp n_\grp(B,\lambda)\\
         =&\frac{2^{s(B,\calO)}\Delta(B,\calO)h(B)}{\abs{\SG(\calO)}}\prod_\grp n_\grp(B,\lambda)=\frac{2^{s(B,\calO)}\Delta(B,\calO)h(B)}{\abs{\SG(\calO)}}\prod_{\grp\mid \grd(\calO)} n_\grp(B,\lambda),
    \end{split}
\end{equation}
where we have applied Lemma~\ref{lem:n1-pnmid-d} in the last step. 

Recall from  Lemma~\ref{lem:res-unr-card-of-Picent} that $\abs{\Picent(\calO)}=2^{\omega(\calO)}h(F)$ for the residually unramified order $\calO$. 
Plugging both \eqref{eq:mass-sg-value} and \eqref{eq:spinor-tr-sg-value} into the preliminary spinor type number formula \eqref{eq:spinor-typ-prelim},   we  obtain the desired spinor type number formula in \eqref{eq:spinor-typ-final}.   

(ii) Suppose further that $D$ is ramified at some finite place of $F$.  The special form of the spinor type number formula in \eqref{eq:spinor-typ-ram-final} is  obtained by applying \eqref{eq:sp-trf-ram-case} instead of  \eqref{eq:spinor-restricted-trace-formula-2} in the above calculation. Observe that the right hand side of \eqref{eq:spinor-typ-ram-final} depends only on the genus of $\calO$. In particular,  the spinor type number $t_\sg(\calO')\coloneqq \abs{\Tp_\sg(\calO')}$ remains invariant as $[\calO']_\sg$ ranges in $\SG(\calO)$.  
It follows from \eqref{eq:typ-refine} that $t(\calO)=\abs{\SG(\calO)}\cdot t_\sg(\calO)$ in this case. 
\end{proof}

Part (ii) of Theorem~\ref{thm:spinor-type-number-formula} already implies that if $D$ is assumed to be ramified at some finite place of $F$, then $t(\calO)$ is divisible by $\abs{\SG(\calO)}$. 
However, as shown by the following example,  such a  divisibility result is too good to be true
as soon as the aforementioned ramification assumption on $D$ is dropped.

\begin{ex}\label{ex:indivisible}
  Let $F=\bbQ(\sqrt{79})$ and $D=D_{\infty_1,\infty_2}$ be the unique quaternion $F$-algebra ramified exactly at the two infinite places of $F$. We have $h(F)=3$ and $h^+(F)=6$.  According to \cite[Table 1, p.~676]{xue-yang-yu:ECNF}, every maximal $O_F$-order $\calO$ in $D$ has class number $h(\calO)=69$. Using the class-type number relation in \cite[Corollary 3.5]{xue-yu:type_no}, we get $t(\calO)=h(\calO)/h(F)=23$. On the other hand, we have seen in Remark~\ref{rem:order-SGO} that in this case $\SG(\calO)$ is isomorphic to the Gauss genus group $\Cl^+(O_F)/\Cl^+(O_F)^2$, which has order $2$ since $h^+(F)=6$. Alternatively, the order of $\Cl^+(O_F)/\Cl^+(O_F)^2$ can also be deduced from the renowned Gauss genus theorem \cite[Proposition~3.11]{Cox-Primes}.  This  provides an example where the order $\abs{\SG(\calO)}=2$ does not divide the type number $t(\calO)=23$.
\end{ex}

Nevertheless, a slightly weaker divisibility result still holds true even if $D$ is unramified at all the finite places of $F$.  Recall from \eqref{eq:118} that the abelian group structure on $\SG(\calO)$ is defined via the following bijection 
\begin{equation}\label{eq:510}
    \SG(\calO)\simeq F_+^\times\backslash \whF^\times/\Nr(\calN(\wcO)), 
\end{equation}
where we have rewritten $F_D^\times$ as $F_+^\times$ since $D$ is totally definite. 
Consider the following  quotient group $\WSG(\calO)$ of $\SG(\calO)$ obtained by replacing $F_+^\times$ with $F^\times$ on the right hand side of \eqref{eq:510}
\begin{equation}\label{eqn:df-WSG}
    \WSG(\calO)\coloneqq F^\times\backslash \whF^\times /\Nr(\calN(\wcO)).
\end{equation}
We shall call it the \emph{wide spinor genus group} of $\calO$ since it can be realized as the \emph{pushout} of the two canonical homomorphisms $\Xi:\SCl(\calO)\to \SG(\calO)$ and $\eta:\SCl(\calO)\cong\Cl^+(O_F)\to \Cl(O_F)$, where we have identified $\SCl(\calO)$ with the narrow class group $\Cl^+(O_F)$ via the isomorphism in \eqref{eq:SCl-narrow-Cl}; see diagram~\eqref{eq:cf-diagram1}.  For example, 
if $D$ is unramified at all the finite places of $F$ and $\calO$ is maximal in $D$, then $\WSG(\calO)\simeq \Cl(O_F)/\Cl(O_F)^2$. Our general divisibility theorem of $t(\calO)$ can be stated as follows.

\begin{thm}\label{thm:div-typ-5.3}
    Let $\calO$ be a residually unramified $O_F$-order in a totally definite quaternion $F$-algebra $D$. Then its type number $t(\calO)$ is divisible by $\abs{\WSG(\calO)}$. If $D$ is further assumed to be ramified at some finite place of $F$, then $t(\calO)$ is  divisible by $\abs{\SG(\calO)}$.
\end{thm}

We make use of the class field theory to prove this theorem.  In \eqref{eq:8}, we have established an isomorphism
\begin{equation}
    \SG(\calO)\to \Gal(\Sigma/F), \qquad [\calO']_\sg\mapsto \rho(\calO, \calO'),
\end{equation}
where $\Sigma$ is the spinor genus field in Definition~\ref{defn:spinor-field}. Let $\Sigma^\wide$ be the fixed subfield of $\ker\left(q:\SG(\calO)\to \WSG(\calO)\right)$ so that $\WSG(\calO)\simeq \Gal(\Sigma^\wide/F)$.  Since $\SG(\calO)$ is realized as a quotient group of $\Cl^+(O_F)$, the field $\Sigma$ is a subfield of the \emph{narrow  Hilbert class field} $\calH^+/F$, whose Galois group is canonically identified with $\Cl^+(O_F)$ via the Artin reciprocity map.  By construction, $\Sigma^\wide$ is a subfield of the \emph{(wide) Hilbert class field} $\calH/F$, whose Galois group is canonically identified with $\Cl(O_F)$. These groups fit
into the following  commutative diagram 
\begin{equation}\label{eq:cf-diagram1}
    \begin{tikzcd}
        \Gal(\calH^+/F)\ar[d,twoheadrightarrow] & \Cl^+(O_F)\ar[l, "\cong"']\ar[d,twoheadrightarrow] & \SCl(O_F)\ar[l, "\eqref{eq:SCl-narrow-Cl}", "\cong"']\ar[r, "\Xi", twoheadrightarrow]\ar[d, twoheadrightarrow, "\eta"]\ar[dr, phantom, "\textit{pushout}"] & \SG(\calO)\ar[d, "q",twoheadrightarrow] \ar[r, "\cong", "\eqref{eq:8}"']& \Gal(\Sigma/F)\ar[d, twoheadrightarrow]\\
        \Gal(\calH/F) & \Cl(O_F)\ar[l, "\cong"'] \ar[r, equal]& \Cl(O_F) \ar[r, twoheadrightarrow]& \WSG(\calO) \ar[r, "\cong"]& \Gal(\Sigma^\wide/F)
    \end{tikzcd}
\end{equation}

\begin{proof}[Proof of Theorem~\ref{thm:div-typ-5.3}]  
To prove that $\abs{\WSG(\calO)}$ divides the type number $t(\calO)$, it is enough to show that the fibers of the composition $\Tp(\calO)\xrightarrow{\Psi_\tp} \SG(\calO)\xrightarrow{q} \WSG(\calO)$ share the same cardinality.  Each element of $\WSG(\calO)$ is of the form $[\calO']_\wsg\coloneqq q([\calO']_\sg)$ for some  spinor genus $[\calO']_\sg\in \SG(\calO)$.  Let $m$ be the order of the kernel of $q: \SG(\calO)\to \WSG(\calO)$, and $q^{-1}([\calO']_\wsg)\coloneqq \left\{[\calO_1]_\sg,\cdots,[\calO_m]_\sg\right\}$ be the pre-image of $[\calO']_\wsg$ under the map $q$. 
Then
\begin{equation}\label{eq:card-of-fiber-over-WSG}
    \abs{(q\circ\Psi_\tp)^{-1}([\calO']_\wsg)}=\sum_{i=1}^m t_\sg(\calO_i).
\end{equation}
The fiber $q^{-1}([\calO']_\wsg)$ can also be characterized in terms of the class fields $\Sigma/F$ and $\Sigma^\wide/F$. In light of the commutative diagram \eqref{eq:cf-diagram1}, given $[\calO'']_\sg\in \SG(\calO)$, we have 
\begin{equation}\label{eq:fiber-cf-char}
    [\calO'']_\sg\in q^{-1}([\calO']_\wsg)\quad \Longleftrightarrow \quad \rho(\calO, \calO'')|_{\Sigma^\wide}=\rho(\calO, \calO')|_{\Sigma^\wide}, 
\end{equation}
where $\rho(\calO, \calO'')|_{\Sigma^\wide}$ denotes the restriction of  $\rho(\calO, \calO'')\in \Gal(\Sigma/F)$ to $\Sigma^\wide$.

From the spinor type number formula \eqref{eq:spinor-typ-final}, it is clear that the optimal spinor selectivity in the restricted sense lies at the root cause of the variance of the cardinality of the fibers of $\Psi_\tp: \Tp(\calO)\to \SG(\calO)$.  We then apply Theorem~\ref{thm:spinor-selectivity} to partition the set $\scrB$  of pointed CM $O_F$-orders in the spinor type number formula into two subsets according to whether the genus $\scrG(\calO)$ is selective for $(B, \lambda)\in \scrB$ or not:
\begin{equation*}
    \scrB_\sel\coloneqq\{(B,\lambda)\in\scrB\mid s(B,\calO)=1\text{ and }\prod_\grp n_\grp(B,\lambda)\neq0\}, \quad \scrB_\non\coloneqq\scrB\smallsetminus\scrB_\sel.
\end{equation*}
Recall from \eqref{eq:defn-sBO} that the symbol $s(B, \calO)$ depends only on the fraction field $\Frac(B)$ and the genus $\scrG(\calO)$. 
For each $(B,\lambda)\in\scrB_\non$, we have
\begin{equation}\label{eq:non-sel-part}
   2^{s(B, \calO_i)} \Delta(B,\calO_i)\prod_\grp n_\grp(B,\lambda)=\prod_\grp n_\grp(B,\lambda), \qquad \forall 1\leq i \leq m. 
\end{equation}
Combining \eqref{eq:non-sel-part}, \eqref{eq:card-of-fiber-over-WSG} with the spinor type number formula \eqref{eq:spinor-typ-final}, we obtain 
\begin{equation}\label{eq:partition-spinor-type-number-formula}
    \begin{split}
   \abs{(q\circ\Psi_\tp)^{-1}([\calO']_\wsg)}&=\frac{1}{C(\calO)} \Bigg( 
        m\Mass(\calO)+\frac{m}{2} \sum_{(B,\lambda)\in\scrB_\non} \frac{h(B)}{w(B)}\prod_\grp n_\grp(B,\lambda)\\
        &+\sum_{(B,\lambda)\in\scrB_\sel} \Big(\sum_{i=1}^m\Delta(B,\calO_i)\Big)\frac{h(B)}{w(B)}\prod_\grp n_\grp(B,\lambda) \Bigg),
    \end{split}
\end{equation}
where for simplicity we have put $C(\calO)\coloneqq 2^{\omega(\calO)}h(F)\abs{\SG(\calO)}$ as usual. 

Thus to prove that $\abs{\WSG(\calO)}$ divides $t(\calO)$, it is enough to prove that the value of the summation $\sum_{i=1}^m\Delta(B,\calO_i)$ (taken inside $\Z$) is independent of the choice of $[\calO']_\wsg\in \WSG(\calO)$ for every $(B, \lambda)\in \scrB_\sel$. In fact, we claim that more precisely  
\begin{equation}\label{eq:claim-sum-Delta}
    \sum_{i=1}^{m} \Delta(B,\calO_i))=\frac{m}{2}=\frac{1}{2}\abs{\ker(\SG(\calO)\xrightarrow{q} \WSG(\calO))}, \qquad \forall (B, \lambda)\in \scrB_\sel.
\end{equation}
Given $(B, \lambda)\in \scrB_\sel$, we put  $K\coloneqq \Frac(B)$ and identify $\Gal(K/F)\cong\zmod{2}$ as in Theorem~\ref{thm:spinor-selectivity}. Since $(B,\lambda)\in\scrB_\sel$, we have $K\subseteq\Sigma$. Now it follows from Proposition~\ref{prop:ru-equal-sel-symb} and \eqref{eq:164} that
    \begin{equation}
        \begin{split}
            &\sum_{i=1}^m \Delta(B,\calO_i))
            =\#\{1\le i\le m\mid\rho(\calO, \calO_i)|_K=1-\Delta(B,\calO)\}\\
            \xeq{\eqref{eq:fiber-cf-char}}&\#\{\sigma\in\Gal(\Sigma/F)\,\mid\,\sigma|_{\Sigma^\wide}=\rho(\calO, \calO')|_{\Sigma^\wide}\quad \text{and}\quad \sigma|_K=1-\Delta(B,\calO)\}.
        \end{split}
    \end{equation}
 Both the restriction maps $\Gal(\Sigma/F)\twoheadrightarrow\Gal(\Sigma^\wide/F)$ and $\Gal(\Sigma/F)\twoheadrightarrow\Gal(K/F)$ factor through $\Gal(K\Sigma^\wide/F)$, where $K\Sigma^\wide$ denotes the compositum of $K$ and $\Sigma^\wide$ inside $\Sigma$; see Figure~\ref{fig:cf-diagram}.  Observe that the field $\Sigma^\wide$ is a totally real field since it is  a subfield of the (totally real) Hilbert class field $\calH/F$. On the other hand, $K/F$ is a CM-extension, and hence is linearly disjoint from $\Sigma^\wide/F$. Thus there is a canonical isomorphism 
    \begin{equation}
        \Gal(K\Sigma^\wide/F)\simeq\Gal(\Sigma^\wide/F)\times\Gal(K/F).
    \end{equation}
This implies that 
    \begin{equation}
      \begin{split}
         \sum_{i=1}^m \Delta(B,\calO_i))=&\abs{\ker(\Gal(\Sigma/F)\twoheadrightarrow\Gal(K\Sigma^\wide/F))}\\=&\frac{1}{2}\abs{\ker(\Gal(\Sigma/F)\twoheadrightarrow\Gal(\Sigma^\wide/F))}=\frac{m}{2}. 
      \end{split}  
    \end{equation}
This finishes the verification of Claim \eqref{eq:claim-sum-Delta} and completes the proof of Theorem~\ref{thm:div-typ-5.3}.
\end{proof}

For the reader's convenience, we draw the diagram of the relevant class fields in Figure~\ref{fig:cf-diagram}, where the fields marked with dashed lines will be more notably used in the proof of Theorem~\ref{thm:div-tr-Brandt}.
\begin{figure}[h]
\begin{center}
    \begin{tikzpicture}
    \node (F) at (0,0) {$F$};
    \node (K) at (-1.7,0.8) {$K$};
    \node (Sigma^wd) at (1.7,0.8) {$\Sigma^\wide$};
    \node (KSigma^wd) at (0,1.6) {$K\Sigma^\wide$};
    \node (Sigma) at (-1.7,2.4) {$\Sigma$};
    \node (H) at (1.7,2.4) {$\calH$};
    \node (KH) at (0,3.2) {$K\calH$};
    \node (H^+) at (0,4.5) {$\calH^+$};

    \draw (F) -- (K);
    \draw (F) -- (Sigma^wd);
    \draw (K) -- (KSigma^wd);
    \draw (Sigma^wd) -- (KSigma^wd);
    \draw (K) -- (Sigma);
    \draw (Sigma^wd) -- (H);
    \draw (KSigma^wd) -- (Sigma);
    \draw (Sigma) -- (H^+);
    \draw (H) -- (H^+);
    \draw[dashed] (K) -- (KH);
    \draw[dashed] (H) -- (KH);
    \draw[dashed] (KH) -- (H^+);
    \draw[dashed] (H) -- (F);
\end{tikzpicture}
\end{center}
\caption{Class fields diagram}
\label{fig:cf-diagram}
\end{figure}
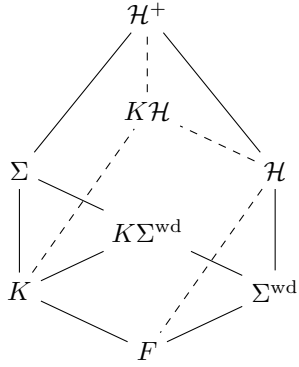

By definition, the type number $t(D)$ of a quaternion $F$-algebra $D$ is the type number of the genus of maximal $O_F$-orders in $D$.

\begin{cor}\label{cor:2powerdiv}
    For any positive integer $n\ge1$, there exists a totally definite quaternion  $D$ over some real quadratic field $F$ such that $2^n \ddiv t(D)$.
\end{cor}
\begin{proof}
    Take $p_1,\dots,p_{n+1}$ to be $n+1$ distinct primes such that $p_i\equiv1\pmod{4}$ for all $i$ and put $F\coloneqq \bbQ(\sqrt{p_1\cdots p_{n+1}})$.  Then by \cite[Theorem 4, Chapter VII, p.12]{seminarCM-1966} we have
    \begin{equation}
        \Cl(O_F)/\Cl(O_F)^2 \simeq (\zmod{2})^n.
    \end{equation}
    Let $D$ be the totally definite quaternion $F$-algebra that is unramified at all the finite places of $F$ and $\calO$ be any maximal order in $D$. Note that $\WSG(\calO)\simeq\Cl(O_F)/\Cl(O_F)^2$ in this case as already mentioned. Then the corollary follows directly from Theorem~\ref{thm:div-typ-5.3}.
\end{proof}

We return to the more general setting and let $D$ be  an arbitrary   totally definite quaternion algebra over a totally real field $F$.  Let $\grn$ be a nonzero integral ideal of $O_F$. 
The same method as in the proof of Theorem~\ref{thm:div-typ-5.3} can also 
be applied to obtain the divisibility of the trace of the Brandt matrix $\grB(\calO, \grn)$ by $h(F)$
as stated in Theorem~\ref{thm:class-num-div}. We recall the definition of the Brandt matrix \cite[Definition~41.2.4]{voight-quat-book}\cite[Definition~3.1.3]{xue-yang-yu:ECNF} below, which applies to any arbitrary $O_F$-order $\calO$ in $D$.
\begin{defn}\label{defn:brandt-matrix}
     let $\{I_i\mid 1\le i\le h\coloneqq h(\calO)\}$ be a complete set of representatives of $\Cl(\calO)$.  The Brandt matrix $\grB(\calO, \grn)$ is  an $h\times h$ integral matrix $(\grB_{ij}(\calO,\grn))\in \Mat_h(\Z)$ whose $(i, j)$-th entry $\grB_{ij}(\calO, \grn)$ is defined as follows
\begin{equation*}
    \grB_{ij}(\calO, \grn)\coloneqq\#\{ \text{invertible right $\calO$-ideal } I'\subseteq I_i\mid  \Nr(I')=\grn\Nr(I_i) \text{ and } [I']=[I_j]     \}.
\end{equation*}
\end{defn}
From \cite[\S41.2.5]{voight-quat-book}, the Brandt matrix $\grB(\calO, \grn)$ can be interpreted as the matrix of the $\grn$-Hecke operator on a space of algebraic modular forms attached to $\calO$.  Eichler \cite{eichler:crelle55} first gave a formula for the trace of $\grB(\calO, \grn)$ that bears his name nowadays.   K\"orner \cite{korner:1987} extends the \emph{Eichler trace formula} from Eichler 
orders to arbitrary $O_F$-orders in \cite[Theorem~2]{korner:1987}. This formula is further generalized by the second named author and his collaborators in \cite[Theorems~3.3.3 and 3.3.7]{xue-yang-yu:ECNF} so that it applies to arbitrary $\Z$-orders.   From the definition, it is clear that both of the following hold
\begin{enumerate}
    \item if $\grn=O_F$, then $\Tr(\grB(\calO, O_F))=h(\calO)$ since $ \grB_{ii}(\calO, O_F)=1$ for every $1\leq i\leq h$ in this case;
    \item if $\grn$ is not principally generated by a totally positive element, then $ \grB_{ii}(\calO, \grn)=0$ for every $1\leq i\leq h$, and hence $\Tr(\grB(\calO, \grn))=0$  in this case.
\end{enumerate}
For the reader's convenience, we restate the divisibility result on   the trace of the Brandt matrix below.

\begin{thm}\label{thm:div-tr-Brandt}
    Let $\calO$ be a residually unramified $O_F$-order in a totally definite quaternion $F$-algebra $D$. Then for every nonzero integral ideal $\grn\subseteq O_F$, $\Tr(\grB(\calO,\grn))$ is divisible by the class number $h(F)$. If $D$ is further assumed to be ramified at some finite place of $F$, then $\Tr(\grB(\calO,\grn)))$ is divisible by the narrow class number $h^+(F)$.  In particular, the above divisibility results hold for the class number $h(\calO)$ as $h(\calO)=\Tr(\grB(\calO, O_F))$.
\end{thm}

To prove this theorem, we apply the \emph{spinor Eichler trace formula} developed by Chia-Fu Yu and the second named author in \cite[(3.18)]{xue-yu:spinor-class-no} that computes the trace of certain diagonal block of $\grB(\calO, \grn)$. 
More precisely, recall from \eqref{eq:decomp-Cl} that $\Cl(\calO)$ admits a decomposition
\[\Cl(\calO)=\bigsqcup_{[J]_\scc\in\SCl(\calO)} \Cl(\calO,[J]_\scc),\]
where $\Cl(\calO, [\calO]_\scc)$ is abbreviated as $\Cl_\scc(\calO)$ with its cardinality denoted by $h_\scc(\calO)$. From the bijection 
$\Cl(\calO, [J]_\scc)\simeq \Cl_\scc(\calO_l(J))$ 
in \eqref{eq:spinor-cl-set-bij}, we have
\begin{equation}
    h=h(\calO)=\sum_{[J]_\scc\in \SCl(\calO)} h_\scc(\calO_l(J)). 
\end{equation}
We rearrange the representatives $I_1, \cdots, I_h$ of $\Cl(\calO)$ so that for each $[J]_\scc\in \SCl(\calO)$, the set of indices $\calS([J]_\scc)\coloneqq \{1\leq i \leq h\mid [I_i]\in \Cl(\calO,  [J]_\scc)\}$ forms a consecutive sequence of integers, and we further require that $\calS([\calO]_\scc)$ coincides with the first segment $\{1, \cdots, h_\scc(\calO)\}$. Let $\grB_\scc(\calO, [J]_\scc, \grn)$ be the $h_\scc(\calO_l(J))\times h_\scc(\calO_l(J))$-submatrix of $\grB(\calO, \grn)$ formed by selecting the rows and columns with indices  in $\calS([J]_\scc)$. In particular, each $\grB_\scc(\calO, [J]_\scc, \grn)$ is a diagonal block of $\grB(\calO, \grn)$. For example, 
\begin{equation}
   \grB_\scc(\calO, \grn)\coloneqq  \grB_\scc(\calO,[\calO]_\scc, \grn)=(\grB_{ij}(\calO,\grn))_{1\le i,j\le h_\scc(\calO)}
\end{equation}
is just the diagonal $h_\scc(\calO)\times h_\scc(\calO)$-block at the upper left corner of $\grB(\calO, \grn)$, and $\Tr(\grB_\scc(\calO, O_F))=h_\scc(\calO)$ when $\grn=O_F$.  For each $[J]_\scc\in \SCl(\calO)$,  
the map $I_i\mapsto I_iJ^{-1}$ induces an identification
\begin{equation}
    \grB_\scc(\calO, [J]_\scc, \grn)=\grB_\scc(\calO_l(J), \grn),
\end{equation}
from which it follows that 
\begin{equation}\label{eq:sum-of-spinor-Brandt-trace}
    \Tr(\grB(\calO,\grn))=\sum_{[J]_\scc\in\SCl(\calO)} \Tr(\grB_\scc(\calO_l(J),\grn)).
\end{equation}
The  spinor Eichler trace formula computes $\Tr( \grB_\scc(\calO, \grn))$ for any residually unramified $O_F$-order $\calO$. Thus it  can be regarded as a refinement of the classical Eichler trace formula, in the same way that the spinor type number formula  refines the classical type number formula (cf.~\eqref{eq:typ-refine}). 

To write down the spinor Eichler trace formula, without loss of generality,  we assume that $\grn$ is generated by a totally positive element $\textsf{n}$. 
Fix a complete set $\scrU$ of representatives of $O_{F, +}^\times/O_F^{\times2}$ as in \S\ref{sect:rep-orbits-XPO}. 
For each CM $O_F$-order $B$ with  fraction field $K$, we define a finite set
\begin{equation}\label{eq:def-TBn}
    T_{B,\grn}\coloneqq\{\alpha\in B\smallsetminus O_F\mid N_{K/F}(\alpha)=\varepsilon \textsf{n} \text{ for some } \varepsilon\in\scrU\}.
\end{equation}
By the same argument as in the proof of Proposition~\ref{prop:def-B}(2), there are only finitely many CM $O_F$-orders $B$ with $T_{B,\grn}\neq\emptyset$, so we collect them together in a finite set 
 set $\scrC_\grn$.  
\begin{prop}[{The spinor Eichler trace formula, \cite[(3.18)]{xue-yu:spinor-class-no}}]\label{prop:spinor-Etrf}
    Let $\calO$ be a residually unramified $O_F$-order in  $D$, and $\grn\subset O_F$ be an  $O_F$-ideal generated by a totally positive element $\textsf{n}$. Then  \begin{equation}\label{eq:spinor-Brandt-trace}
  \begin{split}
      \Tr(\grB_\scc(\calO,&\grn))=\frac{\delta_\grn\Mass(\calO)}{h^+(F)}\\ &+\frac{1}{4h^+(F)}\sum_{B\in\scrC_\grn} \frac{2^{s(B,\calO)}\Delta(B,\calO)h(B)}{w(B)}\abs{T_{B,\grn}}\prod_\grp m(B_\grp, \calO_\grp, \calO_\grp^\times),
  \end{split}  
\end{equation}
where \begin{enumerate}
    \item $\delta_\grn=1$ or $0$ according to whether $\grn$ is the square of a principal ideal or not;
    \item $m(B_\grp, \calO_\grp, \calO_\grp^\times)$ denotes the number of $\calO_\grp^\times$-conjugacy classes of optimal embeddings from $B_\grp$ into $\calO_\grp$ as in \eqref{eq:no-conjcls-opt}.
\end{enumerate}
 If $D$ is further assumed to be ramified at some finite place of $F$, then 
 \begin{equation}\label{eq:529}
   \begin{split}
      \Tr(\grB_\scc(\calO,\grn))=&\frac{\delta_\grn\Mass(\calO)}{h^+(F)}\\ &+\frac{1}{4h^+(F)}\sum_{B\in\scrC_\grn} \frac{h(B)}{w(B)}\abs{T_{B,\grn}}\prod_\grp m(B_\grp, \calO_\grp, \calO_\grp^\times),
  \end{split}  
\end{equation}
from which it follows that $\Tr(\grB(\calO,\grn))=h^+(F)\Tr(\grB_\scc(\calO,\grn))$ in this case. 
\end{prop}

As usual, formula \eqref{eq:529} is just a special form of \eqref{eq:spinor-Brandt-trace}, since the genus $\scrG(\calO)$ is  spinor genial for any CM $O_F$-order $B$ if $D$ is ramified at some finite place of $F$ (See \cite[Theorem~2.15 and Lemma~2.17]{Xue-Yu-Selec-2024}). Clearly, the right hand side of \eqref{eq:529} depends only on the genus $\scrG(\calO)$, so it follows from \eqref{eq:sum-of-spinor-Brandt-trace} that $\Tr(\grB(\calO,\grn))=h^+(F)\Tr(\grB_\scc(\calO,\grn))$ in this case.
If $\grn=O_F$, then we recover the equality $h(\calO)=h^+(F)h_\scc(\calO)$ in \cite[Theorem~3.2]{xue-yu:spinor-class-no}, which can be regarded as a further specialization of our current result.

\begin{proof}[Sketch of proof of Theorem~\ref{thm:div-tr-Brandt}]
   The proof follows exactly the same line of arguments as in the proof of Theorem~\ref{thm:div-typ-5.3}, so we merely provide a sketch. If $D$ is ramified at some finite place of $F$, then it has already been shown in 
   Proposition~\ref{prop:spinor-Etrf} that $\Tr(\grB(\calO,\grn))$ is divisible by the narrow class number $h^+(F)$. 
    Suppose further that $D$ is unramified at all the finite places of $F$. Consider the middle vertical map in the commutative diagram \eqref{eq:cf-diagram1} given as follows
   \begin{equation}
      \eta: \SCl(\calO)\to \Cl(O_F), \qquad [J]_\scc\mapsto [\Nr(J)].
   \end{equation}
 In order to  prove that $\Tr(\grB(\calO,\grn))$ is divisible by $h(F)$, in light of \eqref{eq:sum-of-spinor-Brandt-trace},  it is enough to show that for each ideal class $[\gra]\in \Cl(O_F)$, the value of the summation
  \[\grT([\gra])\coloneqq \sum_{[J]_\scc\in \eta^{-1}([\gra])}\Tr(\grB_\scc(\calO_l(J),\grn)) \]
   is independent of the choice of $[\gra]\in \Cl(O_F)$.  

From the spinor Eichler trace formula \eqref{eq:spinor-Brandt-trace}, it is clear that the optimal spinor selectivity once again lies at the root of the variance of the value of $\Tr(\grB_\scc(\calO_l(J),\grn))$ as $[J]_\scc$ ranges in $\SCl(\calO)$. We put \[k\coloneqq \abs{\ker(\eta: \SCl(\calO)\to \Cl(O_F))}=\abs{\ker(\Cl^+(\calO)\to \Cl(O_F))},\] and write $\eta^{-1}([\gra])=\{[J_1]_\scc, \cdots, [J_k]_\scc\}$. Let $\scrC_{\grn, \sel}$ be the subset of $\scrC_\grn$ consisting of the CM $O_F$-orders $B$ such that the genus $\scrG(\calO)$ is optimally spinor selective for $B$. To prove the invariants of $\grT([\gra])$, it is enough to prove that 
\begin{equation}\label{eq:531}
    \sum_{j=1}^k \Delta(B, \calO_l(J_j))=\frac{k}{2}=\frac{1}{2}\abs{\ker(\Cl^+(\calO)\to \Cl(O_F))}, \qquad \forall B\in \scrC_{\grn, \sel}.
\end{equation}
Observe the similarity between \eqref{eq:531} and \eqref{eq:claim-sum-Delta}. Let $(\gra, \calH/F)\in \Gal(\calH/F)$ be the Artin symbol of $\gra$
in $\Gal(\calH/F)$, and $K$ be the fraction field of $B$.  
Combining the formula for the selectivity symbol \cite[(2.21)]{Xue-Yu-Selec-2024} with the commutative diagram \eqref{eq:cf-diagram1}, we find that 
\[\begin{split}
    \sum_{j=1}^k \Delta&(B, \calO_l(J_j))=\#\big\{\sigma\in\Gal(\calH^+/F)\mid\sigma|_\calH=(\gra, \calH/F)\text{ and } \sigma|_K=1-\Delta(B,\calO)\big\}\\
    &\xeq{(\star)}\abs{\ker(\Gal(\calH^+/F)\to \Gal(K\calH/F))}=\frac{1}{2}\abs{\ker(\Gal(\calH^+/F)\to \Gal(\calH/F))},
\end{split}\]
where the equality $(\star)$ once again follows from the linearly disjointness of the CM-extension $K/F$ from the totally real Hilbert class field $\calH/F$. The relevant class fields involved in the above proof have been marked with dashed lines in Figure~\ref{fig:cf-diagram}. This finishes the verification of \eqref{eq:531} and completes the proof of Theorem~\ref{thm:div-tr-Brandt}.
\end{proof}

\section{\texorpdfstring{Application to ternary quadratic $O_F$-lattices}{Application to ternary quadratic OF-lattices}}\label{sec:ternary-latt}

In this section we re-interpret  the spinor type number as the class number of certain ternary quadratic lattices within a fixed spinor genus. 
Let $F$ be a totally real number field, $V$ be an $F$-vector space with $\dim_F V=3$, and $Q:V\to F$ be a totally positive definite quadratic form. The \emph{even Clifford algebra} \cite[\S5.3]{voight-quat-book} $D\coloneqq \Clf^0(V, Q)$ is a totally definite quaternion $F$-algebra. If we write $D^0\coloneqq \{x\in D\mid \Tr(x)=0\}$ for the trace zero subspace of $D$, then $Q$ can be recovered from the restriction of the reduced norm map $\Nr\mid_{D^0}$ up to a totally positive constant by \cite[Corollary~5.2.6]{voight-quat-book}. As explained in \cite[Chapter~22]{voight-quat-book}, the even Clifford construction extends naturally to $O_F$-lattices in $V$ and establishes a correspondence between  $O_F$-lattices in $V$ and quaternion $O_F$-orders in $D$. We briefly review  the highlights of the theory;  
see \cite[Remark~22.5.13]{voight-quat-book} for the rich history of development in this area.

By definition, the \emph{orthogonal group} $\Ogp_V$ is the linear algebraic $F$-group consisting of all self-isometries of $(V,Q)$. Two $O_F$-lattices $L,L'$ (of full rank) in $V$ are \emph{isometric} if there exists an element $g\in \Ogp_V(F)$ such that $L'=gL$. Since $\dim_FV=3$ is odd, the group  $\Ogp_V$ factors as a direct product  $\SO_V\times \{\pm 1\}$, where $\SO_V$ denotes the \emph{special orthogonal group of $(V, Q)$}. 
Thus two $O_F$-lattices $L, L'$ in $V$ are isometric if and only if they are \emph{properly isometric} \cite[Example~82:4]{o-meara-quad-forms}, that is, if there exists $g\in \SO_V(F)$ such that $L'=gL$.    From \cite[Proposition~4.5.10]{voight-quat-book}, there is a short exact sequence 
\begin{equation}\label{eq:Gspin}
    1\to F^\times \to D^\times\to \SO_V(F)\to 1,
\end{equation}
which realizes $\SO_V$ as the adjoint group $\ul{D}^{\ad}$ of the multiplicative algebraic $F$-group $\ul{D}^\times$.  Accordingly, the \emph{spinor norm map} \cite[\S55]{o-meara-quad-forms}\cite[\S9.3]{Scharlau-Quad-Herm} \[\theta: \SO_V(F)\to F^\times/F^{\times2}\] is induced by the reduced norm map $\Nr: D^\times\to F^\times$. The local spinor norm maps $\theta_\grp:  \SO_V(F_\grp)\to F_\grp^\times/F_\grp^{\times2}$ piece together to form an adelic group homomorphism $\hat{\theta}: \SO_V(\whF)\to\whF^\times/\whF^{\times2}$, and we define  $\Theta_V(\whF)\coloneqq\ker(\hat\theta)$.  As usual, the profinite completion of an $O_F$-lattice $L$ in $V$ is denoted by $\whL$. With an eye specifically on the dimension 3 case,  we introduce the concepts of \emph{genus} and \emph{spinor genus} of $O_F$-lattices in $V$ as follows. 
\begin{defn}
    (1) Two $O_F$-lattices $L,L'$ in $V$ are said to \emph{belong to the same  genus} if there exists $\hat{g}\in \SO(\whF)$ such that $\whL'=\hat{g}\whL$, or equivalently,  if $L_\grp$ and $L'_\grp$ are   isometric at every finite place $\grp$ of $F$. The  genus of $L$ is denoted by $\scrG(L)$ and we write $\Cl(L)\coloneqq\Cl(\scrG(L))$ for the set of isometric classes within $\scrG(L)$.

    (2) Two $O_F$-lattices $L,L'$ in $V$ are said to \emph{belong to the same (proper) spinor genus} if there exists $\hat{g}\in \SO_V(F)\Theta_V(\whF)$ such that $\whL'=\hat{g}\whL$. We write $\Cl_\sg(L)$ for the set of isometric classes within the spinor genus of $L$.
\end{defn}

A priori, the notion of \emph{spinor genus} in \cite[\S102]{o-meara-quad-forms} is defined in terms of the orthogonal group $O_V$. However, as we have already seen,  it does not matter if $\Ogp_V$ is replaced by  $\SO_V$ in the dimension $3$ setting, and there is no distinction between the notion \emph{spinor genus} and \emph{proper spinor genus} in this case.  Meanwhile, there is no distinction between the notion of \emph{genus} and \emph{proper genus} for lattices in general quadratic spaces by \cite[\S91.4a]{o-meara-quad-forms}. 

The even Clifford construction \cite[\S22.3]{voight-quat-book} attaches to each $O_F$-lattice $L$ in $V$ a quaternion $O_F$-order $\calO_L\coloneqq \Clf^0(L)$, where we take the codomain of the quadratic module to be the fractional $O_F$-ideal generated by $Q(L)$. Moreover, the construction is functorial by \cite[Theorem~22.3.1]{voight-quat-book}, so in particular, the stabilizer of $\whL$ in $\SO_V(\whF)$ can be described in terms of the normalizer $\calN(\wcO_L)$ as follows by \cite[(5.12)]{hein2025computinghilbertmodularforms}:
\begin{equation}
   \Stab(\whL)\simeq \calN(\wcO_L)/\whF^\times.
\end{equation}
  From \cite[Proposition~2.1]{Brzezinski-Spinor-Class-gp-1983} (or applying the functoriality again), two $O_F$-lattices $L$ and $L'$ with equal norms in $(V, Q)$ belong to the same genus (resp.~spinor genus) if and only if their corresponding $O_F$-order $\calO_{L}$ and $\calO_{L'}$ in $D$ belong to the same genus (resp.~spinor genus).  In particular, for a fixed $O_F$-lattice $L$ in $(V, Q)$, the order of the spinor genus group $\SG(\calO_L)$ coincides with the number  $g(L)$ of spinor genera within the genus $\scrG(L)$ as studied in \cite[\S102B]{o-meara-quad-forms}.
  Combining the above discussion with \cite[Lemma~5.6 and Corollary~5.8]{hein2025computinghilbertmodularforms}, we obtain the following proposition.

\begin{prop}\label{prop:bij-Cl(L)-Tp(O_L)}
 Fix an $O_F$-lattice $L$ in $(V, Q)$ and put $\calO_L\coloneqq \Clf^0(L)$. The even Clifford map induces a spinor-genus preserving bijection between the isometric class set $\Cl(L)$ and the type set $\Tp(\calO_L)$ of the quaternion $O_F$-order $\calO_L$ as follows
 \begin{equation}\label{eq:663}
     \Cl(L)\simeq \Tp(\calO_L), \qquad [L']\mapsto [\calO_{L'}]. 
 \end{equation}
In particular, the above bijection  restricts to a finer bijection 
 \begin{equation}\label{eq:664}
     \Cl_\sg(L)\simeq \Tp_\sg(\calO_L)
 \end{equation}
 between the spinor isometric class set $ \Cl_\sg(L)$ and the spinor type set $\Tp_\sg(\calO_L)$. 
\end{prop}

In light of the identification of $\SO_V$ with the adjoint group $\ul{D}^{\ad}$,   the bijection in \eqref{eq:663} can be interpreted adelically as a composition of bijections between the following double coset spaces
\begin{equation}
\begin{split}
    \Cl(L)\simeq   \SO_V(F)\backslash \SO_V(\whF)/\Stab(\whL)&\simeq \ul{D}^{\ad}(F)\backslash \ul{D}^{\ad}(\whF)/\Big(\calN(\wcO_L)/\whF^\times\Big)\\
    &\simeq D^\times\backslash \whD^\times/\calN(\wcO_L)\simeq \Tp(\calO_L). 
\end{split}
\end{equation}

For each finite place $\grp$ of $F$, we write $(V_\grp,Q_\grp)$ for the corresponding quadratic space over $F_\grp$. By \cite[Theorem~5.4.4]{voight-quat-book}, $(V_\grp,Q_\grp)$ is an anisotropic quadratic $F_\grp$-space if and only if $D_\grp$ is a division $F_\grp$-algebra.
The bijection \eqref{eq:664} allows us to apply the main results of this paper to totally positive ternary quadratic $O_F$-lattices. 

\begin{thm}\label{thm:ternary-latt}
Let $L$ and $\calO_L$ be as in Proposition~\ref{prop:bij-Cl(L)-Tp(O_L)} and assume further that $\calO_L$ is residually unramified. 
\begin{enumerate}
    \item The spinor class number $h_\sg(L)\coloneqq \abs{\Cl_\sg(L)}$ of the ternary quadratic lattice $L$ coincides with the spinor type number $t_\sg(\calO_L)$ of $\calO_L$, which  can be calculated by the formulas in Theorem~\ref{thm:spinor-type-number-formula}.
    \item The class number $h(L)\coloneqq \abs{\Cl(L)}$ is always divisible by the order of the wide spinor genus group $\WSG(\calO_L)$. If further 
    $(V_\grp,Q_\grp)$ is anisotropic over $F_\grp$ for some finite place $\grp$ of $F$, then $h(L)$ is divisible by  the number $g(L)$ of spinor genera within $\scrG(L)$.
\end{enumerate}
\end{thm}

To apply Theorem~\ref{thm:ternary-latt} to a given $O_F$-lattice $L$ in $(V, Q)$, it is necessary to ensure that  $\calO_L$ is residually unramified.  The Eichler invariant of $\calO_L$ at each finite place $\grp$ of $F$ can be calculated in terms of the local lattice $L_\grp$ by the method explained in \cite[\S24.3.9]{voight-quat-book}.  Write $Q_\grp(L_\grp)=a_\grp O_{F_\grp}$ for some $a_\grp\in F_\grp^\times$. According to \cite[Theorem~22.3.1]{voight-quat-book},  similar lattices give rise to isomorphic quaternion orders. Without loss of generality, we replace $Q_\grp$ by $\frac{1}{a_\grp}Q_\grp$ so that $Q_\grp(L_\grp)=O_{F_\grp}$.  The $O_{F_\grp}$-lattices corresponding to residually unramified orders can be characterized as follows. 

\begin{prop}\label{prop:res-unr-equiv}
Let $L_\grp$ be an $O_{F_\grp}$-lattice in $(V_\grp, Q_\grp)$ with $Q_\grp$ suitably normalized so that $Q_\grp(L_\grp)=O_{F_\grp}$ as above. The following are equivalent:
\begin{enumerate}[label=(\arabic*), leftmargin=*]
    \item $\calO_{L_\grp}$ is residually unramified;
    \item the invariant $\Omega(L_\grp)$ as defined in \cite[p.~477]{Brzezinski-Spinor-Class-gp-1983} coincides with the unit ideal $O_{F_\grp}$;
    \item The discriminant module $L_\grp^\#/L_\grp$ is a cyclic $O_{F_\grp}$-module, where $L_\grp^\#$ denotes the dual $O_{F_\grp}$-lattice of $L_\grp$;
    \item  The $O_{F_\grp}$-lattice $L_\grp$ decomposes as the orthogonal direct sum $U_\grp\perp O_{F_\grp}z$ for some $z\in L_\grp$, where $U_\grp$ is a unimodular quadratic $O_{F_\grp}$-lattice of rank $2$.
\end{enumerate}
\end{prop}

\begin{proof}
    The equivalence (1) $\Leftrightarrow$ (2) is due to Brzezinski \cite[Proposition~2.7]{Brzezinski-Spinor-Class-gp-1983}. The equivalences (1) $\Leftrightarrow$ (3) $\Leftrightarrow$ (4) are given in \cite[Lemma~5.14]{hein2025computinghilbertmodularforms}.
\end{proof}

\section*{Acknowledgments}
Xue is partially supported by the National Natural Science Foundation
of China grants No.~12271410 and No.~12331002. He thanks Chia-Fu Yu for helpful discussions.

\bibliographystyle{hplain}
\bibliography{TeXBiB}

\end{document}